\documentclass[a4paper,11pt]{article}
\usepackage{amssymb}
\usepackage[utf8]{inputenc}
\usepackage[T1]{fontenc}
\usepackage{graphicx}
 \usepackage{amssymb,amsmath,latexsym,amsthm}
\usepackage[a4paper]{geometry}
\usepackage{url}
\usepackage{tocloft}  
\usepackage[utf8]{inputenc}   
\usepackage[T1]{fontenc}
\usepackage{graphicx}
\usepackage{mathrsfs}
\usepackage{frcursive}
\usepackage{color}
\usepackage{calligra}
\usepackage{pst-node}
\usepackage{breqn}
\usepackage{tikz}
\usepackage[mathscr]{euscript}
\usepackage{calrsfs}
\usepackage{float}
\usepackage{amssymb}
\usepackage{amssymb}
\usepackage{titlesec}
\usepackage[Rejne]{fncychap}
\usepackage{lipsum}
\usepackage{float}
\usepackage{xfrac} 
\usepackage[Rejne]{fncychap}
\usepackage{faktor}
\usepackage{amsmath}
\usepackage[T1]{fontenc}
\usepackage[utf8]{inputenc}
\usepackage[colorlinks=true,urlcolor=black,linkcolor=blue,citecolor=black]{hyperref}
\usepackage[english]{babel}
\usepackage{amsthm}
\usepackage{bbm}
\usepackage{float}

\newcounter{mythm}[section]

\newtheorem{tm}[mythm]{Theorem}
\newtheorem{lm}[mythm]{Lemma}
\newtheorem{defi}[mythm]{Definition}
\newtheorem{pr}[mythm]{Proposition}
\newtheorem{rem}[mythm]{Remark}
\newtheorem{cor}[mythm]{Corollary}
\newtheorem{ex}[mythm]{Example}

\newtheorem{innercustomthm}{Theorem}
\newenvironment{customthm}[1]  {\renewcommand\theinnercustomthm{#1}\innercustomthm}
  {\endinnercustomthm}

\def\C{\mathbb{C}}
\def\R{\mathbb{R}}

\def\N{\mathbb{N}}

\newcommand\restr[2]{{
  \left.\kern-\nulldelimiterspace 
  #1 
  \littletaller 
  \right|_{#2} 
  }}

\newcommand{\littletaller}{\mathchoice{\vphantom{\big|}}{}{}{}}

\begin{document}
\begin{center}
\textsc{\textbf{\LARGE{Asymptotics for resolutions and smoothings of Calabi-Yau conifolds}}}
\end{center}

\begin{center}
\textsc{ABDOU OUSSAMA BENABIDA}
\end{center}

\noindent \textsc{Abstract.} We show that the Calabi–Yau metrics with isolated conical singularities of Hein-Sun \cite{hein_calabi-yau_2017} admit polyhomogeneous expansions near their singularities. Moreover, we show that, under certain generic assumptions, natural families of smooth Calabi-Yau metrics on crepant resolutions and on polarized smoothings of conical Calabi–Yau manifolds degenerating to the initial conical Calabi-Yau metric admit polyhomogeneous expansions where the singularities are forming. The construction proceeds by performing weighted Melrose-type blow-ups and then gluing conical and scaled asymptotically conical Calabi-Yau metrics on the fibers, close to the blow-up's front face without compromising polyhomogeneity. This yields a polyhomogeneous family of Kähler metrics that are approximately Calabi-Yau. Solving formally a complex Monge-Ampère equation, we obtain a polyhomogeneous family of Kähler metrics with Ricci potential converging rapidly to zero as the family is degenerating. We can then conclude that the corresponding family of degenerating Calabi-Yau metrics is polyhomogeneous by using a fixed point argument.

\vspace{1cm}

\tableofcontents
\newpage
\section{Introduction} 
\subsection{Overview}
Calabi-Yau manifolds form an important class of complex manifolds, defined by being Kähler and having vanishing first Chern class. Thanks to Yau's theorem \cite{yau_ricci_1978}, compact closed Calabi-Yau manifolds admit a unique Ricci-flat Kähler metric in every Kähler class. Yau's theorem has been generalized by Eyssidieux-Guedj-Zeriahi \cite{eyssidieux_singular_2009} to the case where the Calabi-Yau manifold is the regular part of a normal projective variety with only canonical singularities. In terms of the Berger classification, Calabi-Yau manifolds constitute building blocks of Riemannian manifolds with special holonomy included in $SU(n)$. These play an important role in string theory compactifications. Calabi-Yau manifolds also occur in mirror symmetry. \\

In this paper, we are interested in the situation where $X_0$ is a compact normal complex-analytic variety which is smooth Calabi-Yau outside a set of isolated singularities such that for all $x \in X_0^{\text{sing}}$, the germ $(X_0,x)$ is biholomorphic to a neighborhood of the vertex in a Calabi-Yau cone $C_x$ (with smooth cross section) with a Ricci-flat Kähler cone metric $\omega_{C_x}$. By abuse of language, we say Kähler forms or metrics interchangeably. It is well known that, if the singularities of $X_0$ are orbifold singularities i.e. $(X_0,x) \cong \C^n / G$, where $G \subset SU(n)$ acts freely on $\C^n \setminus \{ 0\}$ and $X_0$ admits orbifold Kähler metrics, then $X_0$ admits a Ricci-flat orbifold Kähler metric $\omega_{CY}$ in every Kähler class; see \cite{joyce2000compact} for example. More generally, according to the result established in \cite{hein_calabi-yau_2017} by Hein-Sun and its recent improvements by Chiu-Székelyhidi \cite{ChiuSzekelyhidi} and Zhang \cite{zhang2024polynomial}, if $X_0$ is a normal projective variety with only canonical singularities and $L_0$ is an ample line bundle, then the unique Ricci-flat Kähler metric $\omega_{CY} \in 2 \pi c_1(L_0)$ on $X_0$ with bounded potential on the germ $(X,x)$ is conical and asymptotic to the cone metric $\omega_{C_x}$ on $C_x$ near $x$. We refer to $(X_0, \omega_{CY})$ as above as a \textbf{conical Calabi-Yau manifold modeled on the cone $(C_x, \omega_{C_x})$ near $x$}. \\

For simplicity, we suppose that $X_0$ has a unique singular point $x \in X_0$. There are two main ways of desingularizing $X_0$ as a Calabi-Yau manifold. When these exist, they are described as follows :
\begin{itemize}
\item One is a \textbf{crepant resolution} given by a smooth Kähler manifold $\hat{X}$ and a proper bi-meromorphic map $\hat{\pi} : \hat X \rightarrow X_0$ such that $\hat{\pi} : \hat{X} \setminus \hat{\pi}^{-1} ( \{ x\}) \rightarrow X_0 \setminus \{x \} $ is a bi-holomorphism and the canonical divisor satisfies $K_{\hat{X}} = \hat{\pi}^{*} K_{X_0}$. Hence, $\hat{X}$ is a smooth compact Calabi-Yau manifold and therefore admits a Ricci-flat Kähler metric in every Kähler class by Yau's theorem.
\item The other one is a \textbf{smoothing} given by $\pi : \mathcal{X} \rightarrow \mathbb{D}$ where $\mathcal{X}$ is an $(n+1)$-dimensional projective variety, $\mathbb{D} \subset \C$ is the unit disk and $\pi$ is a proper flat morphism such that $\pi^{-1}(\{ 0\}) \cong X_0$ and $X_{t} := \pi^{-1}(\{ t\}) $ is smooth for $t \neq 0$ and the relative canonical bundle is trivial, i.e. $\mathcal{K}_{\mathcal{X}/ \mathbb{D}} \cong \mathcal{O}_{\mathcal{X}/\mathbb{D}}$. Therefore, $X_t$ is a smooth compact Calabi-Yau manifold for all $t \in \mathbb{D} \setminus \{0\}$ and hence admits a Ricci-flat Kähler metric in every Kähler class by Yau's theorem. 
\end{itemize}

We are interested in studying families of smooth Ricci-flat Kähler metrics on crepant resolutions and smoothings degenerating to the initial conical Calabi-Yau metric. An earlier work by Chan \cite{chan_desingularizations_2006, chan_desingularizations_2009} gave a general construction, in the complex three dimensional case, of families of smooth Ricci-flat Kähler metrics desingularizing conical Calabi-Yau manifolds by gluing scaled asymptotically conical Calabi-Yau manifolds near the singular points and making a small perturbation using $G_2$ techniques. However, as noted in \cite[Appendix A]{hein_calabi-yau_2017}, in the construction obtained by Chan, a perturbation in the (almost) complex structure makes it unclear which smooth Calabi-Yau 3-folds are being produced. To stay within the realm of crepant resolutions and smoothings, a later work by Arezzo-Spotti \cite{arezzo2016csck} gave a gluing construction for crepant resolutions in any dimension which is carried out only at the level of Kähler potentials. For smoothings, gluing constructions have been obtained in the surface case by Biquard-Rollin \cite{Biquard} for CSCK metrics and Spotti \cite{Spotti} for Kähler-Einstein metrics on Del Pezzo surfaces.    \\ 

In this paper, we obtain a finer description of certain families of smooth Calabi-Yau metrics degenerating to a conical Calabi-Yau metric, obtained by a gluing construction. More precisely, we improve the work of Arezzo-Spotti \cite{arezzo2016csck} on crepant resolutions by considering a less restrictive condition on the cohomology class which allows us to treat certain small resolutions and we give a similar gluing construction in the case of polarized smoothings of conical Calabi-Yau manifolds under certain generic assumptions. Moreover, we prove that, in both cases, such families have asymptotic expansions, known as \emph{polyhomogeneous expansions}, where the singularities are forming.\\

For every manifold with corners $M$, the space of \textbf{polyhomogeneous functions} denoted by $\mathcal{A}_{phg}(M)$ is roughly defined as follows: If $M$ is without boundary, then $\mathcal{A}_{phg}(M) = C^\infty(M)$. Otherwise, $\mathcal{A}_{phg}(M)$ is the space of functions which are smooth on the interior of $M$ and which admit an asymptotic expansion near every boundary hypersurface $H$ in $M$ of the form 
$$
\sum_{(z, k) \in F(H)} a_{(z, k)} \hspace{0.1cm} x_H^z(\log x_H)^k,
$$
where $x_H$ is a boundary defining function for $H$, $F(H) \subset \C \times \N$ is a suitable index set with only finitely many terms at each order $z$ and $a_{(z,k)}$ are functions which are smooth up to the boundary near $H$ and polyhomogeneous near the other boundary hypersurfaces. See Section \ref{phgfnc} for a more precise definition. Such expansions generalize smoothness up to the boundary and usually appear as the boundary behavior of solutions to elliptic equations on manifolds with corners with an additional suitable structure. The space of polyhomogeneous functions is a $C^\infty$-module, therefore, it makes sense to talk about polyhomogeneous sections of a given vector bundle on $M$.\\

Our method is inspired by the approach of Melrose and Zhu \cite{melrose_resolution_2018}, who obtained similar asymptotics for constant curvature metrics on Lefschetz fibrations of Riemann surfaces. See also \cite{MR4271388, zhu2024gluing}  for recent applications of this approach to construct gravitational instantons. The main idea is that, after describing the proper geometric setting on which we wish to obtain the asymptotics using weighted Melrose-type blow-ups and performing an appropriate gluing of metrics, we solve the complex Monge-Ampère equation formally improving the glued metric in a controlled manner, before applying a fixed point argument. These asymptotic expansions, naturally, give a notion of convergence which is stronger than the Gromov-Hausdorff convergence usually obtained.

\subsection{Main results } \label{setting} 
Let $\left(X_0, \omega_{CY} \right)$ be a conical Calabi-Yau manifold modeled on a Calabi-Yau cone $(C_x, \omega_{C_x})$ near every $x \in X_0^{sing}$ in the sense of Definition \ref{def : conical CY}. Denote by $r_x$ the distance function to the vertex on $C_x$ with respect to the metric induced by $\omega_{C_x}$. \\

To state our first result, we recall that a Kähler form $\omega$ on $X_0^{reg}$ is said to be \emph{smoothly Kähler on $X_0$} if it is smooth on the regular part and near a singularity, it is given by the restriction of a smooth Kähler form under local embedding into a smooth Kähler manifold. In particular, this holds for the restriction of the Fubini-Study metric on projective varieties.
\begin{customthm}{A} \label{tmA} 
If $(X_0, \omega_{CY})$ as above is such that $\omega_{CY} \in [\omega] \in H^{1,1}(X_0^{reg}, \R)$ for a smoothly Kähler form $\omega$ and $X_0$ has trivial canonical bundle, then the conical Calabi-Yau metric $\omega_{CY}$ admits a polyhomogeneous expansion near every singularity in terms of the radial function $r_x$, with non-negative index set $F_x$ i.e. $F_x \subset \left( (0, \infty) \times \N \right) \cup \left( \{0 \} \times \{  0\} \right)$.
\end{customthm}
The proof of this theorem is given in Section \ref{conphg}. We follow a similar approach to that in \cite{conlon_moduli_2015}, where a corresponding result is obtained for asymptotically cylindrical and asymptotically conical Calabi-Yau metrics. This relies on studying the solutions of the linearized equation of the complex Monge-Ampère equation, i.e. $\Delta_c u = v$, where $\Delta_c$ is the Laplacian of a conical metric, using results coming from the theory of b-calculus and b-geometry \cite{melrose_atiyah-patodi-singer_1993, Rafe_elliptic, melrose_pseudodifferential_1990, melrose_calculus_1992}. To avoid complications in our work, we will not keep track of the powers appearing in the expansion, but we mention that they are related to the spectrum of the Laplacian on the link of the cone with its induced metric.\\

In all of the following, the conical Calabi-Yau metric $\omega_{CY}$ will be supposed to be polyhomogeneous.
\subsubsection*{Crepant resolutions} 
Now, let $m := \# X_0^{sing} < \infty$ and we denote the Calabi-Yau cone models near $x_i \in X_0^{sing}$ by $(C_i, \omega_{C_i})$ and the associated radial function by $r_i$. Suppose that, for all $i \in \{ 1,2, \cdots, m\}$, $C_i$ has a crepant resolution given by $ \hat \pi_{\hat C_i} : \hat C_i \rightarrow C_i$ and denote $E_i := \hat \pi_{\hat C_i}^{-1}(\{ o_i\})$, where $o_i$ is the vertex of $C_i$. For $\varepsilon >0$, let $Y_{i,\varepsilon} := \{p \in C_i; \hspace{0.1cm} r_i(p) \geq \frac{1}{\varepsilon}  \}$. Up to composing by a scaling, consider a biholomorphic identification $U_{i,0} \cong \{p \in C_i; \hspace{0.1cm} r_i(p)< \frac{2}{\varepsilon} \}$, where $U_{i,0} \subset X_0$ is a neighborhood of $x_i \in X_0^{sing}$. Together with the biholomorphic identification $\hat C_i \setminus E_i \cong C_i \setminus \{ o_i\}$, we can, therefore, consider the holomorphic gluing $\hat X_\varepsilon := \left((X_0 \setminus X_0^{sing}) \bigcup_{i =1}^m \hspace{0.1cm} (\hat C_i \setminus Y_{i,\varepsilon}) \right) \Big/ \sim,$ where $\sim$ is the equivalence relation identifying the images of $\{p \in C_i;\hspace{0.1cm} r_i(p) < \frac{1}{\varepsilon} \} \setminus \{ o_i\}$ inside $U_{i, 0} \setminus \{ x_i\}$ and $\hat C_i \setminus (E_i \cup Y_{i, \varepsilon})$ respectively, under the biholomorphic identifications described above. As complex manifolds, all $X_\varepsilon$ are biholomorphic, hence, we denote the underlying complex manifold by $\hat X$. Now, by a result of Goto \cite{goto_calabi-yau_2012}, $\hat C_i$ admits an asymptotically conical Calabi-Yau metric in every Kähler class. Let $\omega_{AC,i}$ denote such a metric. Using a gluing construction, Arezzo-Spotti \cite{arezzo2016csck} constructed a family of smooth Calabi-Yau metrics $\omega_{CY, \varepsilon}$ on $\hat X_\varepsilon$ which degenerates as $\varepsilon \rightarrow 0$ to the initial conical Calabi-Yau metric $\omega_{CY}$ under the assumption that the cohomology classes of $\omega_{AC,i}$ are compactly supported in $H^2(\hat C_i, \R)$. Here, we will consider a more general assumption on the cohomology classes of $\omega_{AC,i}$ which allows us to also treat examples of \emph{small resolutions} i.e. resolutions such that $\mathrm{codim}_{\C} \hspace{0.1cm} E_i>1$. 
\begin{itemize}
    \item \textbf{Assumption \hypertarget{R}{R}:} Suppose that there exists $\lambda >0$ such that $$\left( \left[\restr{\omega_{AC,1}}{\hat C_1 \setminus Y_{1,\lambda}} \right], \cdots, \left[\restr{\omega_{AC,m}}{\hat C_m \setminus Y_{m,\lambda}} \right] \right)$$ is in the image of the cohomology restriction map $$ H^{1,1}(\hat X_\lambda, \R) \rightarrow H^{1,1} \left( \bigsqcup_{i=1}^m (\hat C_i \setminus Y_{i,\lambda}), \R \right) \cong \bigoplus_{i=1}^m H^{1,1} \left( \hat C_i \setminus Y_{i,\lambda}, \R \right). $$
\end{itemize}
\hfill
\begin{rem}
\hfill
    \begin{itemize}
        \item Notice that, by a simple gluing argument, Assumption \hyperlink{R}{R} is, in particular, satisfied if the Kähler classes $[\omega_{AC,i}]$ are compactly supported in $H^2(\hat C_i, \R)$. However, as noted by Van Coevering \cite{van_coevering_ricci-flat_2010}, if the resolution is small then $\hat C_i$ has no compactly supported Kähler classes. 
        \item Our methods will show that Assumption \hyperlink{R}{R} will also guarantee that the complex manifold $\hat X$ is, in particular, Kähler. This is not immediate from the holomorphic gluing. 
    \end{itemize}
\end{rem}
To keep notations light, suppose that $m=1$ and $x$ is the only singular point in $X_0$ with associated Calabi-Yau cone $(C, \omega_C)$ whose crepant resolution is given by $(\hat C, \hat \pi_{\hat C})$. In Section \ref{crepant resolution}, we will construct an appropriate manifold with corners $\mathcal{M}_b$ with two boundary hypersurfaces $B_{I}$ and $B_{II}$ such that 
\begin{itemize}
    \item $\mathcal{M}_b \setminus (B_I \cup B_{II}) \cong \hat X \times (0,\varepsilon_0]_\varepsilon$;
    \item $\overset{\circ}{B_I} \cong \hat{C}$, where $\overset{\circ}{B_I}$ is the interior of $B_I$;
    \item $\overset{\circ}{B_{II}} \cong X_0 \setminus \{ x\}$, where $\overset{\circ}{B_{II}}$ is the interior of $B_{II}$;
    \item $\partial B_I = \partial B_{II} \cong L$, where $L$ is the link of the cone $C$.
\end{itemize}
See Figure \ref{Figure 1} for illustration. The construction uses weighted Melrose-type blow-ups. See Section \ref{blow} for definition. 
\begin{rem}
In general, if $u$ is any type of object defined on $\mathcal{M}_b$ (e.g. a function or a form), we use the notation $u_\varepsilon$ to denote its restriction to a fiber $\hat{X} \times \{ \varepsilon \}$. 
\end{rem}

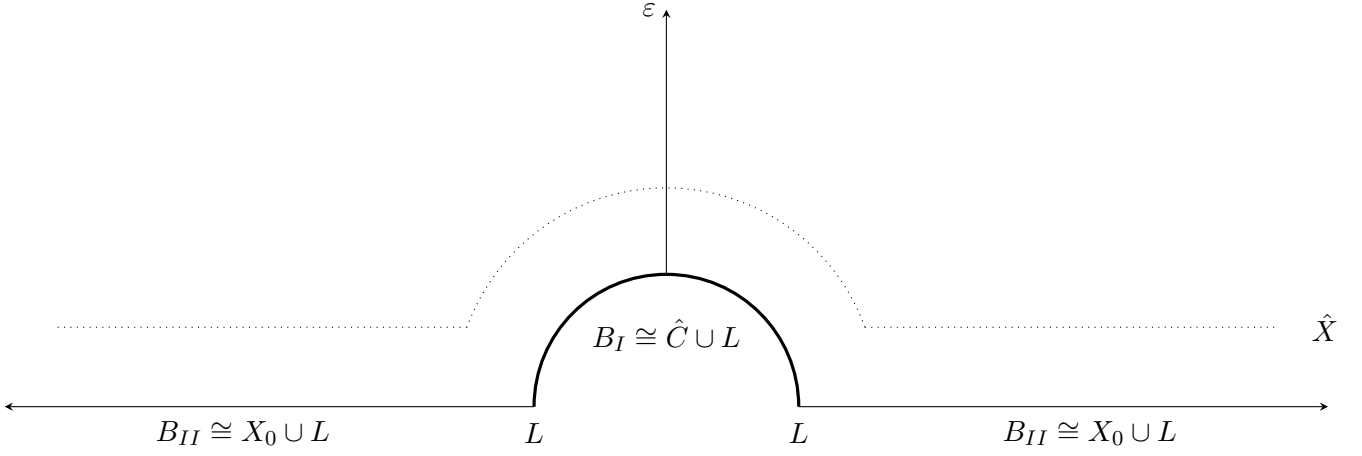
\begin{figure}[ht] 
\centering
\begin{tikzpicture}[scale=3.5,>=stealth]
  \draw[->] (0.5,0) -- (2.5,0) ;
  \draw[->] (0,0.5) -- (0,1.5) node[left] {$\varepsilon$};

  \draw[->]  (-0.5,0) -- (-2.5,0);

  \draw[very thick] (0.5,0) arc[start angle=0, end angle=180, radius=0.5];

  \node at (0, 0.27) {$B_I  \cong \hat C \cup L$};

  \node at (1.6,-0.1) {$B_{II} \cong X_0 \cup L$};
  \node at (-1.6,-0.1) {$B_{II} \cong X_0 \cup L$};
  \node at (0.5, -0.1) {$L$};
  \node at (-0.5, -0.1) {$L$};

  \draw[dotted] (-2.3,0.3) -- (-0.75,0.3);
  \draw[dotted] (0.75,0.3) arc[start angle=20, end angle=160, radius=0.8];
  \draw[dotted] (0.75,0.3) -- (2.3,0.3);
  \node[right] at (2.4,0.3) {$\hat X$};


\end{tikzpicture} \caption{The blown-up parametric space $\mathcal{M}_b$ in the case of a crepant resolution.}
\label{Figure 1}
\end{figure}

\begin{customthm}{B} \label{tmB}
Under assumption \hyperlink{R}{R}, there exists a family of smooth Calabi-Yau metrics $\omega_{CY, \varepsilon}$ on the crepant resolution $\hat X$, for $\varepsilon>0$ sufficiently small, which is polyhomogeneous on $\mathcal{M}_b$ with restrictions $\restr{\omega_{CY, \varepsilon}}{B_{II}} = \omega_{CY}$ and $\restr{\frac{\omega_{CY, \varepsilon}}{\varepsilon^2}}{B_I} = \omega_{AC}$.
\end{customthm}

Roughly speaking, this result means that, if $\rho_1$ and $\rho_2$ are boundary defining functions for $B_I$ and $B_{II}$ respectively, the family $\omega_{CY, \varepsilon}$ admits an expansion near $B_I$ and $B_{II}$ of the form 
\begin{align*}
     \omega_{CY, \varepsilon} \cong \omega_{CY} + i \partial \overline{\partial} u_2, \hspace{0.1cm} & u_2 \sim \sum_{(\lambda,k) \in F(B_{II})} a_{(\lambda,k)}\rho_2^{\lambda} (\log \rho_2)^k \hspace{0.1cm} \text{when} \hspace{0.1cm} \rho_2 \rightarrow 0;\\
     \omega_{CY, \varepsilon} \cong \varepsilon^2 \omega_{AC} + \varepsilon^2 i \partial \overline{\partial} u_1, \hspace{0.1cm} &u_1 \sim \sum_{(\lambda,k) \in F(B_{I})} b_{(\lambda,k)}\rho_1^{\lambda} (\log \rho_1)^k \hspace{0.1cm} \text{when} \hspace{0.1cm} \rho_1 \rightarrow 0.
\end{align*}
\subsubsection*{Polarized smoothings}
Now, suppose, in addition, that $X_0$ is projective with canonical singularity at $x$ and $L_0$ is an ample line bundle on $X_0$ such that $\omega_{CY} \in 2 \pi c_1(L_0)$. Suppose that $(\mathcal{X}, \mathcal{L}, \pi)$ is a polarized smoothing of $(X_0,L_0)$ with trivial relative canonical bundle. Suppose that the smoothing satisfies the following assumptions 
\begin{itemize}
    \item \textbf{ Assumption \hypertarget{S.1}{S.1}:} The smoothing $\pi :\mathcal{X} \rightarrow \mathbb{D}$ is locally isomorphic near $x \in X_0 \subset \mathcal{X}$ to an affine smoothing of the cone $C$ given by $p:W \subset \C^N \times \C_t \rightarrow \C_t$ and $W$ is invariant with respect to a diagonal $\R_{>0}$-action on $\C^N \times \C_t$ with real positive weights, which restricts on the cone $C$ to the scaling action generated by $r \partial_r$. See Section \ref{polarized smoothing} for a precise description. See also Remark \ref{rem ex}.
    \item \textbf{ Assumption \hypertarget{S.2}{S.2}:} The general fiber of the affine smoothing $V \cong p^{-1}(\{ 1\})$ admits an asymptotically conical Calabi-Yau metric $\omega_{AC}$ which is $i \partial \overline{\partial}$-exact outside a compact. 
\end{itemize}
\begin{rem}
   In the case of smoothings, Assumption \hyperlink{S.2}{S.2} is already satisfied for a large set of examples that are of interest as we show in Remark \ref{rem ex}. 
\end{rem}

Consider a ray in $W$ given by $W_\theta := p^{-1}\left( \{ t e^{i \theta},  t \in[0, \infty) \}\right)$ for a fixed $\theta \in S^1$. Without loss of generality, suppose $\theta=0$ and let $\mathcal{X}_0 \subset \mathcal{X}$ be the path in $\mathcal{X}$ which maps to $W_0$ under the local isomorphism. Similar to before, in Section \ref{polarized smoothing}, we will construct an appropriate manifold with corners $\mathcal{X}_b$ with two boundary hypersurfaces $B_{I}$ and $B_{II}$ such that
\begin{itemize}
    \item $\mathcal{X}_b \setminus B_I \cong \mathcal{X}_0 \setminus \{x \}$;
    \item $\overset{\circ}{B_I} \cong V$, where $\overset{\circ}{B_I}$ is the interior of $B_I$;
    \item $\overset{\circ}{B_{II}} \cong X_0 \setminus \{ x\}$, where $\overset{\circ}{B_{II}}$ is the interior of $B_{II}$;
    \item $\partial B_I = \partial B_{II} \cong L$, where $L$ is the link of the cone $C$.
\end{itemize}
Using the same notations, we prove the following result.
\begin{customthm}{C} \label{tmC}
    Under assumptions \hyperlink{S.1}{S.1} and \hyperlink{S.2}{S.2} and considering an appropriate change of coordinates $s := t^{\frac{1}{\mu}}$ for a certain $\mu>0$, there exists a family of smooth Calabi-Yau metrics $\omega_{CY,s}$, for $s>0$ sufficiently small, on the fibers $X_s:= \pi^{-1}(\{s^\mu\})$ of the polarized smoothing $(\mathcal{X}, \mathcal{L}, \pi)$, along the path $\mathcal{X}_0 \subset \mathcal{X}$, which is polyhomogeneous on $\mathcal{X}_b$ with restrictions $\restr{\omega_{CY,s}}{B_{II}} = \omega_{CY}$ and $\restr{\frac{\omega_{CY,s}}{s^{2}}}{B_I} = \omega_{AC}$.
\end{customthm}

In Theorem \ref{main}, we state a more general result and in Section \ref{main examples} we give the necessary constructions to apply the theorem in the above settings of crepant resolutions and polarized smoothings to get theorems \ref{tmB} and \ref{tmC}.\\

Our strategy is similar in the two cases of crepant resolutions and polarized smoothings with a minor difference in the construction of the initial family of the Kähler forms. For convenience, we only describe it in the case of a crepant resolution 
 \begin{enumerate}
     \item For small $\varepsilon>0$, we perform appropriate polyhomogeneous gluing of lifts of $\omega_C$ and $\varepsilon^2 \omega_{AC}$ to get a Kähler form on $\hat{X}$ denoted by $\omega_\varepsilon$. 
     \item Next, if $v_\varepsilon$ denotes the normalized potential of the Ricci-form of $\omega_\varepsilon$ i.e. $\mathrm{Ric}(\omega_\varepsilon) = i \partial \overline{\partial} v_\varepsilon$ such that $\int_{\hat{X}} \omega_\varepsilon^n = \int_{\hat{X}} e^{v_\varepsilon} \omega_\varepsilon^n$,
     we prove that $v_\varepsilon$ is polyhomogeneous on the blown-up space $\mathcal{M}_b$ and vanishes at $B_I$ and $B_{II}$. We do so by considering $v_\varepsilon$ as a solution to
     $$
     \Delta_{\omega_{\varepsilon}} v_\varepsilon = s_{\varepsilon},
     $$
     where $s_{\varepsilon}$ denotes the scalar curvature and $\Delta_{\omega_\varepsilon}$ is the Laplacian.
     \item  After that, we wish to solve
     $$
     \frac{(\omega_\varepsilon + i \partial \overline{\partial} u_\varepsilon)^n}{\omega_\varepsilon^n} = e^{v_\varepsilon}. 
     $$
     First, we solve the equation formally in the polyhomogeneous sense, to get 
     $$
     \frac{(\omega_\varepsilon + i \partial \overline{\partial} u_{0,\varepsilon})^n}{\omega_\varepsilon^n} = e^{v_\varepsilon}-g_\varepsilon, 
     $$
     where $g_\varepsilon$ vanishes rapidly with all its derivatives at $\varepsilon=0$. The formal solution depends on solving linear equations involving the Laplacian $\Delta_\varepsilon$ iteratively.  
     \item Then, using a Banach fixed point argument, we prove that there exists a unique $\tilde u$ vanishing rapidly with all its derivatives at $\varepsilon=0$ such that 
     $$
     \frac{(\omega_\varepsilon + i \partial \overline{\partial} u_{0,\varepsilon} + i \partial \overline{\partial} \tilde u)^n}{\omega_\varepsilon^n} = e^{v_\varepsilon}.
     $$
 \end{enumerate}
\subsection{Examples} 
First, let us show important examples for which our theorems apply.
\begin{ex}
Let $X_0$ be a hypersurface in $\mathbb{CP}^{n+1}$ of degree $n+2$ with a nodal singularity at a point $x \in X_0$. Therefore, $X_0$ is analytically isomorphic near $x$ to
$$
C = \{ z \in \C^{n+1}; \hspace{0.1cm} \sum_{i=1}^{n+1} z_i^2 =0 \}.
$$

As a Calabi-Yau cone, $C$ admits an explicit Ricci-flat Kähler cone metric $\omega_C$ known as the \emph{Stenzel metric} \cite{stenzel_ricci-flat_1993, candelas_comments_1990} given by 
$$
\omega_C = i \partial \overline{\partial}(r^2) = i \partial \overline{\partial} \left( ||z||^{2}\right)^{\frac{n-1}{n}}.
$$
In particular, $r$ is homogeneous of degree $1$ with respect to the diagonal $\R_{>0}$-action given by $
\lambda \cdot z = (\lambda^{\frac{n}{n-1}}z_1, \cdots, \lambda^{\frac{n}{n-1}}z_n)$. If $L_0 := \restr{\mathcal{O}(1)}{X_0}$, then, by Hein-Sun \cite{hein_calabi-yau_2017}, the unique Ricci-flat Kähler metric $\omega_{CY} \in 2 \pi c_1(L_0)$ on $X_0$ is asymptotic to the Stenzel metric $\omega_C$ near $x$. In addition, by Theorem \ref{tmA}, we get that $\omega_{CY}$ is polyhomogeneous near $x$.\\

Consider a smoothing $\pi :\mathcal{X} \rightarrow \mathbb{D}$ of $X_0$ in $\mathbb{CP}^{n+1}$ and $\mathcal{L} := \restr{\mathcal{O}_{\mathbb{D}}(1)}{\mathcal{X}}$. By a result of Kas-Schlessinger \cite{KasSchlessinger}, the smoothing $\mathcal{X}$ is isomorphic near $x$ to $$
W = \{(z,t) \in \C^{n+1}  \times \C; \hspace{0.1cm} \sum_{i=1}^{n+1} z_i^2 = t^k \}, 
$$
for a certain $k \in \N$. Therefore, $W$ is homogeneous with respect to the diagonal $\R_{>0}$-action given by $
\lambda \cdot (z,t) = (\lambda^{\frac{n}{n-1}}z_1, \cdots, \lambda^{\frac{n}{n-1}}z_n, \lambda^{\frac{2n}{k(n-1)}} t),$ hence, the smoothing $\pi :\mathcal{X} \rightarrow \mathbb{D}$ satisfies Assumption \hyperlink{S.1}{S.1}. Moreover, Stenzel \cite{stenzel_ricci-flat_1993}, constructed an asymptotically conical Ricci-flat Kähler metric $\omega_{AC}$ on the general fiber of $W$ which is $i \partial \overline{\partial}$-exact i.e. Assumption \hyperlink{S.2}{S.2} is satisfied. Therefore, Theorem \ref{tmC} implies that, along a path in $\mathcal{X}$ denoted by $\mathcal{X}_0$, corresponding to a ray in $W$, there is a family of Ricci-flat Kähler metrics $\omega_{CY,s}$ on the fibers of $\mathcal{X}_0$ which is polyhomogeneous on $\mathcal{X}_b$ as constructed above with restrictions $\restr{\omega_{CY, s}}{B_{II}} = \omega_{CY}$ and $\restr{\frac{\omega_{CY, s}}{s^2}}{B_I} = \omega_{AC}$.
\end{ex}
\begin{ex}
    Let $X_0$ be a hypersurface in $\mathbb{CP}^4$ given by the equation 
    $$
    X_0 = \left\{[\xi_0 : \cdots : \xi_4] \in \mathbb{CP}^4; \hspace{0.1cm} \xi_3 g(\xi_0, \cdots, \xi_4) + \xi_4h(\xi_0, \cdots, \xi_4) =0 \right\},
    $$
    where $g$ and $h$ are generic homogeneous polynomials of degree $4$. As described in \cite[Section 1, Section 2]{Rossi}, $X_0$ is a Calabi-Yau variety which has a set of $16$ nodal singularities $x_i$ given by 
    $$
    \xi_3 =\xi_4= g(\xi)=h(\xi)=0.
    $$
    We let $\omega_{CY}$ be a conical Calabi-Yau metric on $X_0$. Therefore, $\omega_{CY}$ is polyhomogeneous near the singularities by Theorem \ref{tmA}.\\
    
    Moreover, $X_0$ admits a simultaneous small resolution by a smooth Calabi-Yau manifold $\hat X$ and the nodes $x_i$  are replaced by $(-1,-1)$ curves i.e. rational curves $E_i \cong \mathbb{CP}^1$ with normal bundle identified with the total space of the holomorphic vector bundle $\mathcal{O}(-1) \oplus \mathcal{O}(-1)$ over $\mathbb{CP}^1$. Therefore, a neighborhood $U_i$ of $E_i$ can be identified with $\mathrm{Tot}(\mathcal{O}(-1) \oplus \mathcal{O}(-1)) \setminus Y_i$ where $Y_i \cong \{r \geq 1  \} \subset C = \{z \in \C^4; \hspace{0.1cm} \sum z_i^2=0 \}$.\\
    
    Now, let $\eta$ be an arbitrary Kähler form on $\hat X$. Its restriction on $U_i \cong \mathrm{Tot}(\mathcal{O}(-1)^{\oplus 2}) \setminus Y_i$ gives a  Kähler form that we denote by $\eta_i$.
    Therefore, $\eta_i$ defines a non-trivial class in $H^{1,1}\left(\mathrm{Tot}(\mathcal{O}(-1)^{\oplus 2}) \setminus Y_i, \R \right)$ and, since $H^2\left(\mathrm{Tot}(\mathcal{O}(-1)^{\oplus 2}) \setminus Y_i, \R \right) \cong H^2(\mathbb{CP}^1, \R) \cong \R$, we get that $[\eta_i] = \alpha_i [p^{*} \omega_{FS}]$ where $p : \mathcal{O}(-1)^{\oplus 2} \rightarrow \mathbb{CP}^1$ is the projection, $\omega_{FS}$ is the Fubini-Study metric and $\alpha_i >0$. Now,  let $\omega_{AC,i}$ be an asymptotically conical Calabi-Yau metric on $\mathrm{Tot}(\mathcal{O}(-1)^{\oplus 2})$ such that $[\omega_{CY,i}] = \alpha_i [p^* \omega_{FS}]$. Such metrics have been constructed explicitly by Candelas-De la Ossa \cite{candelas_comments_1990}. Therefore, $\left([\restr{\omega_{AC,1}}{U_1}], \cdots, [\restr{\omega_{AC,16}}{U_{16}}] \right)$ is in the image of the restriction map $H^{1,1}(\hat X, \R) \rightarrow H^{1,1}\left(\bigsqcup_{i=1}^{16} U_i, \R \right) \cong \bigoplus_{i=1}^{16} H^{1,1}(\mathrm{Tot}(\mathcal{O}(-1)^{\oplus 2}) \setminus Y_i, \R)$. Hence, a minor extension of Theorem \ref{tmB} to the case of several isolated conical singularities implies that $\hat X$ admits a family of Ricci-flat Kähler metrics $\omega_{CY,\varepsilon}$, for $\varepsilon$ sufficiently small, which is polyhomogeneous on $\mathcal{M}_b$ similar to the one described above but with $16$ new additional faces $B_{I,i}$ and with restrictions $\restr{\omega_{CY, \varepsilon}}{B_{II}} = \omega_{CY}$ and $\restr{\frac{\omega_{CY, \varepsilon}}{\varepsilon^2}}{B_{I,i}} = \omega_{AC,i}$.
\end{ex}

In the following, we give some general guiding principles for finding examples.
\begin{rem} \label{rem ex}
\hfill
\begin{itemize}
\item By arguments of Conlon-Hein in \cite[Section 5.1]{conlon2013asymptotically} and the results in \cite[Theorem A]{conlon2015asymptotically} and \cite[Theorem A, Theroem B]{conlon_classification_2024}, if the Calabi-Yau cone $C$ is regular and a complete intersection, then every affine smoothing of $C$ satisfies Assumption \hyperlink{S.2}{S.2}. Moreover, the versal deformation of $C$, obtained by Kas-Schlessinger \cite{KasSchlessinger}, itself is $\R_{>0}$-equivariant. In fact, it is $\C^*$-equivariant. See \cite[Theorem 2.5, remark 1) p.12-13]{MR584445}. Therefore, smoothings satisfying Assumption \hyperlink{S.1}{S.1} are classified by equivariant complex-analytic maps to the versal deformation.
\item More generally, asymptotically conical Calabi-Yau metrics have been completely classified by Conlon-Hein \cite{conlon_classification_2024}. Therefore, if the general fiber $V$ of an affine smoothing of $C$ is a deformation of negative $\xi$-weight in the sense of \cite[Definition 1.7]{conlon_classification_2024} and supposing the existence of a compactly supported Kähler class then Assumption \hyperlink{S.2}{S.2} is satisfied on $V$.
\item For orbifold singularities, crepant resolutions and smoothings of $\C^m /G$ for $G \subset SU(n)$ acting freely on $\C^m \setminus \{ 0\}$, are fairly understood. See Joyce \cite[Section 6.4]{joyce2000compact}. For example, If $m \geq 3$, then $\C^m /G$ has no non-trivial deformation.
    \end{itemize}
\end{rem}

\subsection{Future work}

There is a general process where one starts with a smooth Calabi-Yau $\hat{X}$ and contracts certain submanifolds to get a conical Calabi-Yau manifold $X_0$ then passes to a smoothing with trivial relative canonical bundle, in the complex-analytic category. This process is known as a \textbf{conifold transition} and denoted by 
$$
\hat{X} \rightarrow X_0 \leadsto X_t.
$$
A folklore conjecture raised by Reid \cite{reid1987moduli} proposes that all Calabi-Yau 3-folds are connected by a finite sequence of conifold transitions (with nodal singularities). It is known that, even if one starts with $\hat X$ projective, it is still possible to get non-Kähler $X_t$ with trivial canonical bundle after the smoothing. Such manifolds are called non-Kähler Calabi-Yau manifolds. A simple example is given by letting $\hat{X}$ be a smooth quintic threefold, therefore, in partiuclar, we have $b_2(\hat{X})=1$. After contracting some $(-1,-1)$ curves, on the smoothing, one can get that $b_2(X_t)=0$ and therefore $X_t$ cannot be Kähler. We refer to \cite{friedman1991threefolds} for further discussion. Studying such non-Kähler Calabi-Yau manifolds is important to address Reid's conjecture \cite{reid1987moduli}. In such a case, Ricci-flat Kähler metrics are instead replaced by a pair $(g, H)$ such that  
\begin{itemize}
    \item $g$ is a balanced metric i.e $g$ is a Hermitian metric on $T^{1,0}X$ such that 
    $$
    d \omega^{n-1} =0,
    $$
    where $\omega$ is the $(1,1)$-form associated to $g$. 
    \item $H$ is a Hermitian-Yang-Mills metric with respect to $\omega$ i.e 
    $$
    F \wedge \omega^{n-1} =0,
    $$
    where $F$ is the Chern curvature of the metric $H$.
\end{itemize}
The inspiration of this system comes from supersymmetric constraints in string theory. When $X$ is Kähler, $g=H=g_{CY}$, where $g_{CY}$ is a Ricci-flat Kähler metric solves such a system. In the non-Kähler three dimensional case, such a system is solved in \cite{fu2012balanced, collins2024stability}. In a recent work \cite{friedman_gromov-hausdorff_2024}, the authors prove Gromov-Hausdorff continuity of such families of metrics along three-dimensional conifold transitions when the resulting fibers in the smoothing are non-Kähler. Therefore, a natural question to ask is whether we can have similar polyhomogeneity results to the ones we prove in this paper in the non-Kähler case. This, of course, should involve a study of the PDEs coming from the Strominger system which are different from the complex Monge-Ampère equation we consider in our work.\\

Another problem that we plan to address in an ongoing work is to use the existence of such polyhomogeneous expansions along conical degenerations to give a precise uniform construction of the resolvent of the Hodge Laplacian $(\Delta - \lambda)^{-1}$ using the work of \cite{Albin}. Such a construction could be used to address the behavior of different types of spectral invariants along such conical degenerations.

\subsection{Organization of the paper}
The paper is organized in the following way. In Section \ref{prel}, we give general preliminaries related to b-geometry, polyhomogeneity and the blow-up in the sense of Melrose. After that, in Section \ref{conifolds}, we mention some well known results on conical and asymptotically conical manifolds and see examples of these in the context of Calabi-Yau manifolds. We prove Theorem \ref{tmA} in Section \ref{conphg}. In Section \ref{set}, we describe a general setting and give the statement of our main Theroem \ref{main}, then we give the constructions needed to apply it to get Theorems \ref{tmB} and \ref{tmC}. The proof of Theorem \ref{main} is carried out in Sections \ref{secform} and \ref{secfix}. 
\subsection*{Acknowledgement} The author is very grateful to his PhD supervisor, professor Frédéric Rochon for the many helpful discussions and suggestions related to this project. The author is also thankful to Àlvaro Sánchez Hernández and Cipriana Anghel for the fruitful discussions related to b-geometry and b-calculus as well as to the CIRGET working group at UQAM for the introduction of the work of Hein-Sun.
\section{Preliminaries on b-geometry } \label{prel}
In this section, we define various geometric objects and notions of regularity in the language of b-geometry (the 'b' stands for 'boundary'). This is consistent with the analytic and geometric methods carried out in this paper. For further details on the language of b-geometry, see the following important references \cite{melrose_atiyah-patodi-singer_1993, grieser2017scalesblowupquasimodeconstructions, melrose_calculus_1992, melrose_pseudodifferential_1990, melrose1996differential}. 
\subsection{Manifolds with corners}
Roughly, manifolds with corners are spaces modeled on 
$$
\R^{n,k} := [0, \infty)^{k} \times \R^{n-k},
$$
with a smooth structure induced from $\R^{n}$. We will also require the boundary hypersurfaces to be embedded. See \cite{melrose1996differential, melrose_atiyah-patodi-singer_1993} for details. 
\begin{defi}
\label{def:manifold-with-corners}
\hfill
\begin{itemize}
    \item A \textbf{t-manifold of dimension \(n\)} is a second-countable, Hausdorff topological space \(M\) such that for all \(p \in M\) there is a neighborhood \(p \in U\) and a homeomorphism
\[
\varphi_U: U \longrightarrow \Omega \subset \R^{n,k},
\]
where \(\Omega\) is open in \( \R^{n,k} = [0,\infty)^{k} \times \mathbb{R}^{n-k}\) and for two such neighborhoods $U$ and $V$ with non-empty intersection, the \emph{transition map}
\[
\varphi_V \circ \varphi_U^{-1} : \varphi_U(U \cap V) \longrightarrow \varphi_V(U \cap V)
\]
is a diffeomorphism that sends boundary strata to boundary strata. Such $k$ is called the \emph{codimension} of $p$, the maps $\varphi_U$ are called \emph{charts} and the collection of charts is called an \emph{atlas}.
A $C^\infty$ structure with corners on $M$ is a maximal such atlas.
\item A \textbf{boundary face of codimension $k$} of $M$ is the closure of a connected component of the set of points of codimension $k$. The set of boundary faces of codimension $k$ is denoted by $\mathcal{M}_k(M)$. 
\item A t-manifold $M$ is called a \textbf{manifold with corners} if every boundary hypersurface is an embedded t-submanifold of $M$ i.e. for all $H \in \mathcal{M}_1(M)$ and $p \in H$, there exists a chart $(\phi, U)$ based at $p$, a linear transformation $G \in GL(n, \R)$ and a neighborhood $\Omega \subset \R^n$ of $0$ such that 
$$
\restr{\phi}{H} : p \in U \cap H \rightarrow G \cdot \left( \R^{n-1, k'} \times \{ 0\} \right) \cap \Omega,
$$
for some integer $k'$. We write \emph{mwc}, short for manifold with corners.
\end{itemize}
\end{defi}
See figures $2$ and $3$ for examples and non-examples of manifolds with corners.
\begin{rem}
    The embeddedness condition on boundary hypersurfaces implies the existence of a \textbf{boundary defining function} for each $H \in \mathcal{M}_1(M)$ i.e. a smooth function $\rho_H : M \rightarrow \R_{\geq 0}$ such that $H = \{ \rho_H = 0 \}$ and $d \rho_H$ is nowhere zero on $H$. We write \textbf{bdf}, short for boundary defining function. The product of bdf's for all components of the boundary is called a \textbf{total boundary defining function}.
\end{rem}
\vspace{1cm}
\begin{figure}[H]
    \centering

\begin{tikzpicture}[scale=1.2,>=stealth]
\begin{scope}
  \draw (0,0) -- (2,0);
  \fill (0,0) circle (1.5pt);
  \node[below] at (0,0) {1};
  \node at (1,0.3) {0};
  \node at (1,-0.3) {$\mathbb{R}_+$};
\end{scope}

\begin{scope}[xshift=3cm]
  \draw[->] (0.7,0) -- (2,0);
  \draw[->] (0,0.7) -- (0,2);
  \draw (0.7,0) arc[start angle=0, end angle=90, radius=0.7];
  \fill (0,0.7) circle (1.5pt);
  \fill (0.7,0) circle (1.5pt);
  \node[below left] at (0,0.7) {2};
  \node[below left] at (0.7,0) {2};
  \node at (1.2,1.2) {0};
  \node at (0.63, 0.63) {1};
  \node  at (1.35,0.2) {1};
  \node at (-0.2,1.35) {1};
\end{scope}


\begin{scope}[xshift=6cm, yshift = 0.3cm]
  \draw
    (0,0)
      .. controls (1,0.5) and (2,0.5) ..
      (3,0)
      .. controls (2,-0.5) and (1,-0.5) ..
      (0,0);

  \fill (0,0) circle (1.5pt);
  \node[left] at (0,0) {2};
  \fill (3,0) circle (1.5pt);
  \node[right] at (3,0) {2};

  \node at (1.5,0.6) {1};
  \node at (1.5,-0.6) {1};

  \node at (1.5,0) {0};
\end{scope}
\end{tikzpicture}

    \caption{Examples of manifolds with corners with codimensions of faces indicated.}
    \label{fig:enter-label}
\end{figure}
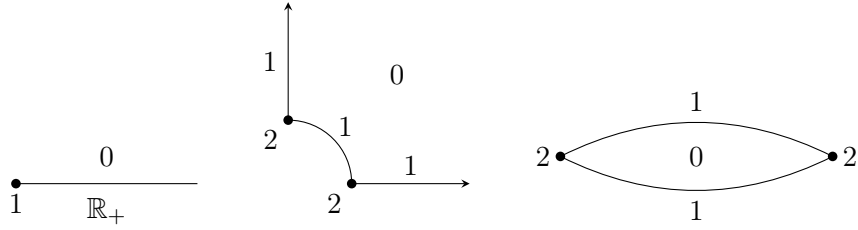

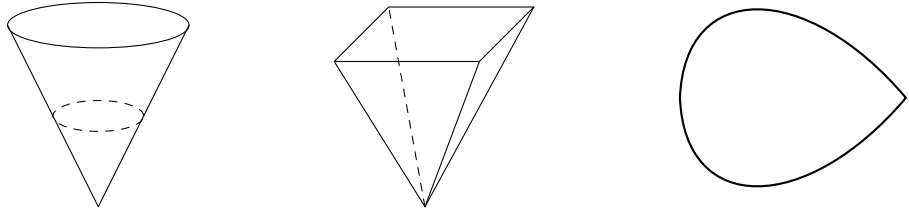
\begin{figure}[H] 
    \centering

\begin{tikzpicture}[scale=1.2, line cap=round, line join=round]

\begin{scope}[xshift = -1.4cm]
  \coordinate (A) at (0,0);

  \coordinate (Top) at (0,2);
  \coordinate (Mid) at (0,1);

  \draw (Top) ellipse [x radius=1, y radius=0.25];
  \draw[dashed] (Mid) ellipse [x radius=0.5, y radius=0.17];

  \draw (A) -- ++(-1,2);
  \draw (A) -- ++( 1,2);
\end{scope}

\begin{scope}[xshift=2.2cm]
  \coordinate (P) at (0,0);

  \coordinate (BL) at (-1,1.6); 
  \coordinate (BR) at (0.6,1.6); 
  \coordinate (TR) at (1.2,2.2); 
  \coordinate (TL) at (-0.4,2.2); 

  \draw (TL)--(TR)--(BR)--(BL)--cycle;

  \draw (P)--(BL);
  \draw (P)--(BR);
  \draw (P)--(TR);
  \draw[dashed] (P)--(TL);
\end{scope}

\begin{scope}[xshift=5cm, yshift = 1.2cm]
  \path[draw=black, line width=0.8pt]
    (0.01,0)
      .. controls (0.06,1.3) and (1.4,1.3) ..
      (2.5,0)
      .. controls (1.4,-1.3) and (0.06,-1.3) ..
      cycle;
\end{scope}

\end{tikzpicture}

    \caption{These are not manifolds with corners. The cone and the pyramid are not modeled on $\R^{n,k}$ near the vertex/apex. The teardrop is a t-manifold but the boundary is not embedded.}
    \label{fig:enter-label}
\end{figure}
\vspace{0.4cm}

For later reasons, we also define a notion of b-manifold.
\begin{defi}
Let $M$ be a smooth non compact manifold. We say that $M$ is a \textbf{b-manifold} if it can be compactified to a manifold with smooth closed boundary $\overline{M}$ i.e. there exists a smooth embedding $i : M \rightarrow \overline{M}$ such that $i(M) = \mathring{\overline{M}}$.
\end{defi}
\begin{rem}
\hfill
\begin{itemize}
\item Such manifolds are also called manifolds with ends: a non compact manifold $M$ of dimension $m$ is called a manifold with ends if it satisfies the following
        \begin{enumerate}
    \item There is a compact subset $K \subset M$ such that $E:=M \backslash K$ has a finite number of connected components $E_1, \ldots, E_n$ i.e. $E=\amalg_{i=1}^n E_i$.
    \item For each $i \in \{1,2,\cdots, n \}$, there is a connected $(m-1)$-dimensional compact manifold $\Sigma_i$ without boundary.
    \item There exist diffeomorphisms $\phi_i: \Sigma_i \times[1, \infty) \rightarrow \overline{E_i}$, where $\overline{E_i}$ is the closure of $E_i$ in $M$.
\end{enumerate}
\item The two definitions are equivalent due to the collar neighborhood theorem for manifolds with boundary. 
\item We choose the name b-manifold in accordance with the notions of b-geometry and b-calculus.
\end{itemize}
\end{rem}
Now, we suppose $M$ is any manifold with corners. Working within the paradigm of b-geometry, we can make use of notions previously defined in the literature.\\

The Lie algebra of \textbf{b-vector fields} on $M$ is the Lie subalgebra of vector fields on $M$ which are tangent to all the boundary faces 
$$
\mathcal{V}_b(M) := \left\{  
V \in \Gamma (M, T M) ; \hspace{0.2cm} V \hspace{0.2cm} \text{is tangent to each} \hspace{0.2cm}F \in \mathcal{M}(M) \right\}.
$$
Elements of $\mathcal{V}_b(M)$ are also sometimes called \textbf{b-derivatives}. The space of \textbf{b-differential operators of order $k$} denoted by $\mathrm{Diff}_b^k(M)$ consists of the linear maps on $C^\infty(M)$ given by a finite sum of up to $k$-fold products of elements of $\mathcal{V}_b(M)$
$$
\mathrm{Diff}_b^k(M) := \mathrm{span}_{0\leq j \leq k}\mathcal{V}_b(M)^j.
$$

In the simple case when $M$ is a manifold with boundary, if $x, y_1, y_2, \cdots y_n$ are coordinates near a component of the boundary and $x$ is a bdf for that component, then $\mathcal{V}_b(M)$ is locally spanned over $C^\infty (M)$ by the vector fields 
$$
x \partial_x , \hspace{0.1cm} \partial y_1, \partial y_2, \cdots , \partial y_n. 
$$

By the Serre-Swan theorem, there is an associated vector bundle on $M$, called the \textbf{b-tangent bundle} and denoted by \textbf{$^{b} T M$} such that $\mathcal{V}_b(M) \cong \Gamma (M, ^{b} T M)$. We denote its dual by $^{b}T^{*} M$. Therefore, when $M$ is a manifold with boundary, sections of $^{b}T^{*} M$ are locally spanned by 
$$
\frac{dx}{x} , \hspace{0.1cm} d y_1, d y_2, \cdots , d y_n. 
$$
We also define the \textbf{b-density} bundle $^{b} \Omega M := \Omega ( \hspace{0.01cm}^{b} T M)$. \\

When $M$ is a manifold with boundary and $\rho$ is a total boundary defining function, we define the following rescaled versions of the previous spaces
\begin{align*}
\mathcal{V}_c(M) & := \{ V \in \Gamma(M; TM); \hspace{0,2cm} V\in \frac{1}{\rho} \mathcal{V}_b(M) \}.\\ \mathcal{V}_{sc}(M) & := \{ V \in \Gamma(M; TM); \hspace{0,2cm} V\in \rho \mathcal{V}_b(M) \}.
\end{align*}
We denote the related bundles by $^{c} T M$ and $^{sc} T M$ respectively and we call these the \textbf{c-tangent bundle} and the \textbf{sc-tangent bundle} respectively. Similarly, we denote their duals by $^{c} T^{*} M$ and $^{sc} T^{*} M$ respectively.

\subsection{Function spaces on manifolds with corners}
In this part, we let $\overline{M}$ be a manifold with corners and $M := \overline{M} \setminus \partial \overline{M} $ its interior. We will define several function spaces on $M$ that are of interest to the analysis carried out in this work. For that, first, we give the following definition. 
\begin{defi}
    A \textbf{b-metric} is a Riemannian metric on $M$ given as the restriction of a metric on $^{b} T \overline{M}$ i.e. a smooth positive definite section of $C^\infty (\overline{M}, \hspace{0.01cm} ^{b}T \overline{M}).$\\
    In particular, if $\overline{M}$ is a manifold with boundary, a b-metric $g_b$ on $M$ is written, locally near a boundary component, as 
$$
g_b = a_{00} \left( \frac{dx^2}{x^2} \right) + \sum_{i} a_{0i} \frac{dx}{x} dy_i
 + \sum_{i,j} a_{i,j} dy_i dy_j.$$
\end{defi}

Now, let $g_b$ be a b-metric on $M$ and $E \rightarrow \overline{M}$ a Euclidean vector bundle with metric $g_E$. Using the Levi-Cevita connection for $g_b$ and a compatible connection $\nabla^E$ for $E$, we define the space of \textbf{b-conormal} sections of $E$ 
$$
C_b^k(M ; E) := \left\{ \sigma \in C^k(M ; E) ; \hspace{0.2cm} \sup_{p \in M} \left| \nabla^{j} \sigma(p) \right|_{g_b, g_E} < \infty \hspace{0.2cm} \forall j \in \{0,1, \cdots, k \} \right\},
$$
with the norm 
$$
|| \sigma ||_{b,k} := \sum_{j=0}^k \sup_{p \in M} \left| \nabla^{j} \sigma(p) \right|_{g_b, g_E},
$$
making it a Banach space. This space is independent of the choice of the b-metric on $M$ (but the norm is not). For example, $C_b^k(M )$ can be, equivalently, defined as the $k$-differentiable functions on $M$ that are bounded and continuous together with their b-derivatives of order less or equal than $k$.\\

Similarly, for $\alpha \in(0,1]$, we can also consider the \textbf{b-Hölder space} $C_b^{k, \alpha}(M ; E)$ of sections $\sigma \in$ $C_b^k(M; E)$ such that
$$
\left[\nabla^k f\right]_{b, \alpha}:=\sup \left\{\left.\frac{\left|P_\gamma\left(\nabla^k f(\gamma(0))\right)-\nabla^k f(\gamma(1))\right|}{\ell(\gamma)^\alpha} \right.; \hspace{0.1cm} \gamma \in C^{\infty}([0,1] ; M ), \gamma(0) \neq \gamma(1) \right\}<\infty
$$
where $P_\gamma: T^{*}M^{\otimes k} \otimes E |_{\gamma(0)} \rightarrow T^{*}M^{\otimes k} \otimes E|_{\gamma(1)}$ is the parallel transport along $\gamma$ and $\ell(\gamma)$ is the length of $\gamma$ with respect to the metric $g_b$. This is also a Banach space with norm given by
$$
\|f\|_{b, k, \alpha}:=\|f\|_{b, k}+\left[\nabla^k f\right]_{b, \alpha}.
$$
Finally, using a volume density $\nu_b$ associated to $g_b$, we can define $L_b^2(M ; E) := L^2 (M; E; \nu_b)$. Therefore, the \textbf{b-Sobolev space} $H_b^k(M; E)$ is given by
$$
H_b^k(M; E) := \left\{ f \in L^2(M;E) ; \hspace{0.2cm} \nabla^{j}f \in L_b^2(M ; (\hspace{0.01cm} ^{b}T^{*} \overline{M})^j \otimes E ) \hspace{0.2cm} \forall j \in \{0,1, \cdots, k \} \right\},
$$
with its associated $L^2$ Sobolev norm.
\subsubsection{Polyhomogeneous functions} \label{phgfnc}
Now, we define another space of functions known as \textit{polyhomogeneous functions}. These are b-conormal functions with a certain asymptotic expansion near the boundary. Such expansions generalize smoothness up to the boundary and appear as the boundary behavior of solutions to elliptic equations on certain singular and non-compact spaces (See Proposition \ref{phgconlap} and \ref{aclap}).  For references, see \cite{melrose_atiyah-patodi-singer_1993, melrose_calculus_1992, melrose_pseudodifferential_1990, grieser_singular_2001, grieser2017scalesblowupquasimodeconstructions}. \\

Before giving precise definitions of polyhomogeneous functions on a general manifold with corners $\overline{M}$, we describe roughly the simple setting where $\overline{M}=[0, \infty)$.

Consider functions on $(0, \infty)$ of the form 
$$
f(x) = x^z (\log x)^k, \hspace{0.1cm} (z,k) \in \C \times \N,
$$
and consider the b-derivative $x \partial_x$ acting on $f$. We get 
$$
x \partial_x f = x \partial_x (x^z (\log x)^k) = z x^z (\log x)^k + k x^z (\log x)^{k-1}. 
$$
Therefore, the space spanned by $x^z (\log x)^j$ for $j \in \{ 0,1, \cdots, k \}$ is invariant under the action of b-vector fields on $[0, \infty)$ in general. The same remains true if we allow $z$ to vary over a finite set. To be able to consider infinite sums, we need to ensure additional conditions on the index set over which $(z,k)$ varies. For that reason, notice that for  $(z_1, k_1), (z_2,k_2) \in \C \times \N$, we have 
$$
x^{z_1} (\log x)^{k_1} \in o\left( x^{z_2} (\log x)^{k_2}\right), \hspace{0.1cm} x \rightarrow 0 \iff \mathrm{Re} (z_1) > \mathrm{Re}(z_2) \hspace{0.2cm} \text{or} \hspace{0.2cm} \mathrm{Re} (z_1) = \mathrm{Re}(z_2), \hspace{0.15cm} k_1 < k_2.
$$
Therefore, to make sense of infinite sums over an index set $E$, we need to ensure that $\{ (z,k) \in E; \hspace{0.1cm}  \mathrm{Re} (z) \leq s \}$ is finite for all $s \in \R$. With that understood, a function on $(0, \infty)$ will be said to be \emph{polyhomogeneous} if it admits, near $x =0$, an asymptotic expansion (in an appropriate sense) over such an index set $E$ with terms of the form $x^z (\log x)^k,\hspace{0.1cm} (z,k) \in E$. From the comments above, the space of polyhomogeneous functions is invariant under the action of b-differential operators. But, more importantly, we also have the following converse \cite[Proposition 5.61]{melrose_atiyah-patodi-singer_1993}: if a function $f$ on $(0, \infty)$ satisfies a mild decay condition near $x=0$ ($f \in x^{\alpha} H_b^m((0,\infty))$) and $P f$ is polyhomogeneous for an elliptic b-operator $P$, then $f$ is polyhomgeneous.\\

An important fact to keep in mind is that, by Taylor's expansions, functions which are smooth up to the boundary are, in particular, polyhomogeneous with index set $E = \N \times \{ 0\}$ (we do not need to assume analyticity because we allow quickly vanishing error terms). Therefore, if $P f$ is smooth up to the boundary, for an elliptic b-operator $P$, then $f$ will not necessarily be smooth up to the boundary, but it will be polyhomogeneous with possibly logarithmic terms and complex powers of $x$ appearing. In this sense, polyhomogeneity is a generalization for smoothness up to the boundary which allows to state the elliptic regularity result mentioned above.\\

With these comments in mind, we start our definitions in the more general setting. 
\begin{defi} \label{ind}
 An \textbf{index set} $F$ is a discrete subset of $\mathbb{C} \times \mathbb{N}_0$ such that 
 \begin{enumerate}
     \item  $\left(z_j, k_j\right) \in F, \hspace{0.1cm} \left|\left(z_j, k_j\right)\right| \rightarrow \infty \Longrightarrow \operatorname{Re} z_j \rightarrow \infty$,
     \item $(z, k) \in F \Longrightarrow(z, p) \in F \hspace{0.1cm}  \forall p=0, \ldots, k$, 
     \item  $(z, k) \in F \Longrightarrow(z+p, k) \in F \hspace{0.1cm} \forall p \in \mathbb{N}$.
 \end{enumerate}

The index set $F$ is called 
\begin{itemize}
 \item \textbf{real} if
$F \subset \R \times \N_{0}$.
\item   \textbf{positive} if
$F \subset \R_{>0} \times \N_{0}$, we write: $F > 0$. 
 \item \textbf{nonnegative} if $F \subset \R_{\geq 0} \times \N_{0}$ and 
$$
(0, k) \in F \Longrightarrow k=0 .
$$
We write: $F \geq 0$.
\item \textbf{trivial} if $F = \N \times \{ 0\}$. We write $F=0$.
\item In general, if $F \subset \R \times \N$, we define $\inf F$ to be the smallest element of $F$ with respect to the order relation 
$$
(z_1,k_1) < (z_2, k_2) \iff z_1 < z_2 \hspace{0.1cm} \text{or} \hspace{0.1cm} z_1 = z_2 \hspace{0.1cm} \text{and} \hspace{0.1cm} k_1 > k_2.
$$
\item We write $F \geq m$ if $F -m := \left\{ (z-m,k) \in \C \times \N_0 \hspace{0.1cm} ; \hspace{0.1cm} (z,k) \in F   \right\}$ is non-negative.
\end{itemize}
\end{defi}
\begin{defi}
  Let $M$ be a manifold with corners. An \textbf{index family} $\mathcal{F}$ for $M$ is an assignment of an index set $\mathcal{F}(H)$ to each boundary hypersurface $H \in \mathcal{M}_1(M)$. The index family is said to be \textbf{trivial} if the index set $\mathcal{F}(H)$ is trivial for all $H \in \mathcal{M}_1(M)$.
\end{defi}
Now, let $\overline{M}$ be a manifold with corners and $M = \overline{M} \setminus \partial \overline{M}$ its interior. Denote $k = \# \mathcal{M}_1(\overline{M})$, the number of boundary hypersurfaces of $\overline{M}$ and $I := \{1,2, \cdots k \}$. Accordingly, we index the boundary hypersurfaces by $H_i$, $i \in I$. If $\mathcal{F}$ is an index family for $\overline{M}$, we define $\mathcal{F}_i$ to be the index family for $H_i$ given by 
$$
\mathcal{F}_i(H_j \cap H_i) := \mathcal{F}(H_j), \hspace{0.1cm}  \text{for all} \hspace{0.1cm} j \neq i \hspace{0.1cm} \text{such that} \hspace{0.1cm} H_i \cap H_j \neq \emptyset.
$$
We define polyhomogeneous functions as follows.
\begin{defi}
Let $\mathcal{F}$ be an index family for $\overline{M}$. We define the space of \textbf{polyhomogeneous functions with index family $\mathcal{F}$} denoted by $\mathcal{A}_{\text{phg}}^{\mathcal{F}}(\overline{M})$ inductively as follows 
    \begin{itemize}
        \item If $\mathcal{F}$ is trivial, then $\mathcal{A}_{\text{phg}}^{\mathcal{F}}(\overline{M}) := C^\infty (\overline{M})$.
\item If not, then $f \in \mathcal{A}_{\text{phg}}^{\mathcal{F}}(\overline{M})$ if $f$ is smooth on the interior $M$ and for every boundary hypersurface $H_i$, $i \in I$ and boundary defining function $\rho_{H_i}$ for $H_i$ we have
$$
 f \sim \sum_{(z, k) \in \mathcal{F}(H_i)} a_{(z, k)} \rho_{H_i}^z(\log \rho_{H_i})^k, \hspace{0.1cm} a_{z, k} \in \mathcal{A}_{\text{phg}}^{\mathcal{F}_i}(H_i),
$$
here, $\sim$ means here that for all $N \in \mathbb{N}$,
$$
 f-\sum_{\substack{(z, k) \in \mathcal{F}(H_i) \\ \operatorname{Re} z \leq N}} a_{(z, k)} \rho_{H_i}^z(\log \rho_{H_i})^k \in \rho_{H_i}^N C_b^{\infty}(M).
$$
        
    \end{itemize}
\end{defi}
\begin{rem} The following remarks are true for any manifold with corners, but, for convenience, we suppose that $\overline{M}$ is a manifold with boundary and that its boundary $\partial \overline{M}$ has only one connected component with bdf $\rho$ and associated index set $F$. From Definition \ref{ind}
\begin{itemize}
\item Condition \textbf{1)} implies that, for a polyhomogeneous function, all but a finite number of terms in the expansion vanish at any order at $\rho=0$.
\item Condition \textbf{2)} is required to insure the space of polyhomogeneous functions with a given index set is invariant by a change of the boundary defining function $\rho$.
\item Condition \textbf{3)} makes $\mathcal{A}_{\text {phg }}^F(\overline{M})$ a $C^{\infty}(\overline{M})$-module. Therefore, for a vector bundle $E$ on $\overline{M}$, we can define the space of polyhomogeneous sections of $E$ with index set $F$ by: 
$$\mathcal{A}_{\text {phg }}^F(\overline{M} ; E) := \mathcal{A}_{\text {phg }}^F(\overline{M}) \otimes_{C^{\infty}(\overline{M})} C^{\infty}(\overline M ; E). $$
\item A polyhomogeneous function with nonnegative index set is, in particular, b-conormal. More generally, if $s = \inf F$ then $\mathcal{A}_{\text {phg }}^F(\overline{M}) \subset x^{s - \varepsilon} C_{b}^{\infty}(\overline M)$ for every $\varepsilon>0$. 
\item Since, for $F=0$, $\mathcal{A}_{\text {phg }}^F(\overline{M})=C^{\infty}(\overline M) $, polyhomogeneous functions are seen as a generalization of functions smooth up to the boundary. 
\item For the empty index set $F= \emptyset$, we have
 $\mathcal{A}_{\text {phg }}^F(\overline{M})=\dot{C}^{\infty}(\overline M) $, the set of functions vanishing, with all their derivatives, at the boundary. These are also refered to as Schwartz functions.
 \item Sometimes, we will use the notation $\mathcal{A}_{phg}$ for polyhomogeneous functions without prescribing the index set. In the case of a manifold with boundary, we also use $\mathcal{A}_{phg \geq0}$ and $\mathcal{A}_{phg > 0}$ for polyhomogeneous functions with nonnegative and positive index set respectively.
\end{itemize}
\end{rem}

In \cite[Proposition 5.27]{melrose_atiyah-patodi-singer_1993}, the author gives the following correspondence between polyhomogeneous functions and meromorphic functions under the Mellin transform  
\begin{pr}
Let $Y$ be a closed compact manifold  and let $\mathcal{E}=(E, \emptyset)$ be the pair of index sets for $[-1,1] \times Y$ which assigns index set $E$ to $\{-1\} \times Y$ and the empty index set $\emptyset$ to $\{1\} \times Y$, then the Mellin transform
$$
u_M(\lambda, y)=\int_0^{\infty} x^{-i \lambda} u\left(\frac{x-1}{x+1}, y\right) \frac{d x}{x}
$$
gives an isomorphism from $\mathcal{A}_{\mathrm{phg}}^{\mathcal{E}}([-1,1] \times Y)$ to the space of meromorphic functions with values in $C^{\infty}(Y)$ having poles of order $k$ only at points $\lambda=-i z \in \mathbb{C}$ such that $(z, k-1) \in E$ and satisfying for each large $N$: 

 $$\hspace{0.1cm}\left\|u_M(\lambda, \cdot)\right\|_N \leq C_N(1+|\lambda|)^{-N} \hspace{0.1cm} \text{in}  \hspace{0.1cm} |\operatorname{Im} \lambda| \leq N,|\operatorname{Re} \lambda| \geq C_N$$
 where $\|\cdot\|_N$ is a norm on $C^N(Y)$.
 \end{pr}
\subsection{b-fibrations and the push-forward theorem}
In the following, we define the notion of b-fibrations on manifolds with corners that will be of importance to the general setting we consider in Section \ref{set}. 
\begin{defi}
\label{def:b-map}
A smooth map 
\[
f: M \,\longrightarrow\, N
\]
between manifolds with corners is called a \textbf{b-map} if, for each boundary hypersurface $H_j \in \mathcal{M}_1(N)$, there exist nonnegative integers $e_f(G_i, H_j)\ge 0$ and a nowhere-vanishing smooth function $u_j$ on $M$ such that
\[
\rho_{H_j}\bigl(f(p)\bigr) 
\;=\; 
u_j(p)\,\prod_{G_i\in \mathcal{M}_1(M)} \bigl(\rho_{G_i}(p)\bigr)^{e_f(G_i, H_j)}
\hspace{0.1cm}
\text{for all } p \in M,
\]
where $\rho_{H_j}$ and $\rho_{G_i}$ are boundary defining functions for $H_j \in \mathcal{M}_1(N)$ and $G_i\in \mathcal{M}_1(M)$ respectively. In particular, the preimage of any boundary hypersurface of $N$ is a union of boundary hypersurfaces of $M$.
\end{defi}
\begin{rem}
   If $f : M \rightarrow N$ is a b-map, then there is a naturally induced bundle map $^{b}f_{*} : \hspace{0.01cm} ^{b}TM \rightarrow \hspace{0.01cm} ^{b}TN$  which agrees with the usual differential of $f$ over the interior of $M$.
\end{rem}
\begin{defi}
\label{def:b-fibration}
Let  
\[
f: M \,\longrightarrow\, N
\]
be a b-map between manifolds with corners. We say that $f$ is a \textbf{b-submersion} if $^{b}f_* $ is surjective at any point. We say that it's \textbf{b-normal} if the natural restriction $\restr{^{b}f_*}{p} : \hspace{0.01cm} ^{b} N_{p} M \rightarrow \hspace{0.01cm} ^{b} N_{f(p)} N$ is surjective for all point $p \in M$. $f$ is called a \textbf{b-fibration} if it is both a b-submersion and b-normal.
\end{defi}
 If $f : X \rightarrow Y$ is a b-fibration, the map which sends each b-density to its integral over the fibers of $f$: 
 $$
 \mu \in \Gamma(X; \hspace{0.01cm} ^{b}\Omega X) \longrightarrow f_{*}\mu \in \Gamma(Y; \hspace{0.01cm} ^{b}\Omega Y) : q \in Y \rightarrow f_{*}\mu (q) =\int_{f^{-1}(\{ q \})} \mu 
 $$
 is called the \emph{push-forward map} and is denoted by $f_{*}$. Now, we mention an important result due to Melrose \cite{melrose_calculus_1992} which describes the behavior of polyhomogeneous b-densities under the push-forward map. 
 \begin{defi}
     Let $f : X \rightarrow Y$ be a b-fibration and $K$ an index family for $X$. We define $f_{\#} K$ to be the index family for $Y$ given by 
     $$
f_{\#} K(H):=\overline{\bigcup}_{\substack{G \in M_1(X) \\ e_f(G, H) \neq 0}}\left\{\left.\left(\frac{z}{e_f(G, H)}, p\right) \right\rvert\,(z, p) \in K(G)\right\}
$$
for all $H \in \mathcal{M}_1(Y)$, where the extended union used above is defined by 
$$
K \overline{\cup} I=K \cup I \cup\left\{\left(z, p^{\prime}+p^{\prime \prime}+1\right) \mid\left(z, p^{\prime}\right) \in K,\left(z, p^{\prime \prime}\right) \in I\right\},
$$
for any two index sets $K$ and $I$.\\
We also define the null-set of $f$ by 
$$
\mathrm{null}(e_f) = \{ G \in \mathcal{M}_1(X) \hspace{0.2cm} | \hspace{0.1cm} e_f(G,H) =0 \hspace{0.1cm} \forall H \in \mathcal{M}_1(Y)\}.
$$
 \end{defi}
\begin{pr}[Melrose push-forward theorem] \label{fwd}
If $f : X \rightarrow Y$ is a b-fibration between manifolds with corners and $K$ is an index family for $X$ such that 
$$
G \in \mathrm{null}(e_f) \implies \mathrm{Re} (K(G)) >0,
$$
then the push-forward map $f_{*}$ induces a map 
$$
f_{*} : \mathcal{A}_{phg}^K(X; \hspace{0.01cm} ^{b} \Omega X) \longrightarrow \mathcal{A}_{phg}^{f_{\#} K}(Y; \hspace{0.01cm} ^{b} \Omega Y). 
$$
\end{pr}
\subsection{Blow-up in the sense of Melrose} \label{blow}
Next, we define blow-ups in the sense of Melrose. These are blow-ups in the real sense which correspond to working in polar coordinates around the considered submanifold. They can be used to resolve functions, spaces or vector fields. The resulting geometry is that of manifolds with corners. For references, we mention the following main works \cite{melrose_calculus_1992, melrose_atiyah-patodi-singer_1993, melrose_pseudodifferential_1990}.\\

A blow-up is specified by two pieces of information: a space and a blow-down map. First, we start by defining the non-weighted blow-up.  The \emph{blow-up of the origin in \(\mathbb{R}^{n,k} = [0, \infty)^k \times \R^{n-k}\)} denoted by $[\mathbb{R}^{n,k} ; \{0\}]$ is given by the space 
$$
[\mathbb{R}^{n,k} ; \{0\}] := S^{n-1,k} \times [0, \infty)
$$
as a manifold with corners, where $S^{n-1,k} := S^{n-1} \cap \R^{n,k}$, together with a blow-down map 
\begin{align*}
\beta : [\R^{n,k};0] \cong S^{n-1,k} \times \R_{\geq 0} &\rightarrow \R^{n,k} \\ (\omega,r) & \rightarrow r\cdot \omega.
\end{align*}

In other words, when restricted, $\beta$ induces a diffeomorphism $\beta : [\R^{n,k};0] \setminus S^{n-1,k} \times \{ 0\} \rightarrow \R^{n,k} \setminus \{ 0\}$
and the origin is replaced by \(S^{n-1,k}\), the space of possible directions in which one can approach it in $\R^{n,k}$.\\

This blow-up can be equivalently defined in the following way:
$$
[\R^{n,k};0] : = \left( [0, \infty) \times \R^{n,k} \right) \setminus \left( [0, \infty) \times \{ 0\}\right) / \R_{>0},
$$
with respect to the $\R_{>0}$-action given by 
$$
t \cdot (x_0, x) = \left( t^{-1} x_0, t x_1, t x_2, \cdots, t x_n \right),
$$
for $t \in \R_{>0}$ and $(x_0, x) \in [0, \infty) \times \R^{n,k}$. The corresponding blow-down map is given by 
\begin{align*}
\beta : [\R^{n,k};0] &\rightarrow \R^{n,k} \\ [x_0 : x] & \rightarrow (x_0 x_1, x_0 x_2, \cdots, x_0 x_n ).
\end{align*}

In our applications, the $\R_{>0}$-action we have is not exactly the same as the action described above, but rather, there are additional weights. For that reason, we also need to introduce \emph{weighted blow-ups} following the definition introduced in \cite{conlon2023warped}. 

\begin{defi} \label{weighted blow-up}
With respect to a choice of weight $w \in\left(\mathbb{R}_{>0}\right)^n$, the \textbf{weighted blow-up} of $\{0\}$ in $\mathbb{R}^{n,k}$ is the quotient
$$
\left[\mathbb{R}^{n,k} ;\{0\}\right]_w:=\left(\left([0, \infty) \times \mathbb{R}^{n,k}\right) \backslash([0, \infty) \times\{0\})\right) / \mathbb{R}_{>0}
$$
with respect to the $\mathbb{R}_{>0}$-action given by
$$
t \cdot\left(x_0, x\right)=\left(t^{-1} x_0, t^{w_1} x_1, \ldots, t^{w_n} x_n\right)
$$
for $t \in \mathbb{R}_{>0}$and $\left(x_0, x\right) \in[0, \infty) \times \mathbb{R}^{n,k}$. The corresponding blow-down map $\beta:\left[\mathbb{R}^{n,k} ;\{0\}\right]_w \rightarrow \mathbb{R}^{n,k}$ is given by
$$
\beta\left(\left[x_0: x\right]\right)=\left(x_0^{w_1} x_1, \ldots, x_0^{w_n} x_n\right)
$$
where $\left[x_0: x\right]_w$ denotes the class in $\left[\mathbb{R}^{n,k} ;\{0\}\right]_w$ corresponding to $\left(x_0, x\right) \in\left([0, \infty) \times \mathbb{R}^{n,k}\right) \backslash([0, \infty) \times\{0\})$.\\

The weighted blow-up $\left[\mathbb{R}^{n,k} ;\{0\}\right]_w$ is naturally a manifold with corners diffeomorphic to $\mathbb{S}^{n-1,k} \times[0, \infty)$ via the map
$$
\begin{array}{rllc}
F: \hspace{0.1cm} \mathbb{S}^{n-1,k} \times[0, \infty) & \rightarrow & {\left[\mathbb{R}^{n,k} ;\{0\}\right]_w} \\
(\omega, t) & \mapsto & {[t: \omega]_w}
\end{array}
$$
\end{defi}
\begin{rem} \label{projective coordinates}
    The weighted blow-up $\left[\mathbb{R}^{n,k} ; \{0\}\right]_w$ admits an important set of coordinate systems known as \textbf{projective coordinates}. Consider a coordinate system $(x_1, x_2, \cdots, x_n)$ on $\R^{n,k}$. For $x_1 \neq 0$, we can consider the following coordinates on $\left[\mathbb{R}^{n,k} ; \{0\}\right]_w$ given by $$\left[|x_1|^{\frac{1}{w_1}} : \pm 1 : |x_1|^{- \frac{w_2}{w_1}} x_2 : |x_1|^{- \frac{w_3}{w_1}} x_3 : \cdots : |x_1|^{- \frac{w_n}{w_1}} x_n \right].$$
    Therefore, if $\beta$ is the blow-down map described above, then $$\beta \left( \left[|x_1|^{\frac{1}{w_1}} : \pm 1 : |x_1|^{- \frac{w_2}{w_1}} x_2 : |x_1|^{- \frac{w_3}{w_1}} x_3 : \cdots : |x_1|^{- \frac{w_n}{w_1}} x_n \right] \right) = (x_1, x_2, \cdots, x_n)$$ 
    Obviously, one can consider similar coordinates for every $x_i \neq 0$. This, in particular, implies that, given a function $f$ on $\R^{n,k}$, $\beta^*f:= f \circ \beta$ is smooth (resp. polyhomogeneous) on $\left[\mathbb{R}^{n,k} ; \{0\}\right]_w$ if and only if $f$ is smooth (resp. polyhomogeneous) as a function of $|x_i|^{- \frac{w_1}{w_i}} x_1, |x_i|^{- \frac{w_2}{w_i}} x_2, \cdots, x_i^{\frac{1}{w_i}}, \cdots, |x_i|^{- \frac{w_n}{w_i}} x_n$ for every $i =1,2, \cdots, n$.
\end{rem}

\section{Conifolds} \label{conifolds}
\begin{defi} \label{conical}
The \textbf{Riemannian cone} over a given closed connected Riemannian manifold $\left(\Sigma, g_\Sigma\right)$ is the Riemannian manifold $\left(C, g_C\right)$ where $C=\mathbb{R}_{>0} \times \Sigma$ and $g_C=d x^2+x^2 g_\Sigma$, with $x: C \rightarrow \mathbb{R}_{>0}$ the projection onto the first factor. We often write $C=C(\Sigma)$ and call $\Sigma$ the link of $C$.
\end{defi}
\begin{rem}
    We allow ourselves to include or exclude the vertex of the cone depending on the context. By the \emph{vertex} of the cone, we mean a point denoted by $o$ attached to the cone at $\{x=0\}$. We also sometimes consider the cone to be the b-manifold $\R_{\geq 0} \times \Sigma$. This will always be clear from the situation.
\end{rem}

\begin{defi} \label{def : conifolds}
Let $M$ be a b-manifold with compactification $\overline{M}$. Let $g$ be a Riemannian metric on $M$. Choose a connected component of $\partial \overline{M}$ that we denote by $\Sigma$.
\begin{itemize}
\item 
We say that $g$ is a \textbf{product cone metric near $\Sigma$} if, in a collar neighborhood $U$ of $\Sigma$, it is isometric to a neighborhood of the vertex $\{ o\}$ in a Riemannian cone over $\Sigma$, i.e. there is a boundary defining function for $\Sigma$ denoted by $x$ and $\restr{g}{U} = dx^2 + x^2 g_{\Sigma}$.\\ 
We say that $g$ is a \textbf{conical metric near $\Sigma$} if there exists a product cone metric $g_p$ near $\Sigma$  and $\nu >0$ such that $g - g_p \in x^{\nu} C^{\infty}_b(M; \hspace{0.01cm} ^{c} T^{*} \overline{M} \otimes \hspace{0.01cm} ^{c} T^{*} \overline{M})$.

\item We say that $g$ is a \textbf{product scattering metric near $\Sigma$} if, in a collar neighborhood $U$ of $\Sigma$, $g$ is isometric to a neighborhood of infinity in a Riemannian cone over $\Sigma$ i.e. there is a boundary defining function for $\Sigma$ denoted by $x$ such that $\restr{g}{U} = \frac{dx^2}{x^4} + \frac{g_{\Sigma}}{x^2}$. This is equivalent to writing $\restr{g}{U} = dr^2 + r^2 g_{\Sigma}$ for $r:= \frac{1}{x}$.\\
We say that $g$ is an \textbf{asymptotically conical metric near $\Sigma$} if there exists a warped product scattering metric $ g_p$ near $\Sigma$ and $\nu >0$ such that $g - g_p \in x^{\nu} C^{\infty}_b(M; \hspace{0.01cm} ^{sc} T^{*} \overline{M} \otimes \hspace{0.01cm} ^{sc} T^{*} \overline{M}) $. 
\end{itemize}
In either of the above situations, we say that $(M,g)$ is a \textbf{conifold near $\Sigma$ modeled on the cone $\left(C(\Sigma), g_{C(\Sigma)}\right)$} and we call $\nu$ the convergence rate.\\
We say that $g$ is a \textbf{polyhomogeneous} conical (resp. asymptotically conical) metric if it is conical (resp. asymptotically conical) in the above sense and it is polyhomogeneous as a section of $^{c} T^{*} \overline{M} \otimes \hspace{0.01cm} ^{c} T^{*} \overline{M}$ (resp. $^{sc} T^{*} \overline{M} \otimes \hspace{0.01cm} ^{sc} T^{*} \overline{M}$).  
\end{defi} 
\subsection{Analysis on conifolds} \label{anacon}
There has been extensive study of analysis on conical and asymptotically conical manifolds, see for example \cite{cheeger_spectral_1983, joyce_special_2004, pacini_desingularizing_2013, lockhart_elliptic_1985, marshall_deformations_nodate, hein_calabi-yau_2017, Rafe_elliptic}. In the following, we mention some mapping properties of the Laplacian on these manifolds.
\subsubsection*{Conical Laplacian:}
Let $M_1$ be a b-manifold of dimension $n \geq 3$ and compactification $\overline{M_1}$. For convenience, we suppose that $ \partial \overline{M_1}$ has only one connected component $\Sigma$ and we let $x_1$ be a bdf for $\Sigma$. Let $g_c$ be a polyhomogeneous conical metric on $M$ modeled on a cone $(C, g_0)$ through an isomorphism $\psi : V \subset C \rightarrow U$ where $U$ is a neighborhood of $\Sigma$. Let $\Delta_c$ be the Laplacian operator associated to $g_c$. We have the following mapping result \cite{hein_calabi-yau_2017}: 
\begin{pr} \label{conlap} Consider $\alpha \in (0,1)$ and an integer $k > \frac{3n -1 }{2}$. Let $\mathcal{P}$ be the set of indicial roots of $\Delta_{g_0}$ i.e.
$$
\mathcal{P} = \left\{ - \frac{n-2}{2} \pm \sqrt{\frac{(n-2)^2}{4} + \lambda }; \hspace{0.1cm} \lambda \in \mathrm{Spec}(\Delta_L) \right\},
$$
where $\Delta_L$ is the Laplacian on the link $L$ of the cone $(C, g_0)$. 
For a weight $\nu \in (0, 2 ) \setminus \mathcal{P}$, consider the following spaces:  
$$
\mathcal{U}=\left\{u \in x_1^{\nu} C_{b}^{k+2, \alpha}\left(M_1\right) \oplus \hat{\chi}\left( \mathcal{H}_{[0, \nu)}\right) \circ \psi^{-1}; \int_{M_1} u =0\right\} 
$$ $$\mathcal{F}=\left\{f \in x_1^{\nu-2}C_{b}^{k, \alpha}\left(M_1\right) ; \int_{M_1} f =0\right\} 
$$
where $ \mathcal{H}_{[0, \nu)}$ is the vector space spanned by homogeneous harmonic functions on the cone $(C, g_0)$ with growth rate in $[0, \nu) $, $\hat{\chi}$ is a cut-off function equal to $1$ near boundary and $0$ outside a collar neighborhood. 
Then, the conical Laplacian $\Delta_c$ induces an isomorphism: 
$$
\Delta_{c} : \mathcal{U} \rightarrow \mathcal{F}.
$$ 
For a weight $2-n < \nu < 0 $, we also have an isomorphism but without the additional terms coming from $ \mathcal{H}_{[0, \nu)}$.  
\end{pr}
The following offers finer description of the asymptotic behaviour of solutions.
\begin{pr} \label{phgconlap}
Let $g_{\mathrm{c}}$ be a polyhomogeneous conical metric. For any weight $\nu \in \R$ (even indicial roots), if $u \in x_1^{\nu}   H_b^{k}(M_1)$ and $\Delta_c u$ is polyhomogeneous, then $u$ is polyhomogeneous. Moreover, if $u \in x_1^{\alpha+2}   C_b^{\infty}(M_1)$ satisfies
$$
\Delta_{\mathrm{c}} u=f_1+f_2 \hspace{0.1cm} \text { with } f_1 \in x_1^{\alpha+\beta} C_b^{\infty}(M_1), \hspace{0.1cm} f_2 \in \mathcal{A}_{\mathrm{phg}}^{G}\left(\overline{M_1}\right)
$$
for some $\beta>0$ and some index set $G$ such that $\inf G > \alpha$, then $u=u_1+u_2$ with $u_1 \in \bigcap_{\delta>0} x_1^{\alpha+\beta+2-\delta} C_b^{\infty}(M_1)$ and $u_2 \in \mathcal{A}_{\mathrm{phg}}^{F}(\overline{M_1})$ such that $\inf F  \geq 2+ \alpha$.
\end{pr}
\begin{proof}
    The operator $P := x_1^2 \Delta_c$ is an elliptic b-differential operator with polyhomogeneous coefficients, therefore we get the first claim by \cite[Corollary 3.7]{conlon_moduli_2015} which generalizes \cite[Proposition 5.61]{melrose_atiyah-patodi-singer_1993} and the second claim by \cite[Corollary 3.8]{conlon_moduli_2015}. 
\end{proof}
\subsubsection*{Asymptotically conical Laplacian:}
Similarly, let $M_2$ be a b-manifold of dimension $m \geq 3$ and compactification $\overline{M_2}$. For convenience, we suppose that $ \partial \overline{M_2}$ has only one connected component $\Sigma'$ and we let $x_2$ be a bdf for $\Sigma'$. Let $g_{AC}$ be a polyhomogeneous asymptotically conical metric on $M_2$ modeled on a cone $(C, g_0)$. Let $\Delta_{AC}$ be the Laplacian operator associated to $g_{AC}$. We have the following mapping result \cite{lockhart_elliptic_1985, marshall_deformations_nodate, conlon_moduli_2015, melrose_atiyah-patodi-singer_1993}:
\begin{pr} \label{aclap}
Let $\alpha \in (0,1)$ and $k \in \N$. The asymptotically conical Laplacian $\Delta_{AC}$ is an isomorphism when acting on the following spaces:
    $$\Delta_{AC}: x_2^{ \beta } C_{b}^{k+2, \alpha}(M_2) \rightarrow x_2^{\beta + 2 } C_{b}^{k, \alpha}(M_2)$$
    provided the weight $\beta \in (0,m-2)$. Moreover, if $\Delta_{AC} u$ is polyhomogeneous, then $u$ is polyhomogeneous. 
\end{pr}

\subsection{Calabi-Yau conifolds}
\begin{defi}
    A \textbf{Calabi-Yau manifold} is a complex Kähler manifold $X$ (compact or non compact) with vanishing first Chern class, i.e. $c_1(X) = 0$ in $H^2(X; \R)$. We say that $X$ is Calabi-Yau \textbf{in the strong sense} if the canonical bundle of $X$ is trivial, i.e. $K_X \cong \mathcal{O}_X$.
\end{defi}
\begin{rem}
    Since, $c_1(K_X) = - c_1(X)$, being Calabi-Yau in the strong sense implies being Calabi-Yau. Conversely, if $X$ is a compact Calabi-Yau manifold, then a positive power of the canonical bundle of $X$ is trivial, i.e. $K_X^{\otimes m} \cong \mathcal{O}_X$.
\end{rem}
By Yau's solution of Calabi's conjecture \cite{yau_ricci_1978}, every closed compact Calabi-Yau manifold has a unique Ricci-flat Kähler metric in every Kähler cohomology class. In general, we will say Ricci-flat Kähler metrics or Calabi-Yau metrics interchangeably.  If $X$ is compact Calabi-Yau and $\omega$ is any choice of a Kähler form, finding a Ricci-flat Kähler metric is equivalent to solving the complex Monge-Ampère equation:
$$
\frac{(\omega+i \partial \overline{\partial}u)^{n}}{\omega^n} = e^v
$$
where $v$ satisfies $\mathrm{Ric}(\omega)= i \partial \overline{\partial}v$ and $\int_X e^v \omega^n = \int_X \omega^n$.\\

Solving this PDE on open Calabi-Yau manifolds is generally more complicated because additional conditions on the geometry near infinity are needed. The following is a general existence theorem for projective varieties with canonical singularities \cite{eyssidieux_singular_2009}.
\begin{tm}[Eyssidieux-Guedj-Zeriahi] \label{EGZ}
 Suppose $X_0$ is the regular part of a complex projective variety $X$ with trivial canonical bundle $K_X \cong \mathcal{O}_X$ and with only canonical singularities. Then for every choice of an ample line bundle $L$ on $X$, there is a unique Kähler current $\omega_{CY} \in 2 \pi c_1(L)$  that has bounded local potentials and restricts to a smooth Ricci-flat Kähler metric on $X_0$. We refer to $X$ verifying the above conditions as a \textbf{projective Calabi-Yau variety} and to $\omega_{CY}$ as a \textbf{singular Calabi-Yau metric}.  
\end{tm}
\begin{rem}
    The above result establishes the existence and uniqueness of singular Ricci-flat Kähler metrics on Calabi-Yau varieties but it does not specify the behavior of such metrics near the singularities.
\end{rem}
We also give the following definition needed for later 
\begin{defi} \label{polsmo}
Let $(X, L)$ be a polarized projective Calabi-Yau variety. A \textit{\textbf{polarized smoothing}} for $(X, L)$ is a flat polarized family $\pi:(\mathcal{X}, \mathcal{L}) \rightarrow \mathbb{D}$ over the unit disk $\mathbb{D} \subset \C$ such that $(X, L)$ is isomorphic to $\pi^{-1}(\{0\})=\left(X_0, L_0\right), \pi^{-1}(\{t\})=\left(X_t, L_t\right)$ is smooth for all $t \neq 0$, and the relative canonical bundle $K_{\mathcal{X} / \mathbb{D}}$ is trivial. We say that such a pair $(X,L)$ is smoothable.
\end{defi}
\begin{rem}
    If $(\mathcal{X}, \mathcal{L})$ is a polarized smoothing, then there is an embedding
$
  \mathcal{X} \;\hookrightarrow\; \mathbb{CP}^N \times \mathbb{D}
$ such that
$
  \mathcal{L}^m \;=\; \mathcal{O}_{\mathbb{D}}(1)\bigl|_{\mathcal{X}}
  \hspace{0.1cm}\text{for some} \hspace{0.2cm} m \ge 1
$, $\pi$ is a proper surjection given by the restriction on $\mathcal{X}$ of the projection of
$\mathbb{CP}^N \times \mathbb{D}$ onto the second factor and $\pi_*$ is of rank $1$ on $\mathcal{X} \setminus \{ x\}$. This, in particular, implies that $X_{t}$, for $t \in \mathbb{D} \setminus \{0\}$, are all diffeomorphic (but not necessarily bi-holomorphic).
\end{rem}

In this work, we are interested in Calabi-Yau manifolds with conical and asymptotically conical Ricci-flat Kähler metrics. Examples of such manifolds have been constructed in many instances. we will mention a couple of these that are of interest. We start with the case of Calabi-Yau cones. The book of Boyer and Galicki \cite{boyer2007sasakian} is a comprehensive reference.
\begin{defi}
 A \textbf{Kähler cone } is a Riemannian cone $\left(C, g_C\right)$ such that $g_C$ is Kähler with respect to some complex structure $J_C$. We then have a Kähler form $\omega_C(X, Y)=g_C\left(J_C X, Y\right)$, and $\omega_C=\frac{i}{2} \partial \bar{\partial} r^2$ with respect to $J_C$ for a distance function $r$.
 \end{defi}
We also give the following definition 

\begin{defi}
    The \textbf{Reeb vector field} of a Kähler cone is the holomorphic Killing field tangent to the link $\xi=J_C\left(r \partial_r\right)$. The closure of the 1-parameter subgroup of $\operatorname{Isom}\left(C, g_0\right)$ generated by $\xi$ is a compact torus $\mathbb{T}$ of holomorphic isometries called the Reeb torus of the cone. We say that the cone is \textbf{regular} if $\mathbb{T}=S^1$ acting freely, quasi-regular if $\mathbb{T}=S^1$ not acting freely and irregular if $\operatorname{dim} \mathbb{T}>1$.
\end{defi}
The following result can be found in \cite[Theorem 3.1]{van2011examples} and \cite[Lemma 2.20]{DonaldsonSun}
\begin{tm} \label{convar}
 If $\left(C, g_0, J_0\right)$ is a Kähler cone, then the complex manifold $\left(C, J_0\right)$ is isomorphic to the regular part of a normal algebraic variety $V \subset \mathbb{C}^N$ with one singular point. In addition, the $\R_{>0}$ scaling action on $C$ generated by $r \partial_r$ can be extended to a diagonal $\R_{>0}$-action on $\C^N$ given by $\left(t, z_1, \ldots, z_N\right) \mapsto\left(t^{w_1} z_1, \ldots, t^{w_N} z_N\right)$ such that all $w_i >0$.
\end{tm}
Due to this result, sometimes we will not differ between the cone $C$ and the variety $V$. This $\R_{>0}$-equivariant realization of the cone as an affine algebraic variety will be important in our work as it will be the main ingredient to construct appropriate weighted Melrose-type blow-ups.  
\begin{defi} \label{cycone}
    We say that $\left(C, g_C, J_C, \Omega_C\right)$ is a \textbf{Calabi-Yau cone} if
        \begin{enumerate}
            \item $\left(C, g_C, J_C\right)$ is a Ricci-flat Kähler cone of complex dimension $n$,
            \item the canonical bundle $K_C$ of $C$ with respect to $J_C$ is trivial, and
            \item $\Omega_C$ is a nowhere vanishing section of $K_C$ with $\omega_C^n=i^{n^2} \Omega_C \wedge \bar{\Omega}_C$.
        \end{enumerate}
        For simplicity, sometimes we will specify a Calabi-Yau cone only by $(C, \omega_C)$. A Calabi-Yau cone is called \emph{quasi-regular} (resp. regular or irregular) if it is quasi-regular (resp. regular or irregular) as a Kähler cone. 
\end{defi}
\begin{ex}
\hfill
    \begin{itemize}
        \item A trivial example is given by $\C^n$ with the Euclidean metric as a cone over $S^{2n-1}$.
        \item More generally, if $G$ is a finite subgroup of $SU(n)$ acting freely on $\C^n \setminus \{ 0 \}$, then the quotient $\C^n / G$ with the induced quotient metric is a Calabi-Yau cone. 
        \item The quadric $C := \left\{ z\in \C^{n+1}  : \sum_{i=1}^n z_i^2 = 0\right\}$ is a Calabi-Yau cone. In \cite{stenzel_ricci-flat_1993}, Stenzel constructed a Ricci-flat Kähler cone metric given by
        $$\omega_{C} = i \partial \overline{\partial} \left( \| z\|^2 \right)^{\frac{n-1}{n}}.$$
        We refer to a singularity which is isomorphic to the quadric $C$ above as a \emph{nodal singulairty}.
        \item The Calabi ansatz \cite{calabi_metriques_1979} is a general construction for regular Calabi-Yau cones. If $D$ is a Kähler-Einstein Fano manifold, then, for every integer $k>0$ dividing $c_1(D)$, there exist a Ricci-flat Kähler cone metric on $(\frac{1}{k} K_D)^{\times}$, the blowdown of the zero section of $\frac{1}{k} K_D$.
    \end{itemize}
\end{ex}

\subsubsection*{Conical Calabi-Yau manifolds}
We work with the following definition of a conical Calabi-Yau manifold.
\begin{defi} \label{def : conical CY}
    A \textbf{conical Calabi-Yau manifold} is a pair $(X_0, \omega_{CY})$ such that $X_0$ is a compact normal complex analytic variety with finitely many isolated singularities which is smooth Calabi-Yau on its regular part with an associated smooth Calabi-Yau metric $\omega_{CY}$ and such that for each $p \in X_0^{sing}$, there exists a biholomorphism $P : V \rightarrow U$ from a neighborhood $V$ of the vertex in a Calabi-Yau cone $(C_p, \omega_{C_p})$ to a neighborhood $U$ of $p$ in $X_0$ such that 
    $$
    P^* \omega_{CY} - \omega_{C_p} = i \partial \overline{\partial}u,
    $$
    where $u \in r^{2+\lambda} C_b^{\infty} (C_p)$, $\lambda >0$ and $r$ is the radial function of $\omega_{C_p}$.
\end{defi}
Notice that, here, we are supposing that the metric is $i \partial \overline{\partial}$-exact near the singularities. The simplest examples of conical Calabi-Yau manifolds are those with orbifold singularities. We refer to the book of Joyce \cite{joyce2000compact} for a reference. 

\begin{defi}
    A \textbf{complex orbifold } is a singular complex-analytic variety of dimension $n$ whose singularities are all locally isomorphic to quotient singularities $\C^n / G$, for finite subgroups $G \subset GL(n, \C)$. We say that $g$ is a \textbf{Kähler metric} on a a complex orbifold $X$ as above if $g$ is a Kähler metric on the non-singular part of $X$ and, whenever $X$ is locally isomorphic to $\C^n / G$, we can identify $g$ with the quotient of a $G$-invariant Kähler metric defined near $0$ in $\C^n$. A \textbf{Kähler orbifold} $(X, J, g)$ is a complex orbifold $(X,J)$ equipped with a Kähler metric $g$.
\end{defi}

It is a well-known fact that compact complex orbifolds with vanishing first Chern class admit Ricci-flat Kähler metrics in the orbifold sense \cite[Theorem 6.5.6]{joyce2000compact}.

\begin{tm} \label{conorbifold}
    Let $(X, J)$ be a compact complex orbifold with $c_1(X) =0 $ in $H^2(X, \R)$, admitting Kähler metrics. Then there is a unique Ricci-flat Kähler metric $g$ in every Kähler class on $X$. We refer to $(X, J, g)$ as a \emph{Calabi-Yau orbifold}.
\end{tm}

If $(X,J,g)$ is a Calabi-Yau orbifold and an orbifold point $x$ has orbifold group $G$ then 
$$
G \subset SU(n).
$$
Therefore, if $G$ acts freely away from $x$, then the Ricci-flat Kähler metric $g$ is conical near $x$ modeled on the Calabi-Yau cone $\C^m / G$ with its induced quotient metric. Moreover, it is clearly $i \partial \overline{\partial}$-exact near the singularities. \\

In the projective setting, in \cite{hein_calabi-yau_2017}, Hein-Sun, obtained a general existence result of conical Calabi-Yau manifolds under the assumptions that the variety admits a polarized smoothing and the cone is strongly regular. More recent works by Chiu-Székelyhidi \cite{ChiuSzekelyhidi} and Zhang \cite{zhang2024polynomial} showed that such assumptions are not necessary. We state the general final result in the following form.
\begin{tm} \label{HS}
Let $X$ be a projective Calabi-Yau variety in the sense of Theorem \ref{EGZ}, of complex dimension $n \geq 3$ with isolated singularities and let $L$ be an ample line bundle on $X$. For convenience, suppose that $X^{sing} = \{ p\}$. Suppose that the germ $(X,p)$ is analytically isomorphic to a neighborhood of the vertex $o$ in a Calabi-Yau cone $(C_p, \omega_{C_p})$. Then the unique Ricci-flat Kähler metric $\omega_{CY} \in 2 \pi c_1(L)$ on $X$ with bounded potential on the germ $(X,x)$ is conical in the sense of Definition \ref{def : conical CY} and asymptotic near $p$ to $\omega_{C_p}$.
\end{tm}
The following is an important example where this result applies  
\begin{ex} \label{ODP}
    Consider a hypersurface $X \subset \mathbb{CP}^{n+1}$ of degree $n+2$ and suppose that $X$ has isolated nodal singularities. A nodal singularity is locally analytically isomorphic to the quadric
\[
C = \left\{ z \in \C^{n+1}, \hspace{0.1cm} Q(z)= z_1^2 + z_2^2 + \cdots + z_{n+1}^2 = 0 \right\}.
\]
As seen before, $C$ is a Calabi-Yau cone when equipped with the Stenzel metric \cite{stenzel_ricci-flat_1993}
$$
\omega_{C} = i \partial \overline{\partial} \left( \| z\|^2 \right)^{\frac{n-1}{n}}.
$$
Therefore, if we let $L : = \restr{\mathcal{O}(1)}{X}$, then $(X,L)$ is smoothable and Theorem \ref{HS} implies that it admits a unique singular Calabi-Yau metric $\omega_X \in 2 \pi c_1(L)$ which is asymptotic to the Stenzel metric $\omega_C$.
\end{ex}

\subsubsection*{Asymptotically conical Calabi-Yau manifolds}
Now, we look at asymptotically conical Calabi-Yau manifolds by which we mean Calabi-Yau manifolds with Ricci-flat Kähler metrics which are asymptotically conical in the sense of Definition \ref{def : conifolds}. There have been several works constructing examples of asymptotically Calabi-Yau manifolds \cite{goto_calabi-yau_2012,van_coevering_ricci-flat_2010, van2011examples, conlon2013asymptotically, conlon2015asymptotically, Tian_Yau1991}, but recently, in \cite{conlon_classification_2024}, the authors gave a complete classification of these.\\

Let $C$ be a Kähler cone. Following Theorem \ref{convar}, we will see $C$ as a normal affine variety. Let $\xi = J(r \partial_r)$ be the associated Reeb vector field and $\mathbb{T}$ be the Reeb torus. Before stating the main result of \cite{conlon_classification_2024}, we need the following technical definition following \cite[Definition 1.7]{conlon_classification_2024}. 

\begin{defi} \label{deformation of negative weight}
 An affine variety $V$ is a \textbf{deformation of negative $\xi$-weight} of $C$ if and only if there exists a sequence $\xi_i$ of elements of $\operatorname{Lie}(\mathbb{T})$ and a sequence $c_i$ of positive real numbers such that
 \begin{enumerate}
\item  $\xi_i \rightarrow \xi$ as $i \rightarrow \infty$;
\item  the vector field $-J\left(c_i \xi_i\right)$ generates an effective algebraic $\mathbb{C}^*$-action on $C$;
\item  there exists a $\mathbb{C}^*$-equivariant deformation of $V$ to $C$, i.e., a triple $\left(W_i, p_i, \sigma_i\right)$, where
\begin{itemize}
\item  $W_i$ is an irreducible affine variety,
\item $p_i: W_i \rightarrow \mathbb{C}$ is a regular function with $p_i^{-1}(0) \cong C$ and $p_i^{-1}(t) \cong V$ for $t \neq 0$, and
\item $\sigma_i: \mathbb{C}^* \times W_i \rightarrow W_i$ is an effective algebraic $\mathbb{C}^*$-action on $W_i$ such that
\item $p_i\left(\sigma_i(\lambda, x)\right)=\lambda^{\mu_i} p_i(x)$ for some $\mu_i \in \mathbb{N}$ and all $\lambda \in \mathbb{C}^*, x \in W_i$, and
\item $\sigma_i$ restricts to the $\mathbb{C}^*$-action on $p_i^{-1}(0) \cong C$ generated by $-J\left(c_i \xi_i\right)$;
\end{itemize}
\item  $\lim _{\lambda \rightarrow 0} \sigma_i(\lambda, x)=o$ for every $x \in W_i$, where $o \in C$ is the apex of the cone and 
\item the sequence $\nu_i:=-\left(k_i \mu_i\right) / c_i$ is uniformly bounded away from zero, where $k_i \in \mathbb{N} \cup\{\infty\}$ is the vanishing order of the deformation $\left(W_i, p_i\right)$, i.e., the supremum of all $k \in \mathbb{N}$ such that $\left(W_i, p_i\right)$ becomes isomorphic to the trivial deformation of $C$ after base change to Spec $\mathbb{C}[t] /\left(t^k\right)$.
 \end{enumerate}
If $V$ satisfies this condition, then we define the $\xi$-weight of $V$ to be the infimum over all possible sequences $\xi_i$ and $\left(W_i, p_i, \sigma_i\right)$ as above of $\limsup _{i \rightarrow \infty} \nu_i$. This is a negative real number $\nu$.
\end{defi}

\begin{rem} \label{remark negative weight}
\hfill
\begin{itemize}
    \item If the cone $C$ is quasi-regular i.e. $\mathbb{T} = S^1$, then in the definition of a deformation with negative $\xi$-weight above, the sequence $(W_i, p_i, \sigma_i)$ collapses to a single element $(W, p, \sigma)$. Following \cite[Theorem 2.2]{conlon_classification_2024}, if $(W,p, \sigma)$ is as above, then there exist $\mu_1, \mu_2, \cdots, \mu_N \in \N$ coprime and an embedding $\phi : W \rightarrow \C_t \times \C^N$ such that 
    \begin{itemize}
        \item $p = \pi_t \circ \phi$, where $\pi_t : \C_t \times \C^N \rightarrow \C_t $ is the projection into the first coordinate.
        \item $\phi (\sigma(\lambda, x)) = \mathrm{diag}(\lambda^\mu, \lambda^{\mu_1}, \cdots, \lambda^{\mu_N}) \cdot \phi(x)$ for all $\lambda \in \C^{*}$ and $x \in W$.
    \end{itemize}
    \item When $C$ is quasi-regular and a complete intersection of quasi-homogeneous hypersurfaces, the authors gave a characterization \cite[Lemma 4.1]{conlon_classification_2024} which simplifies the condition of a deformation of negative $\xi$-weight. In such a case, if $V$ is a deformation of $C$ with negative $\xi$-weight then it is isomorphic as an affine variety to a connected component of some fiber of the semi-universal Artin-Elkik \cite{artin1976lectures, Elkik} deformation $\mathcal{W} \rightarrow S$ of $C$ which can be made to be $\C^{*}$-equivariant thanks to the work of Slodowy \cite{MR584445}. See \cite[Section 4.1]{conlon_classification_2024} for more details.
    \end{itemize}
\end{rem}
With this definition, we state the main result classifying asymptotically conical Calabi-Yau manifolds \cite[Theorem A, Theorem B]{conlon_classification_2024}.
\begin{tm}[Conlon-Hein] \label{ConHein}
We fix a Calabi-Yau cone $(C, \omega_C)$. Suppose $V$ is an affine variety which is a deformation of $C$ with negative $\xi$-weight and let $\pi : M \rightarrow V$ be a holomorphic crepant resolution such that $M$ is Kähler. Then, for any class $\mathrm{t} \in H^2(M, \R)$ such that $\langle\mathrm{t}^d , Z \rangle >0$ for all irreducible subvarieties $Z$ of $\mathrm{Exc} (\pi)$, $d = \dim Z >0$, and for all $g \in \mathrm{Aut}^{T}(C)$, a transverse automorphism of the cone $C$, $M$ admits an asymptotically conical Ricci-flat Kähler metric which is asymptotic to $g^{*}\omega_C$. In addition, these classify all asymptotically conical Calabi-Yau manifolds up to diffeomorphism.
\end{tm}

\begin{rem} \label{remark AC CY}
Theorem \ref{ConHein} above includes, in particular, the earlier results by Van Coevering and Goto \cite{van_coevering_ricci-flat_2010, goto_calabi-yau_2012}, which prove the existence of asymptotically conical Ricci-flat Kähler metrics on crepant resolutions of Calabi-Yau cones. It also includes the general case of appropriate affine smoothings of a Calabi-Yau cone, hence, generalizing the example by Stenzel \cite{stenzel_ricci-flat_1993} which gives an asymptotically conical Ricci-flat Kähler metric on the cotangent bundle of the sphere $T^{*}S^n$ which is isomorphic to a smoothing of the node singularity i.e. 
$$
T^{*}S^n \cong \{ z \in \C^{n+1};  \hspace{0.1cm} \sum_{i=1}^{n+1} z_i^2 =1 \}.
$$ 
\end{rem}

\section{Polyhomgeneity of conical Calabi-Yau manifolds} \label{conphg}
In \cite[Theorem 6.3]{conlon_moduli_2015}, the authors prove the following polyhomogeneity result.
\begin{tm} \label{acphg}
Suppose that $\widetilde{M} \backslash \partial \widetilde{M}$ is a complex manifold and that the complex structure $J$ extends to an element $J \in \mathcal{A}_{\mathrm{phg}}^Q\left(\widetilde{M} ; \operatorname{End}\left({ }^{\text {sc }} T \widetilde{M}\right)\right)$ for some $Q \geq 0$. Suppose $g_{\mathrm{sc}}$ is a polyhomogeneous scattering metric which is Kähler with respect to $J$, and has Kähler form $\omega_{\mathrm{sc}}$. Let $F$ be a positive index set. If $f \in \mathcal{A}_{\mathrm{phg}}^F(\widetilde{M})$ and for some $\epsilon>0, u \in \rho^{\epsilon-2} C_b^{\infty}(\widetilde{M})$ satisfies
$$
\frac{\left(\omega_{\mathrm{sc}}+i \partial \bar{\partial} u\right)^n}{\omega_{\mathrm{sc}}^n}=e^f,
$$
then $u \in \mathcal{A}_{\mathrm{phg}}^{G-2}\left(\widetilde{M}\right)$ for some $G>0$.
\end{tm}

As a corollary, the authors apply this theorem to prove the polyhomogeneity at infinity of several examples of asymptotically conical Calabi-Yau manifolds such as the ones from the Tian-Yau \cite{TianYau2} construction and its generalization by Conlon-Hein \cite{conlon2015asymptotically} as well as Goto and Van Coevering's constructions \cite{goto_calabi-yau_2012, van_coevering_ricci-flat_2010} on crepant resolutions of Calabi-Yau cones. In this section, we prove Theorem \ref{tmA} previously stated in the introduction which gives a similar result for Calabi-Yau manifolds with isolated conical singularities.
\begin{lm} \label{tmphg}
Let $X$ be a complex manifold admitting a compactification by a manifold with boundary $\overline{X}$. For simplicity, we suppose there is only one boundary component with boundary defining function $r$. Suppose that the complex structure $J$ is polyhomogeneous with positive index set as a section of $End( \hspace{0.01cm} ^{c}T \overline{X})$. Let $g_c$ be a polyhomogeneous conical metric on $X$ which is Kähler with respect to $J$ with Kähler form $\omega_c$. Let $F$ be a positive index set. If $v \in \mathcal{A}_{\mathrm{phg}}^F(\overline{X})$ and for some $\lambda>0$, $u \in r^{\lambda + 2} C_b^{\infty}(X)$ satisfies
$$
\frac{\left(\omega_{\mathrm{c}}+i \partial \bar{\partial} u\right)^n}{\omega_{\mathrm{c}}^n}=e^v
$$
then $u \in \mathcal{A}_{\text {phg }}^{G}\left(\overline{X}\right)$ for some index set $G > 2$.
\end{lm}
\begin{proof}
Using the binomial expansion, we have 
$$
\frac{\left(\omega_{\mathrm{c}}+i \partial \bar{\partial} u\right)^n}{\omega_{\mathrm{c}}^n} = 1+\Delta_{\omega_{\mathrm{c}}} u+\sum_{j=2}^n \frac{n!}{(n-j)!j!}\left(\frac{\omega_{\mathrm{c}}^{n-j} \wedge(i \partial \bar{\partial} u)^j}{\omega_{\mathrm{c}}^n}\right)
$$
where $\Delta_{\omega_{ c}}$ is the $\bar{\partial}$-Laplacian associated to the Kähler form $\omega_{\mathrm{c}}$. It is equal to half the Laplace-Beltrami operator associated to the corresponding Riemannian metric $g_c$. We put 

$$f_1(u):=\sum_{j=2}^n \frac{n!}{(n-j)!j!}\left(\frac{\omega_{\mathrm{c}}^{n-j} \wedge(i \partial \bar{\partial} u)^j}{\omega_{\mathrm{c}}^n}\right) \hspace{0.1cm} \text{and  } \hspace{0.1cm} f_2:=e^v - 1.$$
Therefore, we write the complex Monge-Ampère equation in the linear form 
$$\Delta_{\omega_{ c}} u = f_1(u) + f_2$$
with 
$f_2 \in \mathcal{A}_{\mathrm{phg}}^{\widetilde{F}}(\overline{X})$ for a positive index set $\widetilde{F}$ and since all the terms in $f_1$ are at least quadratic, we also have $f_1(u) \in r^{2 \lambda} C_b^{\infty}(X)$. By Proposition \ref{phgconlap}, we get $u=\tilde u_1 + u_1$ with $\tilde u_1$ polyhomogeneous with index set $\widetilde{G_1} >  2$ and $u_1 \in r^{2 + 2 \lambda - \delta} C_b^{\infty}(X) $ for $\delta = \frac{\lambda}{2}$. \\
We set $\omega_{ c, 1} : = \omega_{ c} + i \partial \bar{\partial} \tilde u_1  $. Up to multiplying $\tilde u_1$ by a cut-off function, we can suppose that $\omega_{ c, 1}$ is the Kähler form of a polyhomogeneous conical metric. We, therefore, get the following complex Monge-Ampère equation for $u_1$ 
$$
\frac{\left(\omega_{\mathrm{c,1}}+i \partial \bar{\partial} u_1\right)^n}{\omega_{\mathrm{c,1}}^n}=e^{v_1} \hspace{0.1cm} \text{with} \hspace{0.1cm} v_1 = v + \log\left({\frac{\omega_{\mathrm{c}}^n}{\omega_{\mathrm{c,1}}^n}}\right).
$$
Now, we can apply the same argument as above for $u_1$ to deduce $u_1= \tilde u_2 + u_2$ with $\tilde u_2$ polyhomogeneous with index set $\widetilde{G_2} > 2$ and $u_2 \in r^{2 + 3 \lambda - \delta} C_b^{\infty}(X) $. \\
Using induction, we can iterate the argument to get for each $k \in \N $, 
$u=(\sum_{i=1}^k \tilde u_i) + u_k $ with $\tilde u_i$ polyhomogeneous with index set $\widetilde{G_i} > 2$ for all $i \leq k$ and $u_k \in r^{2+ (k+1) \lambda - \delta} C_b^{\infty}(X) $ and therefore $u$ is polyhomogeneous with index set $G > 2$.  
\end{proof}

 \begin{proof}[Proof of Theorem \ref{tmA}]
To keep notations light, we suppose that $X_0$ has only one isolated singularity. Using notations of Definition \ref{def : conical CY}, let $\omega_c$ be a Kähler form such that $\restr{\omega_c}{V} = P^{*} \omega_{C_p}$ and $\omega_c \in [\omega_{CY}]$. Such a Kähler form exists by \cite[Proposition 2.4.]{arezzo2016csck} since, by supposition, $[\omega_{CY}]=[\omega]$ for smoothly Kähler form $\omega$. Therefore, $\omega_{CY} = \omega_c + i \partial \overline{\partial} u$ for some function $u$ smooth on $X \setminus \{ p\}$. By \cite[Lemma A.1]{hein_calabi-yau_2017}, we can suppose that $u \in r^{\lambda + 2} C_{b}^{\infty}(X)$. Therefore, $u$ is a solution in $ r^{\lambda + 2} C_{b}^{\infty}(X)$ to the complex Monge–Ampère equation:  
$$
\frac{\left(\omega_{c}+i \partial \bar{\partial} u\right)^n}{\omega_{c}^n}=e^v \text{ , }\hspace{0.1cm} \hspace{0.1cm}  v = \log \left(\frac{ i^{n^2} \Omega \wedge \bar \Omega}{\omega_{c}^{n}}\right),
$$
where $\Omega$ is a nowhere vanishing holomorphic $(n,0)$-form on $X^{\text{reg}}$. Since $X$ is a normal variety and $K_X$ is trivial, then $\Omega$ is analytic near $x \in  X \backslash X^{\text {reg }}$ and therefore, in particular, polyhomogeneous. Also, by construction, $\omega_{c}$ is a cone metric and therefore also  polyhomogeneous. Since the ratio of the volume forms is a positive function with non-vanishing leading term at $r=0$, we conclude, using the Taylor series of the functions $\frac{1}{1+y}$ and $\log(1+y)$ that $v= \log \left(\frac{i^{n^2} \Omega \wedge \bar \Omega}{\omega_{c}^{n}}\right)$ is itself polyhomogeneous and since $u \in r^{2+ \lambda} C_b^\infty$ implies $i \partial \overline{\partial} u \in r^{\lambda} C_b^\infty$, then $\lim_{r \rightarrow 0} v =0$ and therefore it has positive index set. Thus, all conditions of Lemma \ref{tmphg} are verified for and we get the polyhomogeneity of $\omega_{CY} = \omega_{c} + i \partial \bar \partial u $ from that of  $u$ and $\omega_{c}$.   
 \end{proof}
Here the condition of being smoothly Kähler is needed for the existence of the Kähler form $\omega_c$ from above. See \cite[Proposition 2.4]{arezzo2016csck}. The case of Calabi-Yau orbifolds is already well-known since the Ricci-flat Kähler metrics are induced near an orbifold point $x$ locally isomorphic to $\C^n/G$  by genuine smooth Kähler metrics on $\C^n$.\\
 
 The rest of this paper will be dedicated to statements, constructions and proofs related to Theorems \ref{tmB} and \ref{tmC}. 
 \section{Conical degenerations of Calabi-Yau metrics} \label{set}
 In this section, we will give the setting and the statement of our main general theorem (See Theorem \ref{main}) and after that we will describe the constructions needed on crepant resolutions and polarized smoothings to apply our theorem to get Theorems \ref{tmB} and \ref{tmC}. The proof of Theorem \ref{main} is carried out in Sections \ref{secform} and \ref{secfix}.
\subsection{Surgery space}
\begin{defi} \label{srg}
We call a pair $\left(\mathcal{X}_b, \pi_b \right)$ a \textbf{surgery space} if it satisfies the following: 
\begin{itemize}
    \item $\mathcal{X}_b$ is a manifold with corners.
    \item $\pi_b : \mathcal{X}_b \rightarrow [0, \varepsilon_0]_\varepsilon $ is a b-fibration in the sense of Melrose such that 
    \begin{itemize}
        \item For $\varepsilon >0$, $X_\varepsilon := \pi_b^{-1} (\left\{ \varepsilon \right\})$ is a smooth closed manifold;
        \item $\pi_b^{-1} (\left\{ 0 \right\}) = \left( \bigcup_{i=1}^N B_{I,i} \right) \cup B_{II}$;
        \item For all $i \in \{1,2, \cdots, N \}$, $B_{I,i}$ is a manifold with smooth, closed, connected boundary $\Sigma_i$;
        \item $B_{II}$ is a manifold with boundary such that: $\partial B_{II} = \bigcup_{i=1}^N \Sigma_i$;
        \item We can choose boundary defining functions for $B_{I,i}$ and $B_{II}$ denoted by $\rho_{1,i}$ and $\rho_{2}$ respectively such that $\varepsilon = \rho_2 \rho_{1,1} \cdots \rho_{1,N}$.
    \end{itemize}
\end{itemize}
\end{defi}
\vspace{0.5cm}
\begin{rem}
\hfill
\begin{itemize}
\item The name surgery space is in accordance with the space introduced in \cite{melrose_analytic_1995}. Surgery, here, refers to the fact that the manifolds $B_{I,i}$ are attached to $B_{II}$ along their respective boundaries.
\item For simplicity, to lighten the notations, we will suppose that $\partial B_{II}$ has only one connected component, i.e, there is only one $B_I$ and $\varepsilon = \rho_1 \rho_2$. However, our methods extend to the case of several connected components of $B_{II}$ with only minor modifications.
\item 
In practical examples, $B_{II}$ represents the lift of some manifold with isolated singularities under an appropriate blow-up, while the $B_{I,i}$ are the new faces appearing when blowing up the singularities.
\item The parametric connected sums considered in \cite[Definition 7.6]{pacini_desingularizing_2013} are surgery spaces in the above sense when the marking is taken to be the whole set of ends.
\item In general, if $u$ is any type of object defined on $\mathcal{X}_b$ (e.g. a function), we use the notation $u_\varepsilon$ to denote its restriction to a fiber $X_\varepsilon$.
\end{itemize}
\end{rem}
 Now, since $\mathcal{X}_b$ is a manifold with corners, we can define as usual the \textbf{b-tangent bundle} $^{b} T \mathcal{X}_b$. Since $\pi_{b}$ is a b-fibration, following \cite{melrose_analytic_1995}, we can define a new bundle called the \textbf{$(b, \varepsilon)$-tangent bundle} in the following way: 
$$
^{b,\varepsilon} T \mathcal{X}_b := \ker \hspace{0.05cm} ^{b} (\pi_{b})_{*} \subset \hspace{0.05cm} ^{b} T \mathcal{X},
$$
where $^{b} (\pi_{b})_{*}$ is the b-differential of $\pi_{b}$. The sections of this bundle are the vector fields that are tangent to both the boundary and to the fibers $X_\varepsilon$. Therefore, the restriction of  $^{b,\varepsilon} T \mathcal{X}_b$ on $X_\varepsilon$ is isomorphic to $TX_\varepsilon$.\\
For our purposes, we also consider a rescaled version called the \textbf{$(c, \varepsilon)$-tangent bundle} defined by: 
$$
^{c,\varepsilon} T \mathcal{X}_b := \frac{1}{\rho_{1}} \hspace{0.05cm} ^{b,\varepsilon} T \mathcal{X}_b,
$$
and we denote its dual by $^{c,\varepsilon} T^{*} \mathcal{X}_b$.
\subsection{Setting of the general theorem} 
Now, the setting for our main theorem is as follows. We let $\left(\mathcal{X}_b, \pi_b, \omega, J \right)$ be such that:
    \begin{itemize}
    \item $(\mathcal{X}_b , \pi_b)$ is a surgery space and we suppose, for convenience, that $\partial B_{II}$ has only one connected component,
    \item $\omega \in \mathcal{A}_{phg \geq 0} \left( \hspace{0.01cm} ^{c,\varepsilon} T^{*} \mathcal{X}_b \wedge \hspace{0.01cm} ^{c,\varepsilon} T^{*} \mathcal{X}_b \right)$,
    \item $J \in \mathcal{A}_{phg \geq 0} \left( End \left( \hspace{0.01cm} ^{c,\varepsilon} T \mathcal{X}_b \right) \right)$,
    \end{itemize}
    such that:
    \begin{enumerate}
        \item For all $\varepsilon >0$,  $(X_\varepsilon, J_\varepsilon)$ is a smooth closed Calabi-Yau manifold and $\omega_\varepsilon := \restr{\omega}{X_\varepsilon}$ is a Kähler form,
        \item $\left(B_{II}, \restr{\omega}{B_{II}}, \restr{J}{B_{II}}\right)$ is a Calabi-Yau manifold with a polyhomogeneous conical Ricci-flat Kähler metric modeled on a Calabi-Yau cone $\left(C_0, \omega_{0}, J_0 \right)$,
        \item $\left(B_{I}, \restr{\frac{\omega}{\varepsilon^2}}{B_{I}}, \restr{J}{B_{I}}\right)$ is a Calabi-Yau manifold with a polyhomogeneous asymptotically conical Ricci-flat Kähler metric modeled on the same cone $\left(C_0, \omega_{0}, J_0 \right)$. 
     \end{enumerate}
     
    \begin{rem} \label{remden}
    Choosing boundary defining functions $\rho_1$ and $\rho_2$ for $B_{I}$ and $B_{II}$ respectively such that $\varepsilon = \rho_1 \rho_2$, the conditions on $\omega$ and $J$ above imply that the associated metric with parameter $g(u,v) := \omega(u, Jv) $ is of the form $g = \rho_1^2 \cdot g_{b, \varepsilon}$ where $g_{b, \varepsilon}$ is a b-surgery metric in the sense of Mazzeo-Melrose \cite{melrose_analytic_1995}. Therefore, in particular, we have $\omega^n = \rho_1^{2n} d g_{b, \varepsilon}$ as densities.
    \end{rem}
    \subsubsection{Polyhomogeneity of the Ricci potential}
    Before stating our theorem, we start by giving a lemma that proves that, with the conditions considered above, the potential of the Ricci form is in fact polyhomogeneous. To do so, we will see it as a solution to a Laplace equation.  For that reason, we start by mentioning the following result, that we need later, which gives the uniform boundedness of the Laplacian $\Delta_\varepsilon$ associated to $\omega_\varepsilon$ on weighted Sobolev spaces proved in \cite[Theorem 12.3]{pacini_desingularizing_2013}. In what follows, we denote the weighted Sobolev spaces under consideration by 
    $$H_{\nu}^k (\mathcal{X}_b) := \rho_1^{\nu}H_{b, \varepsilon}^k(\mathcal{X}_b).$$
\begin{lm} \label{lmunf}
Choose a weight $\nu \in(2-2n, 0)$. Assume $B_{I}$ has only one connected component. Then there exists a uniform constant $C>0$ and a subspace $\left(H_{ \nu}^k\left(\mathcal{X}_b\right) \right)' \subset H_{ \nu}^k\left(\mathcal{X}_b\right)$ such that, for $\varepsilon >0$ sufficiently small,
$$
\restr{H_{ \nu}^k\left(\mathcal{X}_b\right)}{X_\varepsilon}= \restr{\left(H_{ \nu}^k\left(\mathcal{X}_b\right) \right)'}{X_\varepsilon} \oplus \mathbb{R}
$$
and, for all $f \in \restr{\left(H_{ \nu}^k\left(\mathcal{X}_b\right) \right)'}{X_\varepsilon}$,
$$
\|f\|_{H_{\nu}^k} \leq C\left\|\Delta_{ \varepsilon} f\right\|_{H_{\nu -2 }^{k-2}} .
$$

Furthermore, the image of the restricted operator $\Delta \mid_{\left(H_{ \nu}^k\left(\mathcal{X}_b\right) \right)'}$ coincides with the image of the full operator $\Delta$.
\end{lm}
 Now, we prove the polyhomogeneity of the Ricci potential
 
     \begin{pr} \label{ricpot}
         With the same notations as before, let $v$ be defined on $\mathcal{X}_b$ by: 
    $$
    \mathrm{Ric} (\omega_\varepsilon) =  i \partial \overline{\partial}v_\varepsilon, \hspace{0.1cm} \int_{X_\varepsilon} \omega_\varepsilon^n = \int_{X_\varepsilon} e^{v_\varepsilon} \omega_\varepsilon^{n}, \hspace{0.1cm} v_\varepsilon = \restr{v}{X_\varepsilon}.
    $$ 
    Then, $v$ is polyhomogeneous and vanishes at $B_{I}$ and $B_{II}$. 
     \end{pr}
     
     \begin{proof}
       Since $\omega$ and $J$ are polyhomogeneous, then so is the Ricci-form $\mathrm{Ric}(\omega_\varepsilon)$ as seen on $\mathcal{X}_b$ and it is equal to zero on $B_{I}$ and $B_{II}$ by supposition. We want to prove that the potential fixed by integration is also polyhomogeneous. First, by applying the trace, we get 
       $$
       s_{\omega_\varepsilon} = \mathrm{tr}_{\omega_\varepsilon}(\mathrm{Ric}(\omega_\varepsilon)) = 2n \frac{i \partial \overline{\partial}v_\varepsilon \wedge \omega_\varepsilon^{n-1}}{\omega_\varepsilon^{n}} = \Delta_\varepsilon v_\varepsilon,
       $$
       where $s_{\omega_\varepsilon}$ is the scalar curvature of $\omega_\varepsilon$ and $\Delta_\varepsilon$ is the Laplacian associated to $\omega_\varepsilon$. For simplicity, we write
 \begin{align} \label{linpot}
       \Delta v = s \hspace{0.1cm} \hspace{0.1cm}   
\end{align}
    We will construct a polyhomogeneous solution $u$ to equation (\ref{linpot}) which also vanishes at $B_{I}$ and $B_{II}$. If such a construction is possible, then, since harmonic functions on closed manifolds are constant, we get that $v = u + c$, where $c$ is constant on each fiber and it is fixed by the integral
    $$
    \int_{X_\varepsilon} \omega_\varepsilon^n = \int_{X_\varepsilon} e^{v_\varepsilon} \omega_\varepsilon^n = e^{c_\varepsilon}\int_{X_\varepsilon} e^{u_\varepsilon}\omega_\varepsilon^n,
    $$
    therefore,
    $$
    c_\varepsilon = - \log \left( \frac{\int_{X_\varepsilon} e^{u_\varepsilon} \omega_\varepsilon^n}{\int_{X_\varepsilon} \omega_\varepsilon^n} \right) = - \log \left( 1 + \frac{\int_{X_\varepsilon} \left(e^{u_\varepsilon} -1 \right)\omega_\varepsilon^n}{\int_{X_\varepsilon} \omega_\varepsilon^n} \right).
    $$
    Hence, if $u$ is polyhomogeneous, then $\left(e^{u} -1 \right)\omega^n$ is a polyhomogeneous $(c,\varepsilon)$-density, therefore by the Melrose push-forward theorem \ref{fwd}, $\int_{X_\varepsilon} \left(e^{u_\varepsilon} -1 \right)\omega_\varepsilon^n$ is polyhomogeneous and vanishes at $\varepsilon = 0$ because $u$ does. Similarly, $\int_{X_\varepsilon} \omega_\varepsilon^n$ is also polyhomogeneous and it is non-vanishing at $\varepsilon = 0$. Therefore, using the Taylor series of $\frac{1}{1+y}$, we deduce that the quotient is also polyhomogeneous. Finally, using the Taylor expansion of $\log (1+y)$, we deduce that $c$ is polyhomogeneous.\\
    
    Now, we get back to the construction of the polyhomogeneous solution $u$. First, we will solve formally in the polyhomogeneous sense, and then solve exactly using the isomorphism theorem for the Laplacian given in Lemma \ref{lmunf}. \\

    First, since $s$ is the scalar curvature, we get that $s$ is polyhomogeneous and satisfies
    $$
    \int_{X_\varepsilon} s_\varepsilon \omega_\varepsilon^n =0.
    $$
    Moreover, $s$ vanishes at $B_{II}$ and $\varepsilon^2 s$ vanishes on $B_I$. For this reason, first, we replace $s$ by $\hat s := \varepsilon^2 s$ and we solve 
\begin{align} \label{linpot2}
    \Delta \hat u = \hat s \hspace{0.1cm} \hspace{0.1cm} 
\end{align}
    
    We will solve formally, first, on $B_{I}$ and then on $B_{II}$. 

    \begin{itemize}
        \item Suppose $2< \delta < 2n$ is the smallest power of $\rho_2$ appearing in the polyhomogeneous expansion of $\hat s$ near $B_{II}$ and let $ \hat s_\delta :=  \sum_{i=0}^k \hat s_{\delta, i} \rho_2^{\delta}\hspace{0.03cm}(\log \rho_2)^i$ be the terms of the expansion at such order, i.e.
        $$
        \hat s = \hat s_\delta + o(\rho_2^\delta) =   \sum_{i=0}^k \hat s_{\delta, i} \rho_2^{\delta}\hspace{0.03cm}(\log \rho_2)^i + o(\rho_2^\delta).
        $$
Since $\hat s$ vanishes at $B_{I}$, we get that $\hat s_{\delta, i} \in \rho_1^{\lambda}C_b^\infty (B_{II})$, for some $\lambda >0$. Using the fact that $\varepsilon = \rho_1 \rho_2$, we rewrite: 
        \begin{align*}
        \sum_{i=0}^k \hat s_{\delta, i} \rho_2^{\delta}\hspace{0.03cm}(\log \rho_2)^i & = \sum_{i=0}^k \hat s_{\delta, i} \left(\frac{\varepsilon}{\rho_1}\right)^{\delta}\hspace{0.03cm} \left(\log \left(\frac{\varepsilon}{\rho_1}\right) \right)^i \\ & = \sum_{i=0}^k \hat  s_{\delta, i} \rho_1^{- \delta} \varepsilon^{\delta}\hspace{0.03cm} \left(\log \varepsilon  - \log \rho_1 \right)^i \\ & =\sum_{i=0}^k \hat s_{\delta, i} \rho_1^{- \delta} \varepsilon^{\delta}\hspace{0.03cm}  \left( \sum_{l=0}^i C_{i,l} (\log \varepsilon)^{i-l} (\log \rho_1)^{l} \right) \\ & = \sum_{j=0}^{k} \tilde s_{\delta, j} \varepsilon^{\delta}  \left(\log \varepsilon \right)^j,
         \end{align*}
         where $\tilde s_{\delta, j} := \sum_{r=j}^{k} C_{r,r-j} \hat s_{\delta, r} \hspace{0.02cm}\rho_1^{- \delta} \left( \log \rho_1 \right)^r$ and therefore $\restr{\tilde s_{\delta, j}}{B_{II}} \in \rho_1^{b} C_b^\infty(B_{II})$, for any $- \delta < b < - \delta + \lambda$. For simplicity, we also denote the restriction $\restr{\tilde s_{\delta, j}}{B_{II}}$ by $\tilde s_{\delta, j}$. We choose a weight $ a \in (0,2) \setminus \mathcal{P}$ or $ a \in (2-2n,0) $ as in Proposition \ref{conlap} such that $a -2  > - \delta $ and 
    $$
    \tilde s_{\delta,j} \in \rho_1^{a-2} C_b^{\infty}(B_{II}).
    $$
    This, alone, is not enough to apply the isomorphism result of Proposition \ref{conlap}, we also need to verify that
\begin{align} \label{pushfwdeq1}
    \int_{B_{II}} \tilde s_{\delta,j} \hspace{0.02cm} \omega_2^n =0.
   \end{align}
    Indeed, we know that the map 
\begin{align} \label{pushfwdeq2}
c \;\mapsto\; \int_{\varepsilon = c} 
s\, \omega^n,
\end{align}
vanishes identically and therefore all the coefficients in its polyhomogeneous expansion near $c=0$ are zero. On the other hand, as mentioned in Remark \ref{remden}, since in terms of $(b, \varepsilon)$-densities, we have  
$$
\omega^n = \rho_1^{2n} dg_{b,\varepsilon},
$$
for a $(b,\varepsilon)$-metric $g_{b,\varepsilon}$ in the sense of \cite{melrose_analytic_1995}, then $\omega^n$ is of order $2n$ on $B_I$ and therefore, by the push-forward theorem of Melrose \ref{fwd}, the coefficient at the order $\varepsilon^\delta (\log \varepsilon)^j$ with $\delta < 2n$ in (\ref{pushfwdeq2}) comes from the asymptotic expansion of $s\, \omega^n$ on $B_{II}$ and is, exactly, given by
$$\int_{B_{II}} \tilde s_{\delta,j} \hspace{0.02cm} \omega_2^n.$$
Therefore (\ref{pushfwdeq1}) is verified. Thus, by Proposition \ref{conlap} and Proposition \ref{phgconlap}, there exists a unique polyhomogeneous $w_{\delta,j} \in \rho_1^a C_b^\infty (B_{II})$ with $\int_{B_{II}} w_{\delta,j} \hspace{0.03cm} \omega_2^n = 0$ such that 
$$
\Delta_{II} w_{\delta,j} =  \tilde s_{\delta,j}.
$$
Extending the coefficients of the expansion of $w_{\delta,j}$ off $B_{II}$ and setting $\hat u_{\delta,j} = w_{\delta,j} \varepsilon^\delta (\log \varepsilon)^j$ we get that 
\begin{align*}
\Delta \hat u_{\delta,j} & = \varepsilon^\delta(\log \varepsilon)^j \Delta w_{\delta,j} \\& = \varepsilon^\delta (\log \varepsilon)^j\Delta_{II}  w_{\delta,j} + \varepsilon^\delta(\log \varepsilon)^j s'\\& = \varepsilon^\delta (\log \varepsilon)^j \tilde s_{\delta,j} + \varepsilon^\delta(\log \varepsilon)^j s', \end{align*}
where $s' \in o (1)$ when $ \rho_2 \rightarrow 0$. Therefore, putting $\hat u_{\delta} = \sum_{j=0}^k \hat u_{\delta,j}  = \sum_{j=0}^k w_{\delta,j} \hspace{0.02cm} \varepsilon^\delta (\log \varepsilon)^j$, we get
$$
\Delta \hat u_\delta = \sum_{j=0}^k \tilde s_{\delta,j} \hspace{0.02cm}  \varepsilon^\delta (\log \varepsilon)^j + o(\rho_2^\delta) = \sum_{i=0}^k \hat  s_{\delta, i} \rho_2^{\delta} (\log \rho_2)^i + o(\rho_2^\delta) = \hat s_\delta + o(\rho_2^\delta).
$$
Therefore, $\hat u_\delta$ is a formal solution for (\ref{linpot2}) at order $\delta$ near $B_{II}$. Replacing $\hat s$ by $\hat s - \Delta \hat u_\delta$ and noticing that 
$$
\int (\hat s - \Delta \hat u_\delta) \omega^n = \int \hat s \omega^n - \int \Delta \hat u_\delta \omega^n = 0 - 0 = 0,
$$
we can iterate the argument to remove all terms in the expansion near $B_{II}$ of order $0<\delta < 2n$. Therefore, putting $\hat u_2 = \sum_{0 < \delta < 2n} \hat u_\delta$, we can suppose that 
$$
\hat s - \Delta \hat u_2 \in \rho_2^{2n - \beta} C_b^\infty(\mathcal{X}_b) \hspace{0.1cm}\text{for some}\;\beta>0\;\text{arbitrarily small}.
$$
        \item Now, we solve on $B_{I}$ without comprimising the improvement on $B_{II}$. Suppose $0< \delta< 2n -2 $ is the smallest power of $\rho_1$ in the expansion of $\hat s - \Delta \hat u_2$ on $B_{I}$ and suppose $\hat s_\delta =\sum_{i=0}^k \hat s_{\delta,i} \rho_1^{\delta} (\log \rho_1)^i$ is its expansion at that order, i.e.
$$
\hat s - \Delta \hat u_2 = \hat s_\delta + o(\rho_1^\delta) = \sum_{i=0}^k \hat s_{\delta,i} \rho_1^{\delta} (\log \rho_1)^i + o(\rho_1^\delta).
$$
As before, we have 
\begin{align*}
        \sum_{i=0}^k \hat s_{\delta, i} \rho_1^{\delta}\hspace{0.03cm}(\log \rho_1)^i & = \sum_{i=0}^k \hat s_{\delta, i} \left(\frac{\varepsilon}{\rho_2}\right)^{\delta}\hspace{0.03cm} \left(\log \left(\frac{\varepsilon}{\rho_2}\right) \right)^i \\ & = \sum_{i=0}^k \hat  s_{\delta, i} \rho_2^{- \delta} \varepsilon^{\delta}\hspace{0.03cm} \left(\log \varepsilon  - \log \rho_2 \right)^i \\ & =\sum_{i=0}^k \hat s_{\delta, i} \rho_2^{- \delta} \varepsilon^{\delta}\hspace{0.03cm}  \left( \sum_{l=0}^i C_{i,l} (\log \varepsilon)^{i-l} (\log \rho_2)^{l} \right) \\ & = \sum_{j=0}^{k} \tilde s_{\delta, j} \varepsilon^{\delta}  \left(\log \varepsilon \right)^j,
         \end{align*}
where $\tilde s_{\delta, j}  : = \sum_{r=j}^{k} C_{r,r-j} \hat s_{\delta, r} \hspace{0.02cm}\rho_2^{- \delta} \left( \log \rho_2 \right)^r$. By our previous improvement on $B_{II}$, we know that $\hat s_{\delta,r} \in \rho_2^{2n - \beta} C_b^\infty ( B_{I})$ for arbitrarily small $\beta >0$, hence, $\tilde s_{\delta,j} \in \rho_2^{2n - \beta'} C_b^\infty (B_{I})$ for $ \beta' := \beta + \delta$. Therefore, choosing $0 < a < 2n -2 - \delta$, we get that 
$$
\tilde s_{\delta,j} \in \rho_2^{a +2} C_b^\infty (B_{I}).
$$
Hence, by Proposition \ref{aclap}, there exists a unique polyhomogeneous $w_{\delta,j} \in \rho_2^a C_b^\infty(B_{I})$ such that 
$$
\Delta_{I} w_{\delta,j} = \tilde s_{\delta,j}. 
$$
Extending the coefficients of the expansion of $w_{\delta,j}$ off $B_{I}$ and setting $\hat u'_{\delta,j} = w_{\delta,j} \hspace{0.03cm} \varepsilon^{\delta +2} (\log \varepsilon)^j$ we get 
\begin{align*}
\Delta \hat u'_{\delta,j} = \Delta \left(w_{\delta,j} \hspace{0.03cm} \varepsilon^{\delta +2} (\log \varepsilon)^j \right) & = \varepsilon^\delta (\log \varepsilon)^j \Delta_I w_{\delta,j} + o(\rho_1^{\delta}) \\ & =  \varepsilon^\delta(\log \varepsilon)^j \tilde s_{\delta,j} + o(\rho_1^\delta). 
\end{align*}
Therefore, for $\hat u'_\delta := \sum_{j=0}^k \hat u'_{\delta, j} = \sum_{j=0}^k w_{\delta,j} \hspace{0.03cm} \varepsilon^{\delta +2} (\log \varepsilon)^j$ we get that
$$
\Delta \hat u'_\delta = \hat s_\delta + o(\rho_1^\delta).
$$
On the other hand, by choosing $a < 2n-2-\delta$ close enough to $2n-2-\delta$, we get that, for all $0<\delta'<\delta$ 
$$
\hat u'_\delta\in \varepsilon^{\delta' +2} \rho_2^a C_b^\infty \implies \hat u'_\delta \in \rho_2^{a+ \delta' +2} \rho_1^2 C_b^\infty \implies \Delta \hat u'_\delta \in \rho_2^{a+ \delta' +2} C_b^\infty = \rho_2^{2n - b} C_b^\infty, 
$$

with $b>0$ as small as we want. In other words, we did not comprimise the improvement on $B_{II}$. Therefore $\hat u_2 + \hat u'_\delta$ is a formal solution for equation (\ref{linpot2}) at order $\delta$ both at $B_{I}$ and $B_{II}$. Similarly, replacing $\hat s - \Delta \hat u_2 $ by $s - \Delta (\hat u_2 + \hat u'_\delta)$, we can iterate the argument to eliminate all terms of order $0 < \delta < 2n -2$. In other words, putting $\hat u_1 := \sum_{0< \delta < 2n -2} \hat u'_\delta $ we can now suppose that
$$
\hat s - \Delta (\hat u_1 + \hat u_2 )  \in O(\varepsilon^{2n -2 - \beta }) \hspace{0.1cm}\text{for }\;\beta>0\;\text{arbitrarily small}.
$$
Writing $\hat s = \varepsilon^{2n -2 - \beta } \tilde s $ and plugging $\hat u = \varepsilon^{2n -2 - \beta } \tilde u$ in equation (\ref{linpot2}) then dividing by $\varepsilon^{2n -2 - \beta }$ on both sides we get 
$$
\Delta \tilde u = \tilde s.
$$
    \end{itemize}
Therefore, we can iterate using the same arguments as before to obtain a formal solution at all orders. In other words, now, we can suppose that 
$$
\hat s - \Delta \hat u_0\in \dot{C}^\infty(\mathcal{X}_b),
$$
and since at every step we are only adding terms of the form $\Delta f$, the integral still vanishes. Finally, for all $k >0$ and $\nu \in (2 -2n,0)$ using Lemma \ref{lmunf}, there exists a unique $\tilde u \in \left(H_{ \nu}^k\left(\mathcal{X}_b\right) \right)' $ such that 
$$
\Delta \tilde u = \hat s - \Delta \hat u_0,
$$
and 
$$
\|\tilde u\|_{H_{\nu}^k} \leq C\left\|\hat s - \Delta \hat u_0\right\|_{H_{\nu -2 }^{k-2}}.
$$
Since $\hat s - \Delta \hat u_0 \in \dot{C}^\infty(\mathcal{X}_b)$, we also get that $\tilde u \in \dot{C}^\infty(\mathcal{X}_b)$ and therefore equation (\ref{linpot2}) is solved for $\hat u := \hat u_0 + \tilde u$ which is polyhomogeneous. Finally, to get back to equation (\ref{linpot}), remember that 
$$
\hat s = \varepsilon^2 s.
$$
Therefore, putting $u : = \frac{\hat u}{\varepsilon^2}$, we get 
$$
\Delta u = s,
$$
and $u$ is polyhomogeneous. Moreover, since $\hat u \in o(\varepsilon^2)$, we get that $u$ vanishes at both $B_{I}$ and $B_{II}$ and this finishes the proof. 

     \end{proof}

\subsubsection{Statement of the theorem}  
 With the same notations and conditions as before, our main theorem is the following.
\begin{tm} \label{main} For $\tilde \varepsilon_0 >0$ sufficiently small, there exists a polyhomogeneous $u \in \mathcal{A}_{phg >0}\left(\mathcal{X}_b \cap \{ \varepsilon \leq \tilde \varepsilon_0 \}\right)$ such that
$$
\frac{\left( \omega_\varepsilon + i \partial \overline{\partial} u_\varepsilon \right)^n}{\omega_\varepsilon} = e^{ v_\varepsilon},
$$
therefore, $\omega_\varepsilon + i \partial \overline{\partial} u_\varepsilon:= \restr{\left( \omega + i \partial \overline{\partial} u \right)}{X_\varepsilon}$ is a Ricci-flat Kähler form on $X_\varepsilon$ for all $0 < \varepsilon \leq \tilde \varepsilon_0$.
\end{tm}
The proof of this theorem will be carried out in Sections \ref{secform} and \ref{secfix}. First, we will start by constructing our main examples for which this theorem applies. 
\subsection{Main examples} \label{main examples}
In all of the following, we let $(X_0, \omega_{CY})$ be a conical Calabi-Yau manifold with an isolated singular point $x \in X_0 $ modeled on a Calabi-Yau cone $(C, \omega_C)$ with vertex $o$ in the sense of Definition \ref{def : conical CY} and such that $\omega_{CY} \in [\omega] \in H^{1,1}(X_0^{reg}, \R)$ for a smoothly Kähler form $\omega$. As before, we suppose that $X_0$ has only one singular point only to keep notations light.
\subsubsection{Crepant resolution} \label{crepant resolution}

Suppose that $\hat \pi_{\hat C} : \hat{C} \rightarrow C$ is a crepant resolution of $C$, denote $E := \hat{\pi}^{-1} ( \{ o\})$. Using the local biholomorphism $x \in U_0 \cong \{p \in C; \hspace{0.1cm} r(p) < \frac{2}{\varepsilon} \} $, consider the holomorphic gluing $$\hat X_\varepsilon := \left((X_0 \setminus \{ x\}) \cup  (\hat{C} \setminus Y_\varepsilon) \right) \Big/ \sim,$$ where $Y_\varepsilon := \{p \in C; \hspace{0.1cm} r(p) \geq \frac{1}{\varepsilon} \}$ and $\sim$ is the equivalence relation identifying the images, under the biholomorphic identifications, of $\{p \in C; \hspace{0.1cm} r(p) < \frac{1}{\varepsilon} \} \setminus \{ o\}$ inside $U_{ 0} \setminus \{ x\}$ and $\hat C \setminus (E \cup Y_{ \varepsilon})$ respectively. Denote $\hat X$ the common underlying complex manifold. Consider the natural map $\hat \pi : \hat X \rightarrow X_0$ which is given by $\pi_{\hat C}$ near $E$ and by the identity elsewhere. Therefore, $\hat \pi^{*}K_{X_0} = K_{\hat X}$. Let $\omega_{AC}$ be an asymptotically conical Calabi-Yau metric on $\hat C$ as provided by Goto \cite{goto_calabi-yau_2012} and suppose that $\omega_{AC}$ satisfies Assumption \hyperlink{R}{R}.\\
Now, let $\mathcal{M}$ be the parametric space 
    $$
    \mathcal{M} := \hat X \times [0,1)_\varepsilon.
    $$
    The space $\mathcal{M}$ is locally diffeomorphic near $E \times \{ 0\}_\varepsilon$ to the local model 
    $$
    \hat{\mathcal{C}} := \hat C \times [0,1)_\varepsilon.
    $$

In the following, we describe how to construct the weighted blow-up in the sense of Melrose and the family of Kähler metrics on $\hat{X}$ to be able to apply Theorem \ref{main}.\\

\textbf{Step 1: Local resolution}\\
We will start by constructing the local model of the resolution using the cone $C$. For such a construction in the particular case of orbifold singularities, see \cite[Section 5]{najafpour2024constant}.\\

First, by theorem \ref{convar}, we know that the cone $C$ can be seen as a normal variety in $\C^N$ with the vertex identified with the origin $0 \in \C^n$ and the scaling action on $C$ generated by $r \partial_r$, where $r$ is the radial function associated to the Ricci-flat Kähler cone metric $\omega_C$,  extends to a diagonal $\R_{>0}$-action on $\C^N$ given by $(\lambda, z_1, z_2, \cdots, z_N) \to (\lambda^{w_1} z_1, \lambda^{w_2} z_2, \cdots, \lambda^{w_N} z_N)$ for $w_i >0$. Suppose $C$ is given by 
$$
C = \left\{z \in \C^N; \hspace{0.1cm} P_i(z)=0, \hspace{0.1cm} i =1,2, \cdots, m \right\}, 
$$
for some polynomials $P_i$, homogeneous with respect to the weighted diagonal $\R_{>0}$-action of some real positive degree $d_i:= \deg P_i$.\\

Using the identification $\C^N \cong \R^{2N}$, consider the weighted Melrose blow-up $[\C^N \times [0,1)_\varepsilon ; \{0\}]_{\tilde w}$ in the sense of definition \ref{weighted blow-up} where the weight is given by $\tilde w := (w_1, w_1, w_2, w_2, \cdots, w_N, w_N, 1 ) \in (\R_{>0})^{2N +1}$ and let $$\beta : [\C^N \times [0,1)_\varepsilon ; \{0\}]_{\tilde w} \rightarrow \C^N \times [0,1)_\varepsilon$$ be the blow-down map. Now, denote $\mathcal{C}_b$ the lift of $C \times [0,1)_\varepsilon$ under the blow-down map i.e.
$$
\mathcal{C}_b := \overline{\beta^{-1}\left((C  \times [0,1)_\varepsilon) \setminus \{ o\} \times \{0 \} \right)}. 
$$

Following remark \ref{projective coordinates}, by considering the projective coordinates given by $\tilde z_j := \varepsilon^{- w_j} z_j$, for $j =1,2, \cdots, N$, away from $\varepsilon=0$, we get that for $(z, \varepsilon) \in C \times (0,1)_\varepsilon$, we have
$$ P_i(\tilde z) = P_i(\varepsilon^{-1} \cdot \tilde z) = \varepsilon^{-d_i} P_i(\tilde z) =0.$$
This means that, 
$$\mathcal{C}_b \setminus \overline{\beta^{-1}((C \setminus \{ o\}) \times \{ 0\}_\varepsilon)} \cong \left\{[ \tilde z : \varepsilon : 1] \in [\C^N \times [0,1)_\varepsilon ; \{0\}]_{\tilde w} ; \hspace{0.1cm} P_i(\tilde z) =0, \hspace{0.1cm} i =1,2, \cdots, m \right\}.$$

Therefore, $\mathcal{C}_b$ has the structure of a \textbf{conifold with corners}, by which we mean that $\mathcal{C}_b$ still has cone singularities along $\left\{ [ 0 : \varepsilon :1] \in[\C^N \times [0,1)_\varepsilon ; \{0\}]_{\tilde w} \right\} \cap \mathcal{C}_b$, but away from these singularities $\mathcal{C}_b$ has the structure of a manifold with corners with two boundary hypersurfaces at $\varepsilon=0$: 
\begin{itemize}
    \item 
$B_I:= \mathcal{C}_b \cap \beta^{-1} ( \{ o\})$ is the new face coming from the blow-up of the origin and, by the projective coordinates considered above, it identifies on its interior with the cone $C$ itself under coordinates $\tilde z$;
\item $B_{II} := \overline{\beta^{-1}((C \setminus \{ o\}) \times \{ 0\}_\varepsilon)}$ is the lift of the original boundary $C \times \{ 0\}_\varepsilon$ and it corresponds to a compactification of the cone by its link using the radial function $r$ of the metric $\omega_C$.    
\end{itemize}
Now, consider the crepant resolution $\hat \pi_{\hat C} : \hat{C} \rightarrow C$. Naturally, $\hat \pi_{\hat C}$ induces a resolution on $\mathcal{C}_b$ given by
$$
\hat{\pi}_b : \hat{\mathcal{C}}_b \rightarrow \mathcal{C}_b,
$$
where $\hat{\mathcal{C}}_b$ resolves the singularities along $\left\{ [0 : \varepsilon :1] \in[\C^N \times [0,1)_\varepsilon ; \{0\}]_{\tilde w} \right\} \cap \mathcal{C}_b$. In other words, $\hat{\mathcal{C}}_b$ is a manifold with corners with fibers above $\varepsilon >0$ identified with $\hat{C}$ and two boundary hypersurfaces at $\varepsilon = 0 $:
\begin{itemize}
    \item 
$\hat{B}_I$ corresponds to a radial compactification of the crepant resolution $\hat{C}$, i.e. $B_I \cong \overline{\hat{C}}$,
\item $\hat{B}_{II}$ is the same as $B_{II}$. 
\end{itemize}
Moreover, if we set $\rho_1 := \sqrt{r(z)^2 + \varepsilon^2}$ and $\rho_2 := \frac{\varepsilon}{\sqrt{r(z)^2 + \varepsilon^2}}$, where $r$ is the radial function of $\omega_C$, then $\rho_1$ is a boundary defining function for $B_I$ and $\rho_2$ is a boundary defining function for $B_{II}$ such that $\rho_1 \rho_2 = \varepsilon$ and 
\begin{itemize}
    \item $\restr{\rho_1}{B_{II}} = r(z)$;
    \item $\restr{\rho_2}{B_I} = \frac{1}{ r(\tilde z) \sqrt{1 + \frac{1}{r(\tilde z)^2}}} = \frac{1}{r(\tilde z)} + o\left(\frac{1}{r(\tilde z)} \right)$.
\end{itemize}
Here, we have omitted the $\beta^{*}$ to keep notations light.\\

Notice also  that, using the projective coordinates from above and remark \ref{projective coordinates}, the complex structure is resolved on $\mathcal{C}_b$ in the sense that 
$$
J \in C^\infty \left( \hat{\mathcal{C}_b} ; \mathrm{End}\left(\hspace{0.01cm} ^{c,\varepsilon} T^{*} \hat{\mathcal{C}}_b \right)  \right) \hspace{0.7cm} \text{with} \hspace{0.7cm} J^2 = -Id.
$$

\textbf{Step 2: Global resolution}\\

Now, since $\mathcal{M}$ is locally diffeomorphic to $\hat C \times [0,1)_\varepsilon$ near $E \times \{ 0\}_\varepsilon$, we can glue the model resolution $\hat{\mathcal{C}}_b$ near $E \times \{0\}_\varepsilon$ to get a new space ${\mathcal{M}}_b$ and we still denote the blow-down map by
$$
\beta : \mathcal{M}_b \rightarrow \mathcal{M}.
$$
Therefore, $\mathcal{M}_b$ is a manifold with corners and together with the map $ p_b : = p_\varepsilon \circ \beta$, where $p_\varepsilon$ is the projection into the second factor, we get a surgery space $(\mathcal{M}_b, p_b)$ in the sense of Definition \ref{srg} such that 
\begin{itemize}
    \item The fibers over $\varepsilon >0$ are identified with the crepant resolution $\hat{X}$.
    \item The boundary hypersurface $B_I$ is identified with $\overline{\hat{C}}$, the compactification of $\hat C$ by the link of the cone $C$.
    \item The boundary hypersurface $B_{II}$ is identified with $\overline{X_0}$ the compactification of $X_0 \setminus \{ x\}$ by the link of the cone $C$ using the radial function $r$.
\end{itemize}
In addition, we can extend the previously chosen boundary defining functions $\rho_1$ and $\rho_2$ to the rest of $\mathcal{M}_b$ such that $\varepsilon = \rho_1 \rho_2 $. Also, since the gluing is done locally, we still get that
$$
J \in C^\infty \left( \mathcal{M}_b; \mathrm{End}\left(\hspace{0.01cm} ^{c,\varepsilon} T^{*} \mathcal{M}_b \right)  \right)\hspace{0.7cm} \text{with} \hspace{0.7cm} J^2 = -Id.
$$

\textbf{Step 3: Construction of the Kähler metric}\\

Now, we would like to construct a polyhomogeneous family of metrics on the fibers of $\mathcal{M}_b$ which satisfies the conditions of Theorem \ref{main}.\\

On the one hand, the conical Ricci-flat Kähler metric $\omega_{CY}$ on $X_0$ naturally lifts to $B_{II}$ and by our choice of boundary defining functions, $\omega_{CY}$ is given in a neighborhood of $\partial B_{II}$, by 
$$
\omega_{CY} = \frac{i}{2} \partial \overline{\partial} \left( \rho_1^2 + f_2 \right),
$$
where $f_2 \in \rho_1^{2+ \lambda_2} C_b^\infty(B_{II})$, for $\lambda_2>0$. Moreover, $f_2$ is polyhomogeneous by Theorem \ref{tmA}.\\

On the other hand, if $\omega_{AC}$ is an asymptotically conical Ricci-flat Kähler metric on $\hat C$, then by Goto's construction \cite[Theorem 5.1, Lemma 5.6, Lemma 5.7]{goto_calabi-yau_2012} (see also \cite[Section 4.2]{conlon2013asymptotically}), we have 
$$
\restr{\omega_{AC}}{\hat C \setminus E} = \mathrm{pr}^{*} \xi + \frac{i}{2} \partial \overline{\partial} \left( \rho_2^{-2} + f_1 \right),
$$
where $\mathrm{pr} : C \rightarrow L$ is the radial projection from the cone $C$ on its link $L$, $\xi$ is a primitive basic harmonic $(1,1)$-form on $L$ that represents the restriction of the class $[\omega_{AC}]$ to $L$, $f_1 \in \rho_2^{-2+ \lambda_1} C_b^\infty(B_{I})$, for $\lambda_1>0$ and, by Theorem \ref{acphg}, $f_1$ is polyhomogeneous.\\

\begin{lm} \label{lemma form}
    Using the identification $U_0 \cong \{ p \in C; \hspace{0.1cm} r(p) < \frac{2}{\lambda}\}$ for a given $\lambda >0$, denote by $\tilde U_0$ the region inside $U_0$ identified with $\{p \in C; \hspace{0.1cm}  \frac{ 1}{ 2\lambda} < r(p)\leq \frac{3}{ 4\lambda} \} \cong Y_{2 \lambda} \setminus Y_{\frac{4\lambda}{3}}$. If $\omega_{AC}$ satisfies Assumption \hyperlink{R}{R} for such $\lambda>0$, then there exists a closed $(1,1)$-form $\tilde \eta$ on $X_0$ such that $\restr{\eta}{\tilde U_0} = \mathrm{pr}^* \xi$.
\end{lm}
\begin{proof}
Assumption \hyperlink{R}{R} implies that there exists $\lambda>0$ and a $(1,1)$-form $ \eta$ on $\hat X_\lambda \cong \hat X$ such that $[ \restr{\omega_{AC}}{\hat C \setminus Y_\lambda} ] = [ \restr{\eta}{\hat C \setminus Y_\lambda} ] \in H^{1,1}(\hat C \setminus Y_\lambda, \R)$. Therefore the restriction of $\eta$ on $(\hat C \setminus Y_\lambda) \setminus E$ has the form $\restr{\eta}{(\hat C \setminus Y_\lambda) \setminus E} = \mathrm{pr}^* \xi + d \theta$, for a certain real $1$-form $\theta$. Therefore $d \theta$ is a real exact $(1,1)$-form on $(\hat C \setminus Y_\lambda) \setminus E \cong (C \setminus \{ o\}) \setminus Y_\lambda$. Since $C  \setminus Y_\lambda $ is a Stein variety with trivial canonical bundle and $\hat C \setminus Y_\lambda$ is a crepant resolution, we may apply the same arguments as in \cite[Appendix A]{conlon2013asymptotically} and in \cite[Lemma 5.5, Lemma 5.6]{goto_calabi-yau_2012} to deduce that $d \theta = i \partial \overline{\partial}v$ for some $v$ smooth on $(C \setminus \{ o\}) \setminus Y_\lambda $. Let $\phi$ be a cutoff function $\phi :(C \setminus \{ o\}) \setminus Y_\lambda \rightarrow [0,1]$ such that $\phi =0$ on $Y_{2 \lambda} \setminus Y_{\frac{4 \lambda}{3}}$ and $\phi =1$ on $(C \setminus \{ o\}) \setminus Y_{3\lambda}$ and on $Y_{\frac{5 \lambda}{4}} \setminus Y_{\lambda}$. Let $\tilde \eta$ be a $(1,1)$-form on $X_0$ given by $\mathrm{pr}^{*} \xi + i \partial \overline{\partial}(\phi v)$ on $U_0 \setminus Y_{\frac{5 \lambda}{4}}$ and by $\eta$ elsewhere. Therefore, $\tilde \eta$ is a well defined closed $(1,1)$-form on $X_0$ and $\restr{\tilde \eta}{\tilde U_0} = \mathrm{pr}^{*} \xi$.

 \end{proof}

We are now ready to perform our gluing construction. Our gluing construction differs slightly from the usual gluing constructions in the literature because we need to make sure that we glue in a polyhomogeneous way near the corners.

\begin{pr} \label{glue}
    There exists $\varepsilon_0>0$ and $\omega \in \mathcal{A}_{phg \geq 0} \left( \hspace{0.01cm} ^{c,\varepsilon} T^{*} \mathcal{M}_b \wedge \hspace{0.01cm} ^{c,\varepsilon} T^{*} \mathcal{M}_b \right)$ polyhomogeneous such that 
    \begin{itemize}
        \item $\omega_\varepsilon$ is a Kähler form on $\hat X$ for all $0< \varepsilon < \varepsilon_0 $. 
        \item $\restr{\frac{\omega}{\varepsilon^2}}{B_I} = \omega_{AC}$.
        \item $\restr{\omega}{B_{II}} = \omega_C$.
    \end{itemize}
\end{pr}
\begin{proof}
Let us consider $3$ open regions near the boundary in $\mathcal{M}_b$
\begin{itemize}
    \item $V_0$ a small neighborhood of the corner $\partial B_I= \partial B_{II}$ in $\mathcal{M}_b$ such that $V_0 \cap B_{II} = U_0$ and denote $U_1 := V_0 \cap B_{I}$. We also suppose that $V_0$ has a product structure $V_0 \cong (B_I \cap B_{II}) \times [0,1)_{\rho_1} \times  [0,1)_{\rho_2}$,
    \item $V_1$ a collar neighborhood of $\overline{(B_{I} \setminus U_1)}$ away from the corner,
    \item $V_2$ a collar neighborhood of $\overline{(B_{II} \setminus U_0)}$ away from the corner such that $V_2 \cap B_{II} = \tilde U_0$ as defined in Lemma \ref{lemma form} and $V_1 \cap V_2 = \emptyset$.
\end{itemize} 
Finally, we set $W:= V_0 \cup V_1 \cup V_2$. Up to changing the boundary defining functions away from the corner, we can suppose that $\rho_1$ is constant along the fibers on $V_1 \cap V_0$ and similarly for $\rho_2$ on $V_2 \cap V_0$. Notice that since $V_1 \cap V_2 = \emptyset$, we can still arrange for $\rho_1$ and $\rho_2$ to satisfy $\rho_1 \rho_2 = \varepsilon$.\\
 
 Now, since $ f_1$ and $f_2$ are polyhomogeneous on $B_{I}$ and $B_{II}$ respectively and using the product identification $V_0 \cong (B_I \cap B_{II}) \times [0,1)_{\rho_1} \times  [0,1)_{\rho_2}$,
 we can consider extensions of the coefficients in the expansions of $f_1$ and $f_2$ on $V_0$ such that the extensions are trivial on $V_1 \cap V_0$ and $V_2 \cap V_0$ i.e. independent of $\rho_1$ and $\rho_2$. Hence using such extensions and the boundary defining functions $\rho_1$ and $\rho_2$ allows to extend $ f_1$ and $f_2$ in a polyhomogeneous way on $V_0$. In particular, with such extension, we get that
 $$
 f_1\in \rho_2^{-2+ \lambda_1} C_b^\infty(V_0) \implies  \varepsilon^2 f_1 \in \varepsilon^2 \rho_2^{-2 + \lambda_1} C_b^\infty( V_0) \implies \varepsilon^2 f_1 \in \rho_1^{2} \rho_2^{\lambda_1} C_b^\infty( V_0). 
 $$   
This implies, in particular, that $\restr{\varepsilon^2 f_1}{B_{II}} =0$. Similarly, we have 
$$
f_2 \in \rho_1^{2+ \lambda_2} C_b^\infty(V_0) \implies \varepsilon^{-2} f_2 \in \rho_2^{-2} \rho_1^{\lambda_2} C_b^\infty(V_0),
$$
therefore, $\restr{\varepsilon^{-2} f_2}{B_{I}}=0$. \\

Now, consider two smooth functions $\psi_i : W \rightarrow [0,1]$, for $i \in \{1,2\}$, such that
\begin{itemize}
    \item $\restr{\psi_1}{V_2 \setminus V_0} =0$ and $\restr{\psi_1}{W \setminus V_2}=1$.
    \item $\restr{\psi_2}{V_1 \setminus V_0} =0$ and $\restr{\psi_2}{W \setminus V_1}=1$.
\end{itemize}
In particular, $\psi_1$ and $\psi_2$ are constant and equal to $1$ near the corner. Therefore, we can define $\omega_\varepsilon$ on $W$ as the following

$$
 \omega_\varepsilon := 
\begin{cases}
   \varepsilon^2 \tilde \eta + \omega_C & \text{on $V_2 \setminus V_0$}\\
   \varepsilon^2 \mathrm{pr}^* \xi + i \partial \overline{\partial} \left( \frac{\rho_1^2}{2} +  \psi_2 \cdot f_2 + \psi_1 \cdot \varepsilon^2 f_1 \right) & \text{on $V_0$}\\
   \varepsilon^2 \omega_{AC} & \text{on $V_1 \setminus V_0$},
\end{cases}
$$
where $\tilde \eta$ is from Lemma \ref{lemma form}. Therefore, by construction, $\omega_\varepsilon$ is a well defined closed $(1,1)$-form, it is polyhomogeneous on $W$ and using the fact that $\varepsilon = \rho_1 \rho_2$ and the restrictions of $\varepsilon^2 f_1$ and $\varepsilon^{-2}f_2$ from above, we get $\restr{\omega}{B_{II}} = \omega_C$ and $\restr{\frac{\omega}{\varepsilon^2}}{B_I}= \omega_{AC}$. In particular, $g_\varepsilon := \omega_\varepsilon (., J.) \in \mathcal{A}_{phg}(W, Sym^2(\hspace{0.01cm} ^{c,\varepsilon} T^{*} W))$. Since $g_\varepsilon$ is positive definite as a section of $Sym^2(\hspace{0.01cm} ^{c,\varepsilon} T^{*} W)$ on $B_{I}$ as well as on $B_{II}$ (namely, $\frac{g_\varepsilon}{\rho_1^2}$ is positive definite as a section of $Sym^2(\hspace{0.01cm} ^{b,\varepsilon} T^{*} W)$), we get that $g_\varepsilon$ remains positive definite for $\varepsilon$ sufficiently small.
\end{proof}
Therefore, Theorem \ref{main} applies and we get a polyhomogeneous family of Ricci-flat Kähler metrics on the crepant resolution proving Theorem \ref{tmB}.

\subsubsection{Polarized smoothing} \label{polarized smoothing}
Now, suppose, in addition, that $X_0$ is a projective variety with canonical singularity at $x$ and $L_0$ is an ample line bundle such that $\omega_{CY} \in 2 \pi c_1(L_0)$. Suppose $(X_0, L_0)$ is smoothable and let $\pi : (\mathcal{X}, \mathcal{L}) \rightarrow \mathbb{D}$ be a polarized smoothing of $(X_0, L_0)$ satisfying assumptions \hyperlink{S.1}{S.1} and \hyperlink{S.2}{S.2}.\\

\textbf{Step 1: Local resolution}\\

Assumption \hyperlink{S.1}{S.1} translates to the following: There exists an affine smoothing of the cone $C$ given by $p : W \rightarrow \C$ and an isomorphism of germs of complex-analytic spaces $H : (\mathbb{D}, 0) \rightarrow  (\C,0)$ such that $(\mathcal{X}, x)$ and $(W,0)$ are isomorphic as germs of deformations of $C$ under a map  $I : (\mathcal{X}, x) \rightarrow  H^*(W,0)$. Moreover, $W$ is given by 
$$
W := \left\{ (z,t) \in \C^N \times \C_t ; \hspace{0.1cm} P_i(z) + t Q_i(t, z) =0, \hspace{0.1cm} i \in \{ 1,2, \cdots, m\}  \right\}, 
$$
where the $P_i$ are the polynomials defining $C$ as an affine variety as in Theorem \ref{convar} and $Q_i(t, z)$ are polynomials which are homogeneous with respect to a diagonal $\R_{>0}$-action on $\C^N \times \C_t$ extending the scaling $\R_{>0}$-action on $C$ and given by 
$$\lambda \cdot (z, t) = (\lambda^{w_1} z_1, \lambda^{w_2} z_2, \cdots, \lambda^{w_N} z_N, \lambda^{\mu} t ),$$ 
such that $\mu >0$ and $\deg Q_i + \mu = \deg P_i$ with respect to this action. Here, we are also, implicitly, assuming that the polynomials $Q_i$ are chosen so that for $t \neq 0$, the fibers $V_t := p^{-1}(\{ t\})$ are smooth.\\

As in the introduction, we will only consider a ray in $W$, denoted by $W_0$ given by 
$$
W_0 = \left\{ (z,t) \in \C^N \times [0, \infty)_t ; \hspace{0.1cm} P_i(z) + t Q_i(t, z) =0, \hspace{0.1cm} i \in \{ 1,2, \cdots, m\} \right\}.
$$
For the weight $\tilde w := \left(w_1, w_1, w_2, w_2, \cdots, w_N, w_N, \mu \right) \in \left( \R_{>0} \right)^{2N+1}$ and using the identification $\C^N \cong \R^{2N}$, consider the weighted Melrose blow-up $\left[\C^N \times [0, \infty)_t ; \{ 0\} \right]_{\tilde w}$ together with the blow-down map 
$$\beta : [\C^N \times [0,\infty)_t ; \{0\}]_{\tilde w} \rightarrow \C^N \times [0,\infty)_t.$$
Denote $W_b$ the lift of $W_0$ under the blow-down map $\beta$ i.e. 
$$
W_b := \overline{\beta^{-1}\left(W_0 \setminus \{ 0\} \right)}. 
$$
Following Remark \ref{projective coordinates}, by considering the projective coordinates given by $s: =t^{\frac{1}{\mu}}$, $\tilde z_j := s^{- w_j} z_j$, for $j =1,2, \cdots, N$, away from $t=0$, we get that if $(z, t) \in W \setminus C$, then
\begin{align*}
P_i(z) + t Q_i(t, z) = P_i(s \cdot \tilde z) + s^\mu Q_i(s \cdot (1, \tilde z)) & = s^{d_i} \left( P_i(\tilde z) +  Q_i(1, \tilde z) \right) = 0 \\ & \implies P_i(\tilde z) +  Q_i(1, \tilde z) =0.
\end{align*} 
In other words, we get that 
$$W_b \setminus \overline{\beta^{-1}(C \setminus \{ o\}) } \cong \left\{[s : \tilde z : 1] \in [\C^N \times [0,\infty)_t ; \{0\}]_{\tilde w} ; \hspace{0.1cm} P_i(\tilde z) +  Q_i(1, \tilde z) =0, \hspace{0.1cm} i =1,2, \cdots, m \right\}.$$
Therefore, $W_b$ has the structure of a manifold with corners with two boundary hypersurfaces at $s=0$: 
\begin{itemize}
    \item 
$B_I:= W_b \cap \beta^{-1} ( \{ o\})$ is the new face coming from the blow-up of the origin and, by the projective coordinates considered above, it identifies on its interior with $V \cong p^{-1}(\{ 1\})$, the general fiber of $W$ under coordinates $\tilde z$;
\item $B_{II} := \overline{\beta^{-1}(C \setminus \{ o\}) }$ is the lift of the original boundary $C \times \{ 0\}_t$ and it corresponds to a compactification of the cone by its link.    
\end{itemize}
As before, we set $\rho_1 := \sqrt{r(z)^2 + s^2}$ and $\rho_2 := \frac{s}{\sqrt{r(z)^2 + s^2}}$, where $r$ is the radial function of $\omega_C$, then $\rho_1$ is a boundary defining function for $B_I$ and $\rho_2$ is a boundary defining function for $B_{II}$ such that $\rho_1 \rho_2 = s$ and 
\begin{itemize}
    \item $\restr{\rho_1}{B_{II}} = r(z)$;
    \item $\restr{\rho_2}{B_I} = \frac{1}{ r(\tilde z) \sqrt{1 + \frac{1}{r(\tilde z)^2}}} = \frac{1}{r(\tilde z)} + o\left(\frac{1}{r(\tilde z)} \right)$,
\end{itemize}
and we have omitted the $\beta^{*}$ to keep notations light. Similarly, the complex structure is resolved on $\mathcal{C}_b$ in the sense that 
$$
J \in C^\infty \left(W_b ; \mathrm{End}\left(\hspace{0.01cm} ^{c,\varepsilon} T^{*} W_b \right)  \right)\hspace{0.7cm} \text{with} \hspace{0.7cm} J^2 = -Id.
$$

\textbf{Step 2: Global resolution}\\

We can now use the isomorphism $I : (\mathcal{X},x) \rightarrow H^{*}(W,0)$ to glue the local model for the resolution given by $W_b$ near $x \in \mathcal{X}_0$, where $\mathcal{X}_0 := (H \circ I)^{-1}(W_0)$ is the path in the smoothing $\mathcal{X}$ which is the pullback of the ray $W_0$ under the local isomorphism of germs. We denote the resulting space by $\mathcal{X}_b$ with blow-down map
$$
\beta : \mathcal{X}_b \rightarrow \mathcal{X}_0.
$$
Therefore, $\mathcal{X}_b$ is a manifold with corners and together with the natural projection map $ p_b : \mathcal{X}_b \rightarrow [0, \infty)_s$, we get a surgery space $(\mathcal{X}_b, p_b)$ in the sense of Definition \ref{srg} such that 
\begin{itemize}
    \item The fibers over $s >0$ are identified with $X_s := (H \circ \pi )^{-1}(\{s^\mu \})$.
    \item The boundary hypersurface $B_I$ is identified with $\overline{V}$, the compactification of the general fiber $V \cong p^{-1}(\{ 1\})$ by the link of the cone $C$ using the function $\frac{1}{\tilde r}$, where $\tilde r = r(\tilde z)$.
    \item The boundary hypersurface $B_{II}$ is identified with $\overline{X_0}$ the compactification of $X_0 \setminus \{ x\}$ by the link of the cone $C$ using the radial function $r$.
\end{itemize}
Also, we can extend the previously chosen boundary defining functions $\rho_1$ and $\rho_2$ to the rest of $\mathcal{X}_b$ such that $s= \rho_1 \rho_2 $. In addition, since the gluing is done locally, we still get that
$$
J \in C^\infty \left( \mathcal{X}_b; \mathrm{End}\left(\hspace{0.01cm} ^{c,\varepsilon} T^{*} \mathcal{X}_b \right)  \right)\hspace{0.7cm} \text{with} \hspace{0.7cm} J^2 = -Id.
$$
\textbf{Step 3: Construction of the Kähler metrics}\\

Now, we would like to construct a polyhomogeneous family of Kähler metrics on $\mathcal{X}_b$ which allows us to apply Theorem \ref{main}.\\

On the one hand, the conical Ricci-flat Kähler metric $\omega_{CY}$ on $X_0$ naturally lifts to $B_{II}$ and by our choice of boundary defining functions, $\omega_{CY}$ is given in a neighborhood $U_2$ of $\partial B_{II}$ by 
$$
\omega_{CY} = \frac{i}{2} \partial \overline{\partial} \left( \rho_1^2 + f_2 \right),
$$
where $f_2 \in \rho_1^{2+ \lambda_2} C_b^\infty(B_{II})$, for $\lambda_2>0$. Moreover, $f_2$ is polyhomogeneous by Theorem \ref{tmA}.\\

On the other hand, by Assumption \hyperlink{S.2}{S.2}, the general fiber $V \cong p^{-1}(\{ 1\}) \cong \overset{\circ}{B_{I}}$ admits an asymptotically conical Ricci-flat Kähler metric $\omega_{AC}$, which is $i \partial \overline{\partial}$-exact outside a compact. Using our choice of boundary defining functions, in a neighborhood $U_1$ of $\partial B_I$, we have 
$$
\omega_{AC} = \frac{i}{2} \partial \overline{\partial} \left( C\rho_2^{-2} + f_1 \right),
$$
where $f_1 \in \rho_2^{-2+ \lambda_1} C_b^\infty(B_{I})$, for $\lambda_1>0$ and, by Theorem \ref{acphg}, $f_1$ is polyhomogeneous. Here we use \cite[Lemma 2.15]{conlon2013asymptotically} and the construction in \cite[Proof of Theorem 2.4]{conlon2013asymptotically} to get such form. Up to composing by a scaling, we can suppose that $C=1$.\\

The gluing construction in this case is less evident compared to the case of crepant resolutions because we have a degeneration of the complex structure making it not possible to directly extend $\omega_C$ above $s =0$. We get around that issue by using the polarization and considering a family of Hermitian metrics on the line bundles $L_{s} := \restr{\mathcal{L}}{X_s}$ in a similar way to the construction by Spotti \cite{Spotti}.  
\begin{pr}
    With the same notations as before, there exists $\omega \in \mathcal{A}_{phg \geq 0} \left( \hspace{0.01cm} ^{c,\varepsilon} T^{*} \mathcal{X}_b \wedge \hspace{0.01cm} ^{c,\varepsilon} T^{*} \mathcal{X}_b \right)$ polyhomogeneous such that 
    \begin{itemize}
 \item $\omega_s$ is a Kähler form on $X_s$ for all $s >0$ sufficiently small. 
        \item $\restr{\frac{\omega}{s^2}}{B_I} = \omega_{AC}$.
        \item $\restr{\omega}{B_{II}} = \omega_{CY}$.
    \end{itemize}
\end{pr}
\begin{proof}
   As before, we consider $3$ regions near the boundary in $\mathcal{X}_b$
\begin{itemize}
    \item $V_0$ a small neighborhood of the corner $\partial B_I= \partial B_{II}$ in $\mathcal{X}_b$ such that $\restr{V_0}{B_{I}} = U_1$ and $\restr{V_0}{B_{II}} = U_2$. We also suppose that $V_0$ has a product structure $V_0 \cong (B_I \cap B_{II}) \times [0,1)_{\rho_1} \times  [0,1)_{\rho_2}$,
    \item $V_1$ a collar neighborhood of $\overline{(B_{I} \setminus U_1)}$ away from the corner,
    \item $V_2$ a collar neighborhood of $\overline{(B_{II} \setminus U_2)}$ away from the corner such that $V_1 \cap V_2 = \emptyset$,
\end{itemize} 
and we set $W:= V_0 \cup V_1 \cup V_2$. Similar to before, we can extend $f_1$ and $f_2$ in a polyhomogeneous way along $V_0$.\\

Here, we cannot directly extend $\omega_{CY}$ on $V_2$ as a Kähler form as we did in the case of crepant resolutions because the complex structure varies along the parameter $s$. We get around that using the line bundle $\mathcal{L}$ from the polarization.\\

By our initial choice, we know that $\omega_{CY} \in 2 \pi c_1(L_0)$. In other words, $\omega_{CY}$ is the curvature form of a Hermitian metric on the line bundle $L_0$. We denote such a metric by $h_0$. Since $L_0$ is the restriction of the line bundle $\mathcal{L}$ to $X_0 \cong \mathring{B_{II}}$, we can consider $h$ a Hermitian metric on $\mathcal{L}$ on $V_2$ which restricts to $h_0$ on $X_0 \cap V_2$. We denote $h_s$ its restriction on $L_s = \restr{\mathcal{L}}{X_s}$ for $s >0$ and we denote the curvature form of $h_s$ with respect to the Chern connection by $\tilde \omega_{s}$. Therefore, $\tilde \omega_{s}$ is a closed $(1,1)$-form on $X_s \cap V_2$ such that $\tilde \omega_{s} \in 2 \pi c_1(\restr{L_s}{X_s \cap V_2})$ and $\tilde \omega_0 = \restr{\omega_{CY}}{X_0 \cap V_2}$. Up to making $V_0$ smaller, consider $\rho_s$, the potential of $\tilde \omega_s$ on $V_2 \cap V_0 \cap X_s$. \\

Now, we consider cut-off functions away from the corner defined as follows. Let $\psi_i : W \rightarrow [0,1]$ be smooth functions, for $i \in \{1,2\}$, such that
\begin{itemize}
    \item $\restr{\psi_1}{V_1 \setminus V_0} =0$ and $\restr{\psi_1}{W \setminus V_1}=1$.
    \item $\restr{\psi_2}{V_2 \setminus V_0} =0$ and $\restr{\psi_2}{W \setminus V_2}=1$.
\end{itemize}
In particular, $\psi_1$ and $\psi_2$ are constant and equal to $1$ near the corner. Therefore, we can define $\omega_s$ on $W$ as the following
$$
 \omega_s := 
\begin{cases}
   \tilde \omega_s & \text{on $V_2 \setminus V_0$}\\
   i \partial \overline{\partial} \left( (1-\psi_2) \cdot \rho_s +  \psi_2 \cdot \left(\frac{\rho_1^2}{2} +  \psi_1 \cdot  f_2 + s^2 f_1 \right) \right) & \text{on $V_0$}\\
   s^2 \omega_{AC} & \text{on $V_1 \setminus V_0$}.
\end{cases}
$$
Therefore, by construction, $\omega_s$ is a well defined closed $(1,1)$-form and it is polyhomogeneous on $W$ with restrictions $\restr{\omega}{B_{II}} = \omega_{AC}$ and $\restr{\frac{\omega}{s^2}}{B_I}= \omega_{AC}.$ In particular, $g_s := \omega_s (., J.) \in \mathcal{A}_{phg}(W, Sym^2(\hspace{0.01cm} ^{c,\varepsilon} T^{*} W))$. Since $g_s$ is positive definite as a section of $Sym^2(\hspace{0.01cm} ^{c,\varepsilon} T^{*} W)$ on $B_{I}$ as well as on $B_{II}$ (namely, $\frac{g_s}{\rho_1^2}$ is positive definite as a section of $Sym^2(\hspace{0.01cm} ^{b,\varepsilon} T^{*} W)$), we get that $g_s$ remains positive definite for $s$ sufficiently small.

\end{proof}
Hence, Theorem \ref{main} applies and we get a polyhomogeneous family of Ricci-flat Kähler metrics on the path $\mathcal{X}_0$ for small $s>0$ proving Theorem \ref{tmC}.
\begin{cor}
    There exists a polyhomogeneous family of non-vanishing holomorphic $(n,0)$-form $\Omega_s$ on $X_s$, for $s>0$ sufficiently small, such that the polyhomogeneous family of Ricci-flat Kähler metrics $\omega_{CY,s}$ satisfies
    $$
    \omega_{CY,s} = i^{n^2} \Omega_s \wedge \overline{\Omega_s}.
    $$
\end{cor}
\begin{proof}
    Let $\tilde \Omega$ be any nowhere vanishing relative holomorphic $(n,0)$-form on the smoothing $\mathcal{X}$. Therefore, its restriction to the path $\mathcal{X}_0$ is also holomorphic and therefore, naturally extends to a polyhomogeneous section of $\left(\bigwedge^{n} \hspace{0.07cm} (^{c, \varepsilon} \hspace{0.01cm} T^{*} \mathcal{X}_b) \right)^{\C}$. Moreover, since $\omega_{CY,s}$ is a Ricci-flat Kähler metric, we get that 
    $$
    \omega_{CY,s}^n = c_s i^{n^2} \tilde \Omega_s \wedge \overline{\tilde \Omega_s}, \hspace{0.1cm} \hspace{0.1cm} c_s = \frac{\int_{X_s} \omega_s^n}{\int_{X_s} i^{n^2} \tilde \Omega_s \wedge \overline{\tilde \Omega_s}}. 
    $$
    By the Melrose pushforward theorem, we get that both $\int_{X_s} \omega_s^n$ and $\int_{X_s} i^{n^2} \tilde \Omega_s \wedge \overline{\tilde \Omega_s}$ are polyhomogeneous. Moreover, $ \int_{X_s} i^{n^2} \tilde \Omega_s \wedge \overline{\tilde \Omega_s}$ is bounded below and above by certain uniform positive numbers $M, M'>0$ near $s=0$ and similarly for $\int_{X_s} \omega_s^n$, therefore, using the Taylor expansion of $\frac{1}{1+ y}$ and $\sqrt{1 + y}$, we can conclude that $\sqrt{c}$ is polyhomogeneous. Therefore, $\Omega_s := \sqrt{c_s} \tilde \Omega_s$ satisfies the wanted properties.  
 
\end{proof}
\begin{rem}
    In the work of Melrose and Zhu \cite{melrose_resolution_2018, MR4001023}, the authors use the polyhomogeneity of constant scalar metrics to study the boundary behaviour of the Weil-Petersson metric for for Riemann moduli spaces. A similar result in our context would be trivial because the Weil-Petersson metric for deformations of Calabi-Yau manifolds does not depend on the choice of the polarization and could be expressed only in terms of a choice of holomorphic relative $(n,0)$-form. See Tian \cite[Theorem 2]{MR915841}.
\end{rem}

\section{Formal solution} \label{secform}
Now we get back to proving Theorem \ref{main}. We use the same notations of section \ref{set} but we drop the $\varepsilon$ index for simplicity. In this section, we want to construct a polyhomogeneous approximate solution to the complex Monge-Ampère equation 
\begin{align} \label{MA}
\mathcal{M}(u) =  \frac{\left( \omega + i \partial \bar \partial u \right)^n}{\omega^n} = e^{v}.
\end{align}
First, we have the following expansion 
  \begin{align*}
 \mathcal{M}(u) & = \frac{\left( \omega + i \partial \bar \partial u \right)^n}{\omega^n} \\   & =1 + \Delta u + \sum_{l=2}^{n} C_{l,n} \left( \frac{\omega^{n-l} \wedge (i \partial \bar \partial u )^l}{\omega^{n}} \right) \\ & = 1 + \Delta u + Q(i \partial \bar \partial u), 
 \end{align*}
where $Q(i \partial \overline{\partial} u)$ includes all the non-linear terms. Therefore, equation (\ref{MA}) is equivalent to 
\begin{align} \label{MA2}
\Delta u + Q(i \partial \overline{\partial} u)= e^{v} -1
\end{align}

 Now, remember that, in the conditions of Theorem \ref{main},  $\mathcal{X}_b$ has two boundary hypersurfaces: $B_{I}$ and $B_{II}$. As above, we denote $\rho_1$ and $\rho_2$ a choice of boundary defining functions for $B_{I}$ and $B_{II}$ respectively such that $\varepsilon = \rho_1 \rho_2$. We also denote the two model Kähler forms $\omega_1 := \restr{\frac{\omega}{\varepsilon^2}}{B_{I}}$ and $\omega_2 := \restr{\omega}{B_{II}}$ and the associated metrics by $g_1$ and $g_2$. By the conditions of Theorem \ref{main}, $g_1$ is an asymptotically conical polyhomogeneous metric and $g_2$ is a conical polyhomogeneous metric. We denote their respective Laplacian operators by: $\Delta_{I}$ and $\Delta_{II}$. We have already seen the mapping properties of conical and asymptotically conical Laplacians in section 2.3. We will use those to prove the following formal solution to the complex Monge-Ampère equation 
 \begin{pr}
     Using notations of Theorem \ref{main}, there exists $u_0 \in \mathcal{A}_{phg}(\mathcal{X}_b)$ such that $\omega + i \partial \overline{\partial}u_0 >0$ for small $\varepsilon$ and
     $$
\frac{\left(\omega + i \partial \bar{\partial} u_{0}\right)^n}{\omega^n} = e^{v}-g
$$
where $g \in \dot{C}^{\infty} (\mathcal{X}_b)$ and $\int_{X_\varepsilon} g_\varepsilon \omega_\varepsilon^n = 0$. In other words, $\tilde \omega_0:= \omega + i \partial \bar{\partial} u_{0}$ is a Kähler form which is almost Ricci-flat with an error term vanishing to infinite order at the boundary of $\mathcal{X}_b$.
 \end{pr}
\begin{proof}
    By conditions of Theorem \ref{main}, we know that $v$ is polyhomogeneous and vanishes on $B_{I}$ and $B_{II}$, therefore, the same is true for $e^v -1$ since $e^v -1 = \sum_{i=1}^{\infty} \frac{v^i}{i !}$. Moreover, we have 
    $$
\int_{\varepsilon = c} 
\bigl(e^{\,v}-1\bigr)\,\omega^n = 0.
$$
We follow the same strategy is in the proof of Proposition \ref{ricpot} with minor differences accounting for the non-linearity of equation (\ref{MA2}). Accordingly, We solve, first, on $B_{II}$ and then on $B_{I}$.
    \begin{itemize}
        \item 
        Suppose $0< \delta < 2n$ is the smallest power of $\rho_2$ appearing in the polyhomogeneous expansion of $e^v -1$ near $B_{II}$ and let $ v_\delta :=  \sum_{i=0}^k v_{\delta, i} \rho_2^{\delta}\hspace{0.03cm}(\log \rho_2)^i$ be the terms of the expansion at such order, i.e.
        $$
        e^v -1 = v_\delta + o(\rho_2^\delta) =   \sum_{i=0}^k v_{\delta, i} \rho_2^{\delta}\hspace{0.03cm}(\log \rho_2)^i + o(\rho_2^\delta).
        $$
Following the same arguments as in the proof of Proposition \ref{ricpot}, there exists a polyhomogeneous $u$ such that
$$ \Delta u = v_\delta + o(\rho_2^{\delta}).$$ 
In fact, $u$ is of the form $u=\sum_{j=0}^k w_{\delta,j} \hspace{0.02cm} \varepsilon^\delta (\log \varepsilon)^j$, where $w_{\delta,j} \in \rho_1^{a} C^\infty_b(\mathcal{X}_b)$ is polyhomogeneous and $a$ is a well chosen weight in $(0,2) \setminus \mathcal{P}$ or $(2-2n,0)$ as in Proposition \ref{conlap}. Therefore, for all $0<\delta'< \delta$, we have that 
$$
\partial\overline{\partial}\,u \;\in\; \varepsilon^{\,\delta'}\,\rho_{1}^{\,a-2}\,C_{b,\varepsilon}^{\infty} (\mathcal{X}_b)
\;=\;
\rho_{2}^{\,\delta'}\,\rho_{1}^{\,a-2+ \delta'}\,C_{b,\varepsilon}^{\infty}(\mathcal{X}_b).
$$
Choosing $ \delta' > \max \{ \frac{\delta}{2}, 2 -a \}$, we get $a-2+ \delta' > 0$ and $2\delta' > \delta$, therefore $Q(i \partial \overline{\partial} u) \in o ( \rho_2^\delta)$. Hence, $u$ is a formal solution to equation (\ref{MA2}) at order $\delta$ near $B_{II}$, i.e.
$$
\Delta u + Q(i \partial \overline{\partial}u) = e^v -1 + o(\rho_2^\delta). 
$$
For $\varepsilon>0$ sufficiently small, $\omega + i \partial \overline{\partial} u $ restricts to a Kähler form on the fibers, therefore, replacing $\omega$ by $\omega + i \partial \overline{\partial} u $ and rewriting equation (\ref{MA2}), we get a new $v$ such that $e^v - 1$ vanishes at order $\delta$ on $B_{II}$. Iterating this argument we can remove all terms in the expansion near $B_{II}$ of order $0< \delta < 2n$. Hence, we may now suppose that 
$$
e^v -1 \in O(\rho_2^{2n - \beta}) \hspace{0.1cm}\text{for some}\;\beta>0\;\text{arbitrarily small}.
$$ 

\item Now, similar to Proposition \ref{ricpot}, we solve on $B_{I}$ without compromising the improvement we established on $B_{II}$. Suppose $0< \delta< 2n -2 $ is the smallest power of $\rho_1$ in the expansion of $e^v -1$ on $B_{I}$ and suppose $v_\delta =\sum_{j=0}^k v_{\delta,j} \rho_1^{\delta} (\log \rho_1)^j$ is its expansion at that order, i.e.
$$
e^v -1 = v_\delta + o(\rho_1^\delta).
$$
Again, following the same arguments as in Proposition \ref{ricpot}, there exists a polyhomogeneous $u$ such that 
$$
\Delta u = v_\delta + o(\rho_1^{\delta}), 
$$
where $u$ is of the form $u=\sum_{j=0}^k w_{\delta,j} \hspace{0.02cm} \varepsilon^{\delta+2} (\log \varepsilon)^j$ with $w_{\delta,j} \in \rho_2^{a} C^\infty_b(\mathcal{X}_b)$ polyhomogeneous and $a$ is a chosen weight in $(0,2n-2-\delta)$. Hence, by choosing $a < 2n-2-\delta$ close enough to $2n-2-\delta$, we also get that, for all $0<\delta' < \delta$  
$$
u \in \varepsilon^{\delta' +2} \rho_2^a C_b^\infty \implies u \in \rho_2^{a+ \delta' +2} \rho_1^2 C_b^\infty \implies \Delta u \in \rho_2^{a+ \delta' +2} C_b^\infty = \rho_2^{2n - b} C_b^\infty, 
$$

with $b>0$ as small as we want.  Moreover, for $0<\delta' < \delta$ we have 
\begin{align*}
u \;\in\; \varepsilon^{\,\delta'+2} \rho_{2}^{\,a} C_{b}^{\infty} 
\hspace{0.1cm} & \Longrightarrow\hspace{0.1cm}
\partial\overline{\partial}\,u
\;\in\;
\varepsilon^{\,\delta'+2}\rho_{2}^{\,a}\rho_{1}^{\,-2} C_{b}^{\infty}\\
& \Longrightarrow\hspace{0.1cm}
\partial\overline{\partial}\,u
\;\in\;
\,\rho_{2}^{\,a +2 + \delta'} \rho_{1}^{\,\delta'} C_{b}^{\infty}
\end{align*}
Choosing $\delta' > \frac{\delta}{2}$, we ensure that $\partial\overline{\partial}\,u$ only introduces an error term of order $2n - \gamma$ near the face $B_{II}$ with $\gamma = 2n-2-\delta' - a > 0$ as small as we want and the non-linear terms $Q(i \partial \overline{\partial} u)$ only introduce terms of order $\rho_1^{2 \delta'}$ near the face $B_{I}$ with $2 \delta' > \delta$. In particular, $u$ is a formal solution to equation (\ref{MA2}) at order $\delta$ near both $B_{I}$ and $B_{II}$. As before, by replacing $\omega$ with $\omega + i \partial\overline{\partial}\,u $ and iterating the argument we ensure that 
$$
e^v -1 \in O(\varepsilon ^{2n - 2 - \beta}) \hspace{0.1cm} \text{for $\beta >0$ arbitrarily small.}
$$ 
Writing $e^v -1 = \varepsilon ^{2n - 2 - \beta} \tilde{v}$, and plugging $u = \varepsilon ^{2n - 2 - \beta} \tilde{u}$ in equation (\ref{MA2}), we get 
\begin{align}\label{MA3}
\Delta \,\widetilde{u} \;+\;\varepsilon^{\,2n-2-\beta}\, \tilde Q\bigl( i\partial\overline{\partial}\,\widetilde{u}\bigr)
\;=\;
\,\widetilde{v} 
\end{align}
   \end{itemize}
Repeating the same arguments as before, we can construct formal solutions to (\ref{MA3}) up to order $2n - 2 - \beta' $ for $\beta'>0$ arbitrarily small. By iteration, one obtains a formal solution at all orders. In other words, after summation, we proved that there exists $u_0 \in \mathcal{A}_{phg}(\mathcal{X}_b)$ such that $\omega + i \partial \overline{\partial}u >0$ for small $\varepsilon >0$ and 
$$
\frac{(\omega + i \partial \overline{\partial} u_0)^n}{\omega^n} = e^v - g, 
$$
where $g \in \dot{C}^\infty(\mathcal{X_b})$ and after integrating both sides, we get $\int_{X_\varepsilon} g_\varepsilon \omega_\varepsilon^n= 0$ for $\varepsilon>0$.

\end{proof}

Therefore, we were able to construct a polyhomogeneous formal solution to the complex Monge-Ampère equation. In the next section, we finish the proof of theorem \ref{main} using a perturbation that we get using a Banach fixed point argument.
\section{Banach fixed point argument} \label{secfix}
In the previous section, we improved the Kähler form in the following sense: there exists a polyhomogeneous $u_{0}$ such that 
$$
\frac{\left(\omega + i \partial \bar{\partial} u_{0}\right)^n}{\omega^n} = e^v-g,
$$
and $g \in O( \varepsilon^{\infty})$.\\ In other words, the Kähler form defined by $\tilde \omega_{0} := \omega + i \partial \bar{\partial} u_{0}$ is almost Ricci-flat with an error term vanishing to infinite order in $\varepsilon$.\\

In this section, we want to perturb $\tilde \omega_0$ for small $\varepsilon$ to find a Calabi-Yau metric. More precisely, we want to solve for $\phi \in O(\varepsilon^{\infty})$ the following equation 
\begin{align} \label{MA8}
\frac{\left(\tilde \omega_0 + i \partial \bar{\partial} \phi\right)^n}{\tilde \omega_0^n} = \frac{1}{1 - e^{-v} g} = \frac{e^{v}}{e^{v} - g}.
\end{align}
First, we have the following expansion 
$$
\frac{\left(\tilde \omega_0+i \partial \bar{\partial} \phi\right)^n}{\tilde \omega_0^n} = 1+\Delta_{\tilde \omega_0} \phi+\sum_{j=2}^n \frac{n!}{(n-j)!j!}\left(\frac{\tilde \omega_0^{n-j} \wedge(i \partial \bar{\partial} \phi)^j}{\tilde \omega_0^n}\right).
$$
We rewrite the equation in the following way
\begin{align*}
\Delta_{\tilde \omega_0} \phi & = \left( \frac{e^{v}}{e^{v} - g} - 1 \right) -  \sum_{j=2}^n \frac{n!}{(n-j)!j!}\left(\frac{\tilde \omega_0^{n-j} \wedge(i \partial \bar{\partial} \phi)^j}{\tilde \omega_0^n} \right) \\ & = \left( \frac{g}{e^{v} - g} \right) -  \sum_{j=2}^n \frac{n!}{(n-j)!j!}\left(\frac{\tilde \omega_0^{n-j} \wedge(i \partial \bar{\partial} \phi)^j}{\tilde \omega_0^n} \right).
\end{align*}

Now, we want to prove that, for all $N$ and $M$ sufficiently large, $\nu$ chosen as in Lemma \ref{lmunf}, and for $\varepsilon$ sufficiently small, there exists a unique solution $\phi \in \varepsilon^N H_{\nu}^M$ for (\ref{MA}). After plugging $\phi = \varepsilon^N \hat \phi$ in the equation above and dividing by $\varepsilon^N$ on both sides, the equation becomes 

\begin{align} \label{MAdiv}
\Delta_{\tilde \omega_0} \hat \phi = \frac{g}{\varepsilon^N (e^{v} - g)} - Q_{\varepsilon}(i \partial \bar{\partial} \hat \phi),
\end{align}
where 
$$
Q_{\varepsilon}(i \partial \bar{\partial} \hat \phi):=\sum_{j=2}^n \frac{n!}{(n-j)!j!} \varepsilon^{N(j-1)}\left(\frac{\tilde \omega_0^{n-j} \wedge(i \partial \bar{\partial} \hat \phi)^j}{\tilde \omega_0^n} \right).
$$
We mention a lemma concerning the closedness of weighted Sobolev spaces under multiplication \cite[corollary 6.8]{pacini_desingularizing_2013}. 
\begin{lm} \label{lmmult}
 Let $n=\text{dim($\mathcal{X}_b$)}$. Assume $k >\frac{n}{2}$. Then the corresponding weighted Sobolev spaces are closed under multiplication, in the following sense. For any $\nu_1$ and $\nu_{\mathbf{2}}$ there exists $C>0$ such that, for all $u \in H_{ \nu_1}^k$ and $v \in H_{ \nu_2}^k$,
$$
\|u v\|_{H_{\nu_1 + \nu_2}^k} \leq C\|u\|_{H_{ \nu_1}^k} \cdot\|v\|_{H_{ \nu_2}^k}
$$ 
\end{lm}

Now, following the same notations of Lemma \ref{lmunf}, we state our main theorem:
\begin{tm}
Let $\nu \in (2-n,0)$, $M-2 > \frac{n}{2}$ and $N > 2 - \nu$. For $\varepsilon>0$ sufficiently small, the operator: 
\begin{align*}
K_N : \left(H_{\nu}^M\right)' \rightarrow  \left(H_{\nu}^M \right)'
\end{align*}
that sends each $w \in \left(H_{ \nu}^M \right)'$ to the unique $f \in \left(H_{ \nu}^M \right)'$  such that : \begin{align}
 \Delta_{\tilde \omega_0} f =  \frac{g}{\varepsilon^N (e^{v} - g)} - Q_{\varepsilon}(i \partial \bar{\partial} w)\end{align}
is well defined, induces a contraction on $U_{\delta} := \left\{ w \in \left(H_{\nu}^M \right)', \hspace{0.1cm} ||w||_{H_{\nu}^M} \leq \delta \right\}$ for small $\delta >0$ and therefore has a unique fixed point on $U_{\delta}$. The fixed point gives a solution to the equation (\ref{MAdiv}).
\end{tm}
\begin{proof} We divide the proof to three parts: 
\begin{itemize}
    \item \textbf{Part 1: $K_N$ is well defined:}
    Let $w \in \left(H_{ \nu}^M \right)' $. First, using Lemma \ref{lmmult}, We get that, for all $n \geq j \geq 2$ 
    
    $$ \left\| \frac{\tilde \omega_0^{n-j} \wedge(i \partial \bar{\partial}  w)^j}{\tilde \omega_0^{n}} \right\|_{H_{ j \left(\boldsymbol{\nu} - 2 \right)}^{M -2}} \leq C \left\| 
 w \right\|_{H_{\nu}^{M}}^{j}.$$ 
 In particular, 
 $$
 \frac{\tilde \omega_0^{n-j} \wedge(i \partial \bar{\partial}  w)^j}{\tilde \omega_0^{n}} \in H_{ j \left(\nu - 2 \right)}^{M -2}.
 $$
    
    On the other hand, by supposition that $N > 2 - \nu$ and the fact that $\varepsilon = \rho_1 \rho_2$, we get that
$$
\varepsilon^{N(j-1)} \in {H_{ (j-1) \left( 2 - \nu \right) }^{M -2}}. 
$$
and therefore, using Lemma \ref{lmmult}, we get for $w \in U_{\delta}$
\begin{align*}
\left\| Q_{\varepsilon} (i \partial \bar \partial w) \right\|_{H_{  \nu - 2 }^{M -2}} & \leq \sum_{j=2}^{n} C_{j}
\left\| \varepsilon^{N(j-1)} \left( \frac{\tilde \omega_0^{n-j} \wedge(i \partial \bar{\partial}  w)^j}{\tilde \omega_0^{n}}\right)  \right\|_{H_{  \nu - 2 }^{M -2}} \\ & \leq \sum_{j=2}^{n} \tilde C_{j} \left\| \frac{\tilde \omega_0^{n-j} \wedge(i \partial \bar{\partial}  w)^j}{\tilde \omega_0^{n}}\right\|_{H_{ j \left(  \nu - 2 \right) }^{M -2}} \cdot \left\| \varepsilon^{N(j-1)} \right\|_{H_{ (j-1) \left( 2 - \nu \right) }^{M -2}} \\ & \leq C' \varepsilon^{\mu} \left\| w \right\|_{H_{\nu }^{M}}^2,
\end{align*}
for all $0<\mu < N + \nu -2$. In particular, 
    $$ Q_{\varepsilon} (i \partial \bar \partial w) \in H_{\nu - 2 }^{M -2}$$
    For the remaining term, since $g$ vanishes to infinite order in $\varepsilon$, we get that $\frac{g}{\varepsilon^N (e^{v} - g)} \in H_{  \nu - 2 }^{M -2}$. We also have that the integral of the right hand side paired with the volume form $\tilde \omega_0^n$ vanishes 
    $$ 
    \int (RHS) \hspace{0.1cm} \tilde \omega_0^n =  \int \frac{g}{\varepsilon^N (e^{v} - g)} \tilde \omega_0^n = \frac{1}{\varepsilon^N} \int \frac{g}{e^{v} - g} \tilde \omega_0^n = \int g \omega^n = 0
    $$
    
    The first equality is because all terms involving $i \partial \bar \partial w$ have vanishing integral and the last equality is by construction of $g$ in the previous section. This proves that the right hand side belongs to the image of $\Delta_{\tilde \omega_0}$ acting on $H_{ \nu}^M$ and therefore, by Lemma \ref{lmunf},  $K_{N}$ is well defined.

    \item \textbf{Part 2: $K_N (U_{\delta}) \subset U_{\delta}$:} Let $w \in U_{\delta}$. Using Lemma \ref{lmunf} and the inequalities established in the previous parts, we have 

    \begin{align*}
    \left\| K_{N} (w) \right\|_{H_{ \nu}^M} \leq C \left\| \frac{g}{\varepsilon^N (e^{v} - g)} - Q_{\varepsilon}(i \partial \bar{\partial} w)\right\|_{H_{ \nu - 2 }^{M -2}} & \leq C \left(  \left\| \frac{g}{\varepsilon^N (e^{v} - g)} \right\|_{H_{  \nu - 2 }^{M -2}} +  \left\| Q_{\varepsilon}(i \partial \bar{\partial} w) \right\|_{H_{  \nu - 2 }^{M -2}} \right) \\ & \leq C' \left( \varepsilon + \varepsilon^{\mu} \left\| w \right\|_{H_{\nu }^{M}} \right) \\ & \leq C' (\varepsilon + \varepsilon^{\mu} \delta).
   \end{align*} 
  Therefore, choosing $\varepsilon$ sufficiently small ensures that $\left\| K_{N} (w) \right\|_{H_{ \nu}^M} \leq \delta$. Thus, $K_N$ maps $U_{\delta}$ to itself.

    \item \textbf{Part 3: $K_N$ is a contraction on $U_{\delta}$:} Let $u,v \in U_{\delta}$. 
    \begin{align*}
    \left\| K_{N} (u) - K_{N}(v) \right\|_{H_{ \nu}^M} & \leq C \left\|  Q_{\varepsilon}(i \partial \bar{\partial} u) - Q_{\varepsilon}(i \partial \bar{\partial} v)\right\|_{H_{  \boldsymbol{\nu} - 2 }^{M -2}} \\ & \leq C \sum_{j=2}^{n} \left\| \varepsilon^{N(j-1)} \left( \frac{\tilde \omega_0^{n-j} \wedge\left( (i \partial \bar{\partial}  u)^j - (i \partial \bar{\partial}  v)^j \right)}{\tilde \omega_0^{n}}\right)\right\|_{H_{  \boldsymbol{\nu} - 2 }^{M -2}} \\ & \leq C' \varepsilon^{\mu} \left( \sum_{j=2}^{n} \sum_{k=0}^{j-1} \left\| \frac{\tilde \omega_0^{n-j} \wedge \left( i \partial \bar \partial (u -v) \right) \wedge\left( (i \partial \bar{\partial}  u)^{k} \wedge (i \partial \bar{\partial}  v)^{j - 1-k} \right)}{\tilde \omega_0^{n}} \right\|_{H_{(j-1)(\nu -2)}^{M - 2}} \right) \\ & \leq C'' \varepsilon^{\mu} \left( \sum_{j=2}^{n} \sum_{k=0}^{j-1} \left\| u -v\right\|_{H_{\nu}^{M}} \left\| u\right\|_{H_{\nu}^{M}}^{k} \left\| v\right\|_{H_{\nu}^{M}}^{j-1-k} \right) \\ & \leq C'' \varepsilon^{\mu} \left\| u -v\right\|_{H_{\nu}^{M}} \left( \sum_{j=2}^{n} \delta^{j-1} \right)
    \end{align*}
    Therefore, choosing $\varepsilon$ sufficiently small ensures that $K_{N}$ induces a contraction on $U_{\delta}$. Thus, by the Banach fixed point theorem, the operator $K_{N}$ has a unique fixed point in $U_{\delta}$, i.e, there exists a unique $\hat \phi \in U_{\delta}$ such that 
$$
\Delta_{\tilde \omega_0} \hat \phi =  \frac{g}{\varepsilon^N (e^{v} - g)} - Q_{\varepsilon}(i \partial \bar{\partial} \hat \phi)
$$
    
\end{itemize}
\end{proof}

This proves that, for all $\nu$, $N$ and $M$ satisfying the conditions of the previous theorem, there exists $\phi \in \varepsilon^N \left(H_{\nu}^M \right)' $ such that  
\begin{align*}
\frac{(\tilde \omega_0 + i \partial \bar \partial \phi)^n}{\tilde \omega_0^n} = \frac{e^v}{e^{v} - g} & \implies
\frac{(\omega + i \partial \bar \partial (u_0 + \phi))^n}{\omega^n} = e^v.\end{align*} 
By the regularity of solutions to the complex Monge-Ampère equation, we get that the restriction of $u_0+ \phi$ to each fiber $X_\varepsilon$ is smooth. Hence, by uniqueness of the smooth solution with vanishing integral to the complex Monge-Ampère equation on each fiber (See \cite[Theorem 3.14, Exercise 3.16]{szekelyhidi_extremal_2014} ), if $\phi' \in \varepsilon^{N'} \left(H_{\nu}^{M'} \right)'$ is another solution to equation (\ref{MA8}), with $N' \geq N$ and $M' \geq M$, then $\restr{u_0 + \phi'}{X_\varepsilon} = \restr{u_0 + \phi}{X_\varepsilon}$ and thus $\phi' = \phi$. In other words, we get that $\phi \in \dot{C}^{\infty} (\mathcal{X}_b)$. Therefore, putting $u := u_0 + \phi$, the Kähler form $\tilde{\omega} :=\omega + i \partial \bar \partial u$ is Ricci-flat, finishing the proof of Theorem \ref{main}.

\section*{Declarations}

\subsection*{Funding}
No funding was received for conducting this study.

\subsection*{Competing interests}
The author declares no competing interests.

\subsection*{Data availability}
No datasets were generated or analysed during the current study.

\medskip
\textsc{Département de mathématiques, Université du Québec à Montréal}\\
\textit{Email address :} \texttt{benabida.abdou_oussama@uqam.ca}

\end{document}